\pdfoutput=1
\RequirePackage{ifpdf}
\ifpdf 
\documentclass[pdftex]{sigma}
\else
\documentclass{sigma}
\fi

\usepackage{bbm}
\usepackage{stmaryrd}
\SetSymbolFont{stmry}{bold}{U}{stmry}{m}{n}
\usepackage{euscript}
\usepackage[arrow,curve,matrix,arc,2cell]{xy}
\UseAllTwocells
\usepackage{slashed}
\DeclareFontFamily{U}{rsfs}{}
\DeclareFontShape{U}{rsfs}{n}{it}{<->
rsfs10}{} \DeclareSymbolFont{mscr}{U}{rsfs}{n}{it}
\DeclareSymbolFontAlphabet{\scr}{mscr}
\def\mathscr{\scr}

\def\e#1\e{\begin{equation}#1\end{equation}}
\def\ea#1\ea{\begin{align}#1\end{align}}
\def\eq#1{{\rm(\ref{#1})}}
\theoremstyle{plain}
\newtheorem{thm}{Theorem}[section]

\newtheorem{prop}[thm]{Proposition}
\newtheorem{cor}[thm]{Corollary}
\newtheorem{quest}[thm]{Question}

\theoremstyle{definition}
\newtheorem{dfn}[thm]{Definition}
\newtheorem{ex}[thm]{Example}
\newtheorem{rem}[thm]{Remark}
\newtheorem{conj}[thm]{Conjecture}

\newtheorem{ass}[thm]{Assumption}

\numberwithin{figure}{section}
\numberwithin{equation}{section}
\def\dim{\mathop{\rm dim}\nolimits}

\def\vdim{\mathop{\rm vdim}\nolimits}

\def\Im{\mathop{\rm Im}\nolimits}
\def\det{\mathop{\rm det}\nolimits}
\def\Ker{\mathop{\rm Ker}}

\def\Spec{\mathop{\rm Spec}}
\def\Perf{\mathop{\rm Perf}}

\def\Iso{\mathop{\rm Iso}\nolimits}
\def\Aut{\mathop{\rm Aut}}

\def\Ho{\mathop{\rm Ho}}
\def\Hol{\mathop{\rm Hol}}
\def\PGL{\mathop{\rm PGL}\nolimits}
\def\GL{\mathop{\rm GL}\nolimits}

\def\SO{\mathop{\rm SO}}
\def\SU{\mathop{\rm SU}}
\def\Sp{\mathop{\rm Sp}}
\def\Re{\mathop{\rm Re}}

\def\Spin{\mathop{\rm Spin}}

\def\SF{\mathop{\rm SF}}
\def\SFa{\mathop{\rm SF_{al}}\nolimits}
\def\SFai{\mathop{\rm SF_{al}^{ind}}\nolimits}
\def\CF{\mathop{\rm CF}\nolimits}
\def\CFi{\mathop{\rm CF^{ind}}\nolimits}

\def\U{{\mathbin{\rm U}}}

\def\inv{{\rm inv}}

\def\ad{{\rm ad}}

\def\ch{\mathop{\rm ch}\nolimits}

\def\bdim{{\mathbin{\bf dim}\kern.1em}}

\def\Pic{\mathop{\rm Pic}}

\def\Stab{\mathop{\rm Stab}\nolimits}

\def\supp{\mathop{\rm supp}}

\def\rank{\mathop{\rm rank}\nolimits}
\def\Hom{\mathop{\rm Hom}\nolimits}

\def\Ext{\mathop{\rm Ext}\nolimits}
\def\cExt{\mathop{{\mathcal E}{\rm xt}}\nolimits}
\def\fExact{\mathop{\mathfrak{Exact}}}
\def\id{{\mathop{\rm id}\nolimits}}
\def\Id{{\mathop{\rm Id}\nolimits}}

\def\Map{{\mathop{\rm Map}\nolimits}}

\def\Art{\mathop{\bf Art}\nolimits}
\def\HSta{\mathop{\bf HSta}\nolimits}

\def\VertAlg{\mathop{\bf VertAlg}\nolimits}
\def\LieAlg{\mathop{\bf LieAlg}\nolimits}
\def\gr{{\bf gr}}
\def\TopSta{{\mathop{\bf TopSta}\nolimits}}

\def\pl{{\rm pl}}

\def\ran{{\rm an}}

\def\rst{{\rm st}}
\def\ss{{\rm ss}}

\def\top{{\rm top}}

\def\num{{\rm num}}
\def\irr{{\rm irr}}

\def\virt{{\rm virt}}
\def\fund{{\rm fund}}

\def\coh{\operatorname{coh}}
\def\vect{{\rm vect}}

\def\cs{{\rm cs}}

\def\Top{{\mathop{\bf Top}\nolimits}}

\def\modCQ{\mathop{\text{\rm mod-}\C Q}}
\def\modCQI{\text{\rm mod-$\C Q/I$}}
\def\modCtQ{\mathop{\text{\rm mod-}\C\ti Q}}

\def\bs{\boldsymbol}
\def\ge{\geqslant}
\def\le{\leqslant\nobreak}
\def\boo{{\mathbin{\mathbbm 1}}}
\def\O{{\mathcal O}}
\def\bA{{\mathbin{\mathbb A}}}
\def\bG{{\mathbin{\mathbb G}}}

\def\bL{{\mathbin{\mathbb L}}}

\def\K{{\mathbin{\mathbb K}}}
\def\R{{\mathbin{\mathbb R}}}
\def\Z{{\mathbin{\mathbb Z}}}
\def\bP{{\mathbin{\mathbb P}}}
\def\Q{{\mathbin{\mathbb Q}}}
\def\N{{\mathbin{\mathbb N}}}
\def\C{{\mathbin{\mathbb C}}}
\def\CP{{\mathbin{\mathbb{CP}}}}

\def\fC{{\mathbin{\mathfrak C}\kern.05em}}

\def\A{{\mathbin{\cal A}}}

\def\G{{{\cal G}}}
\def\M{{\mathbin{\cal M}}}
\def\B{{\mathbin{\cal B}}}

\def\cE{{\mathbin{\cal E}}}
\def\cF{{\mathbin{\cal F}}}
\def\cG{{\mathbin{\cal G}}}
\def\cH{{\mathbin{\cal H}}}

\def\cN{{\mathbin{\cal N}\kern .04em}}
\def\cP{{\mathbin{\cal P}}}

\def\cS{{\mathbin{\cal S}}}
\def\T{{{\cal T}\kern .04em}}

\def\oM{{\mathbin{\smash{\,\,\overline{\!\!\mathcal M\!}\,}}}}
\def\tiM{{\mathbin{\smash{\kern.7em\widetilde{\kern-.7em\mathcal M\kern-.4em}\kern.4em}}}}
\def\cV{{\cal V}}

\def\g{{\mathfrak g}}

\def\so{{\mathfrak{so}}}

\def\bW{{\bs W}\kern -0.1em}

\def\bY{{\bs Y}\kern -0.1em}

\def\al{\alpha}
\def\be{\beta}
\def\ga{\gamma}
\def\de{\delta}
\def\io{\iota}
\def\ep{\epsilon}
\def\la{\lambda}

\def\th{\theta}
\def\ze{\zeta}

\def\vp{\varphi}
\def\si{\sigma}
\def\om{\omega}
\def\De{\Delta}

\def\La{\Lambda}
\def\Om{\Omega}
\def\Ga{\Gamma}
\def\Si{\Sigma}
\def\Th{\Theta}
\def\Up{\Upsilon}

\def\ts{\textstyle}
\def\st{\scriptstyle}
\def\sst{\scriptscriptstyle}

\def\sm{\setminus}

\def\bu{\bullet}

\def\op{\oplus}
\def\ot{\otimes}
\def\bt{\boxtimes}

\def\bigop{\bigoplus}
\def\bigot{\bigotimes}
\def\iy{\infty}
\def\es{\varnothing}
\def\ra{\rightarrow}

\def\ab{\allowbreak}
\def\longra{\longrightarrow}
\def\hookra{\hookrightarrow}

\def\lb{\llbracket}
\def\rb{\rrbracket}
\def\ha{{\ts\frac{1}{2}}}
\def\t{\times}
\def\ci{\circ}
\def\ti{\tilde}
\def\d{{\rm d}}
\def\md#1{\vert #1 \vert}

\def\an#1{\langle #1 \rangle}

\begin{document}
\allowdisplaybreaks

\newcommand{\arXivNumber}{2005.05637}

\renewcommand{\thefootnote}{}

\renewcommand{\PaperNumber}{068}

\FirstPageHeading

\ShortArticleName{Universal Structures in $\mathbb C$-Linear Enumerative Invariant Theories}

\ArticleName{Universal Structures in $\mathbb C$-Linear Enumerative\\ Invariant Theories\footnote{This paper is a~contribution to the Special Issue on Enumerative and Gauge-Theoretic Invariants in honor of Lothar G\"ottsche on the occasion of his 60th birthday. The~full collection is available at \href{https://www.emis.de/journals/SIGMA/Gottsche.html}{https://www.emis.de/journals/SIGMA/Gottsche.html}}}

\Author{Jacob GROSS~$^{\rm a}$, Dominic JOYCE~$^{\rm a}$ and Yuuji TANAKA~$^{\rm b}$}

\AuthorNameForHeading{J.~Gross, D.~Joyce and Y.~Tanaka}

\Address{$^{\rm a)}$~The Mathematical Institute, Radcliffe Observatory Quarter,\\
\hphantom{$^{\rm a)}$}~Woodstock Road, Oxford, OX2 6GG, UK}
\EmailD{\href{mailto:jacob.a.gross@gmail.com}{jacob.a.gross@gmail.com}, \href{mailto:joyce@maths.ox.ac.uk}{joyce@maths.ox.ac.uk}}
\URLaddressD{\url{https://people.maths.ox.ac.uk/~joyce/}}

\Address{$^{\rm b)}$~Department of Mathematics, Faculty of Science, Kyoto University,\\
\hphantom{$^{\rm b)}$}~Kitashirakawa Oiwake-cho, Sakyo-ku, Kyoto 606-8502, Japan}
\EmailD{\href{mailto:y-tanaka@math.kyoto-u.ac.jp}{y-tanaka@math.kyoto-u.ac.jp}}
\URLaddressD{\url{https://www.math.kyoto-u.ac.jp/~y-tanaka/}}

\ArticleDates{Received November 22, 2021, in final form September 06, 2022; Published online September 23, 2022}

\Abstract{An {\it enumerative invariant theory} in algebraic geometry, differential geometry, or representation theory, is the study of invariants which `count' $\tau$-(semi)stable objects $E$ with fixed topological invariants $\llbracket E\rrbracket=\alpha$ in some geometric problem, by means of a {\it virtual class} $[\M_\alpha^{{\rm ss}}(\tau)]_{{\rm virt}}$ in some homology theory for the moduli spaces $\M_\alpha^{{\rm st}}(\tau)\subseteq\M_\alpha^{{\rm ss}}(\tau)$ of $\tau$-(semi)stable objects. Examples include Mochizuki's invariants counting coherent sheaves on surfaces, Donaldson--Thomas type invariants counting coherent sheaves on Calabi--Yau 3- and 4-folds and Fano 3-folds, and Donaldson invariants of 4-manifolds. We make conjectures on new universal structures common to many enumerative invariant theories. Any such theory has two moduli spaces $\M$, $\M^{{\rm pl}}$, where the second author (see \url{https://people.maths.ox.ac.uk/~joyce/hall.pdf}) gives $H_*(\M)$ the structure of a {\it graded vertex algebra}, and $H_*\big(\M^{{\rm pl}}\big)$ a {\it graded Lie algebra}, closely related to~$H_*(\M)$. The virtual classes $[\M_\alpha^{{\rm ss}}(\tau)]_{{\rm virt}}$ take values in $H_*\big(\M^{{\rm pl}}\big)$. In most such theories, defining $[\M_\alpha^{{\rm ss}}(\tau)]_{{\rm virt}}$ when $\M_\alpha^{{\rm st}}(\tau)\ne\M_\alpha^{{\rm ss}}(\tau)$ (in gauge theory, when the moduli space contains reducibles) is a difficult problem. We conjecture that there is a natural way to define invariants $[\M_\alpha^{{\rm ss}}(\tau)]_\inv$ in homology over~$\Q$, with $[\M_\alpha^{{\rm ss}}(\tau)]_\inv=[\M_\alpha^{{\rm ss}}(\tau)]_{{\rm virt}}$ when $\M_\alpha^{{\rm st}}(\tau)=\M_\alpha^{{\rm ss}}(\tau)$, and that these invariants satisfy a universal wall-crossing formula under change of stability condition~$\tau$, written using the Lie bracket on $H_*\big(\M^{{\rm pl}}\big)$. We prove our conjectures for moduli spaces of representations of quivers without oriented cycles. Versions of our conjectures in algebraic geometry using Behrend--Fantechi virtual classes are proved in the sequel~[\href{https://arxiv.org/abs/2111.04694}{arXiv:2111.04694}].}


\Classification{14D20; 17B69; 16G20}

\begin{flushright}
\it In honour of Lothar G\"ottsche's 60$^{\it th}$ birthday
\end{flushright}

\newpage

\setcounter{tocdepth}{2}
{\small \tableofcontents}

\renewcommand{\thefootnote}{\arabic{footnote}}
\setcounter{footnote}{0}

\section{Introduction}\label{wc1}

For us, an {\it enumerative invariant theory} in algebraic or differential geometry or representation theory is the study of invariants $I_\al(\tau)$ which `count' $\tau$-semistable objects $E$ with fixed topological invariants $\lb E\rb=\al$ in some geometric problem, usually by means of a virtual class $[\M_\al^\ss(\tau)]_\virt$ in some homology theory for the moduli space $\M_\al^\ss(\tau)$ of $\tau$-semistable objects, with $I_\al(\tau)=\int_{[\M_\al^\ss(\tau)]_\virt}\mu_\al$ for some natural cohomology class $\mu_\al$. Often the invariants $I_\al(\tau)$ have a~deformation-invariance property.

We say the enumerative invariant theory is $\C$-{\it linear} if the objects $E$ to be counted live in a $\C$-linear additive category $\A$. (The algebro-geometric version of our theory should extend to $\K$-linear additive categories, for $\K$ an algebraically closed field.) Here are some examples of such $\C$-linear theories:
\begin{itemize}\itemsep=0pt
\item[(i)] Invariants counting Gieseker semistable coherent sheaves on complex projective surfaces, as in Mochizuki~\cite{Moch}.
\item[(ii)] Donaldson--Thomas invariants counting semistable coherent sheaves on Calabi--Yau or Fano 3-folds, as in Thomas \cite{Thom1} and Joyce--Song~\cite{JoSo}.
\item[(iii)] Donaldson--Thomas type invariants counting semistable coherent sheaves on Calabi--Yau 4-folds, as in Borisov--Joyce \cite{BoJo} and Oh--Thomas~\cite{OhTh}.
\item[(iv)] Donaldson invariants counting anti-self-dual $\U(n)$- or $\SU(n)$-connections on compact oriented 4-manifolds, as in Donaldson--Kronheimer \cite{Dona2,DoKr,Kron}.
\item[(v)] Invariants counting semistable representations of quivers, or quivers with relations, and similar representation-theoretic problems.
\end{itemize}
A non-example is Gromov--Witten invariants counting Deligne--Mumford stable $J$-holomorphic curves $j\colon \Si\ra X$, since these do not form a $\C$-linear category.

We discuss some new universal mathematical structures that we expect to underlie most of the $\C$-linear enumerative invariant theories above. We explain our conjectural picture in detail in Section~\ref{wc4}. Here is a brief partial sketch, for simplicity in the algebraic geometry case with $\A=\coh(X)$ for a smooth projective $\C$-scheme $X$, as in (i), (ii) above.
\begin{itemize}\itemsep=0pt
\item[(a)] There are two ways to form a moduli stack of objects in $\A$: the usual moduli stack $\M$, in which a $\C$-point $[E]\in\M$ corresponding to an object $E\in\A$ has isotropy group $\Iso_\M([E])=\Aut(E)$, and the `projective linear' moduli stack $\M^\pl$, in which $\Iso_{\M^\pl}([E])=\Aut(E)/(\bG_m\cdot\id_E)$. There is a morphism $\M\ra\M^\pl$ which is a $[*/\bG_m]$-fibration over nonzero objects.

The second author \cite{Joyc12} explains how to give $H_*(\M,\Q)$ the structure of a {\it graded vertex algebra}, and $H_*\big(\M^\pl,\Q\big)$ a {\it graded Lie algebra} (both with nonstandard gradings). Here by a well-known construction in vertex algebra theory $H_*(\M,\Q)/D(H_*(\M,\Q))$ is a Lie algebra, and $H_*\big(\M^\pl,\Q\big)\cong H_*(\M,\Q)/D(H_*(\M,\Q))$, giving the Lie algebra structure.

Note that $H_*(\M,\Q)$ and $H_*\big(\M^\pl,\Q\big)$, with their vertex algebra/Lie algebra structures, can often be written down quite explicitly (see~\cite{Gros}).
\item[(b)] Let $\tau$ be a suitable stability condition on $\A$, and write $\M^\rst_\al(\tau)\subseteq\M^\ss_\al(\tau)$ for the moduli schemes of $\tau$-stable and $\tau$-semistable objects in $\A$ with Chern class $\al$. Then the coarse moduli scheme $\M^\ss_\al(\tau)$ is proper. In the cases we are interested in, either $\M^\rst_\al(\tau)$ is smooth, or has a natural perfect obstruction theory in the sense of~\cite{BeFa}. Also $\M^\rst_\al(\tau)\subset\M^\pl$ is an open substack.

Hence if $\M^\rst_\al(\tau)=\M^\ss_\al(\tau)$ (that is, if there are no strictly $\tau$-semistable objects in class $\al$) we have a virtual class $[\M^\ss_\al(\tau)]_\virt$ in the homology $H_*(\M^\ss_\al(\tau),\Z)$ or $H_*(\M^\ss_\al(\tau),\Q)$, and we may regard $[\M^\ss_\al(\tau)]_\virt$ as lying in the Lie algebra $H_*(\M^\pl,\Q)$ from (a). We consider $[\M^\ss_\al(\tau)]_\virt$ for all $\al$ to be the family of invariants we want to study.
\item[(c)] If $\M^\rst_\al(\tau)\ne\M^\ss_\al(\tau)$, the question of defining virtual classes $[\M^\rst_\al(\tau)]_\virt$ or $[\M^\ss_\al(\tau)]_\virt$ is a well known, mostly unsolved problem. In algebraic geometry, the issue is that $\M^\rst_\al(\tau)$ is not proper, and $\M^\ss_\al(\tau)$ does not have an obstruction theory, so we cannot use~\cite{BeFa}. In differential geometry, the problem is reducible connections giving singularities in moduli~spaces.

We conjecture that there are natural invariants $[\M^\ss_\al(\tau)]_\inv\in H_*\big(\M^\pl,\Q\big)$ for all $\al$, with $[\M^\ss_\al(\tau)]_\inv=[\M^\ss_\al(\tau)]_\virt$ in (b) when $\M^\rst_\al(\tau)=\M^\ss_\al(\tau)$, which satisfy a package of properties including (d),~(e) below. Here we must work in homology over $\Q$, not $\Z$, if~$\M^\rst_\al(\tau)\ne\M^\ss_\al(\tau)$.
\item[(d)] Let $\tau$, $\ti\tau$ be two suitable stability conditions on $\A$, for instance, Gieseker stability on $\coh(X)$ with respect to two polarizations $L,\ti L\ra X$.

We conjecture that the invariants $[\M^\ss_\al(\tau)]_\inv$, $[\M^\ss_\al(\ti\tau)]_\inv$ are related by the universal wall-crossing formula taken from the second author~\cite{Joyc7}:
\begin{align}
[\M^\ss_\al(\ti\tau)]_\inv =  \sum_{\al_1+\cdots+\al_n=\al} &\ti U(\al_1,\dots,\al_n; \tau,\ti\tau) \cdot \bigl[\bigl[ \dots \bigl[[\M^\ss_{\al_1}(\tau)]_\inv,\nonumber\\
&	
 [\M^\ss_{\al_2}(\tau)]_\inv\bigr],\dots\bigr], [\M^\ss_{\al_n}(\tau)]_\inv\bigr],\label{wc1eq1}
\end{align}
where the $\ti U(\al_1,\dots,\al_n;\tau,\ti\tau) \in \Q$ are a combinatorial coefficient system with only finitely many non-zero terms (see  \cite[Definition~4.4]{Joyc7}, \cite[Section~3.3]{JoSo}), and \eq{wc1eq1} uses the Lie bracket $[\,,\,]$ on $H_*\big(\M^\pl,\Q\big)$ from (a).
\item[(e)] We do not currently have a direct definition of $[\M^\ss_\al(\tau)]_\inv$ in (c) when $\M^\rst_\al(\tau)\ne\M^\ss_\al(\tau)$. However, as for Donaldson--Thomas invariants in \cite[Section~5.4]{JoSo} and Mochizuki's use of `$L$-Bradlow pairs' in \cite[Section~7.3]{Moch}, there is an indirect way to define $[\M^\ss_\al(\tau)]_\inv$ using the wall-crossing formula \eq{wc1eq1} in an auxiliary category $\B$ of `pairs' $V\ot\O(-N)\ra E$ in~$\A$.
\end{itemize}

Our theory is inspired by two main sources. The second author has a general theory of motivic invariants in algebraic geometry \cite{Joyc4,Joyc2,Joyc5,Joyc6,Joyc3, Joyc7}, including a wall crossing formula \cite[Theorems~5.2 and~5.4]{Joyc7} of the form \eq{wc1eq1} in a Lie algebra of `stack functions', which was applied to Donaldson--Thomas theory in~\cite{JoSo}. This does not apply to the enumerative invariants above, which are not motivic. But his recent work on vertex algebra and Lie algebra structures on homology of moduli spaces~\cite{Joyc12} provides the tools we need to extend it to enumerative invariants.

We hope that in the future our theory will lead to a better understanding of deep properties of enumerative invariants. In particular, the connection between vertex algebras and invariants is relatively unexplored. For example, can we explain modular properties and other structural features of generating functions of invariants in terms of the vertex algebras appearing in~(a)?

So far we have described only conjectures. The main results of this paper, stated in Section~\ref{wc5} and proved in Section~\ref{wc6}, are to prove our conjectures in (c)--(e) above when $\A=\modCQ$ is the abelian category of $\C$-representations of a quiver $Q$ without oriented cycles. In a~sequel~\cite{Joyc13}, the second author will prove the conjectures in other cases in algebraic geometry using Behrend--Fantechi virtual classes~\cite{BeFa}, including $\A=\modCQI$ for $(Q,I)$ a quiver with relations, and $\A=\coh(X)$ for $X$ a smooth projective complex curve, surface, or Fano 3-fold.

We define our invariants in {\it ordinary homology} $H_*\big(\M^\pl,\Q\big)$. To form virtual classes in ordinary homology, one needs {\it proper} moduli schemes, as in (b) above. One might hope that there should be a version of our theory in {\it Borel--Moore homology} $H_*^{\rm BM}\big(\M^\pl,\Q\big)$, that would not require properness.

However, this does not work, as pushforwards are only defined in Borel--Moore homology for proper morphisms, but we need pushforwards along non-proper morphisms at several crucial points, in particular to define the vertex algebra on $H_*(\M,\Q)$ and Lie algebra on $H_*\big(\M^\pl,\Q\big)$ in~(a). Also if $\M^\rst_\al(\tau)=\M^\ss_\al(\tau)$ we need to push $[\M^\ss_\al(\tau)]_\virt$ forward along the not necessarily proper inclusion $\M^\rst_\al(\tau)\hookra\M^\pl_\al$ to regard $[\M^\ss_\al(\tau)]_\virt$ as an element of~$H_*\big(\M_\al^\pl\big)$.

\section{Background on vertex algebras and Lie algebras}\label{wc2}

In this section, we review some material from the second author \cite{Joyc12}. Throughout this section~$R$ is a commutative $\Q$-algebra, for instance $R=\Q,\R$ or $\C$.

\subsection{Vertex algebras and Lie algebras}\label{wc21}

For background on vertex algebras, we recommend Frenkel--Ben-Zvi~\cite{FrBZ}, Kac~\cite{Kac2}, Lepowsky--Li~\cite{LeLi}, and the second author~\cite{Joyc12}. As we work over a $\Q$-{\it algebra} $R$, there are a few simplifications to the general theory. Here is one of several equivalent definitions of vertex algebra.

\begin{dfn}\label{wc2def1}
Let $V_*=\bigop_{n\in\Z}V_n$ be a graded $R$-module. Write $V_*((z)):= V_*[[z]]\big[z^{-1}\big]$ for the $R$-module of Laurent series in a formal variable~$z$. The $R$-modules $V_*[[z]],\ab V_*((z))$ are made $\Z$-graded by declaring~$\deg z=-2$.

A {\it field} on $V_*$ is an $R$-module homomorphism $V_*\ra V_*((z))$. The set of all fields on $V_*$ is denoted $\cF(V_*)$ and is considered as a graded $R$-module by declaring $\cF(V_*)_n$ to be the set of degree $n$ fields $V\ra V((z))$ for~$n\in\Z$.

A {\it graded vertex algebra} $\big(V_*,\boo,{\rm e}^{zD},Y\big)$ over $R$ is a $\Z$-graded $R$-module $V_*$ with an identity element $\boo\in V_0$, a grading-preserving operator ${\rm e}^{zD}\colon V\ra V[[z]]$ with ${\rm e}^{zD}v=\sum_{n\ge 0}\frac{1}{n!}D^n(v)\,z^n$ for $D\colon V_*\ra V_{*+2}$ the {\it translation operator}, and a grading-preserving {\it state-field correspondence} $Y\colon V_*\ra\cF(V_*)_*$ written $Y(u,z)v=\sum_{n\in\Z}u_n(v)z^{-n-1}$, where $u_n$ maps $V_*\ra V_{*+a-2n-2}$ for $u\in V_a$, satisfying:
\begin{itemize}\itemsep=0pt
\item[(i)] $Y(\boo,z)v=v$ for all $v\in V$.
\item[(ii)] $Y(v,z)\boo={\rm e}^{zD}v$ for all $v\in V$.
\item[(iii)] For all $u\in V_a$ and $v\in V_b$, there exists $N\gg 0$ such for all $w\in V_*$
\end{itemize}
\begin{equation*}
(z_1-z_2)^N\bigl(Y(u,z_1)Y(v,z_2)w-(-1)^{ab}Y(v,z_2)Y(u,z_1)w\bigr)=0\quad \text{in $V_*\big[\big[z_1^{\pm 1},z_2^{\pm 1}\big]\big]$}.
\end{equation*}
Part (iii) is called the {\it weak commutativity property}.

Let $\big(V_*,\boo,{\rm e}^{zD},Y\big)$ and $\big(V'_*,\boo',{\rm e}^{zD'},Y'\big)$ be graded vertex algebras over~$R$. A~{\it morphism} $\phi\colon \big(V_*,\boo,{\rm e}^{zD},Y\big)\ra \big(V'_*,\boo',{\rm e}^{zD'},Y'\big)$ is an $R$-module morphism $\phi\colon V_*\ra V_*'$ which preserves all the structures. That is, $\phi$ maps $V_n\ra V_n'$, and $\phi(\boo)=\boo'$, and $\phi\ci D=D'\ci\phi$, and $\phi\ci Y=Y'\ci(\phi\ot\phi)$. Such morphisms make graded vertex algebras over $R$ into a category~$\VertAlg_R^\gr$.
\end{dfn}

Vertex algebras are very complicated objects, and the above brief definition probably communicates little real understanding of them~-- we refer readers to \cite{FrBZ,Joyc12,Kac2,LeLi} for more. In this paper, the main property of (graded) vertex algebras we use is that they have a functor to (graded) Lie algebras.

\begin{dfn}\label{wc2def2}
A {\it graded Lie algebra} over $R$ is a pair $(V_*,[\,,\,])$, where $V_*=\bigop_{a\in\Z}V_a$ is a graded $R$-module, and $[\,,\,]\colon V_*\t V_*\ra V_*$ is an $R$-bilinear map called the {\it Lie bracket}, which is graded (that is, $[\,,\,]$ maps $V_a\t V_b\ra V_{a+b}$ for all $a,b\in\Z$), such that for all $a,b,c\in\Z$ and $u\in V_a$, $v\in V_b$ and $w\in V_c$ we have
\begin{equation*}
[v,u]=(-1)^{ab+1}[u,v],\qquad
(-1)^{ca}[[u,v],w]+(-1)^{ab}[[v,w],u]+(-1)^{bc}[[w,u],v]=0.
\end{equation*}

Let $(V_*,[\,,\,])$, $(V'_*,[\,,\,])$ be graded Lie algebras over $R$. A {\it morphism} $\phi\colon (V_*,[\,,\,])\ra (V'_*,[\,,\,])$ is an $R$-module morphism $\phi\colon V_*\ra V_*'$ which preserves all the structures. That is, $\phi$ maps $V_n\ra V_n'$ and $\phi\bigl([u,v]\bigr)=\bigl[\phi(u),\phi(v)\bigr]$. Such morphisms make graded Lie algebras over $R$ into a~category~$\LieAlg_R^\gr$.
\end{dfn}

The next proposition is due to Borcherds \cite[Section~4]{Borc}.

\begin{prop}\label{wc2prop1}
Let $\big(V_*,\boo,{\rm e}^{zD},Y\big)$ be a graded vertex algebra over $R$. We may construct a~graded Lie algebra $\big(\check V_*,[\,,\,]\big)$ over $R$ as follows. Noting the shift in grading, define a $\Z$-graded $R$-module $\check V_*$ by
\begin{equation*}
\check V_n=V_{n+2}/D(V_n)\qquad\text{for $n\in\Z$},
\end{equation*}
so that $\check V_*=V_{*+2}/D(V_*)$. If $u\in V_{a+2}$ and $v\in V_{b+2}$, the Lie bracket on $\check V_*$ is
\begin{equation*}
\bigl[u+D(V_a),v+D(V_b)\bigr]=u_0(v)+D(V_{a+b})\in\check V_{a+b}.
\end{equation*}

A morphism $\phi\colon \big(V_*,\boo,{\rm e}^{zD},Y\big)\!\ra \!\big(V'_*,\boo',{\rm e}^{zD'},Y'\big)$ induces a morphism $\check\phi\colon (V_*,[\,,\,])\!\ra\! (V'_*,[\,,\,])$ by $\check\phi\bigl(u+D(V_*)\bigr)=\phi(u)+D'(V'_*)$. Mapping $\big(V_*,\boo,{\rm e}^{zD},Y\big)\ab\mapsto(\check V_*,[\,,\,])$ and $\phi\mapsto\check\phi$ defines a~functor~$\VertAlg_R^\gr\ra\LieAlg_R^\gr$.
\end{prop}

\subsection{Stacks, and their homology groups}\label{wc22}

In Section~\ref{wc23} we will explain that if $\A$ is a suitable $\C$-linear abelian category, or $\T$ is a suitable $\C$-linear triangulated category, and $\M$ is the moduli stack of objects in $\A$ or $\T$, and we choose a little extra data, then~\cite{Joyc12} makes the homology $H_*(\M)$ into a graded vertex algebra. First we give some brief background on stacks $\M$ and their homology groups~$H_*(\M)$.

Stacks are a class of spaces in algebraic geometry. In this paper, two types of algebro-geometric stacks are relevant: {\it Artin $\C$-stacks}, as in G\'omez~\cite{Gome}, Laumon and Moret-Bailly \cite{LaMo} and Olsson~\cite{Olss}, which form a 2-category $\Art_\C$, and {\it higher $\C$-stacks}, as in To\"en and Vezzosi \cite{Toen1,Toen2,ToVe1,ToVe2}, which form an $\iy$-category $\HSta_\C$ containing $\Art_\C\subset\HSta_\C$ as a full discrete 2-subcategory.

The general rule is that for any algebro-geometric $\C$-linear abelian or exact category $\A$ appearing in this paper, such as $\A=\coh(X)$ or $\A=\vect(X)$ for $X$ a smooth projective $\C$-scheme, the moduli stack $\M$ of objects in $\A$ is an Artin $\C$-stack, and for any algebro-geometric $\C$-linear triangulated category $\T$, such as $\T=D^b\coh(X)$, the moduli stack $\M$ is a higher $\C$-stack.

If $S$ is an Artin or higher $\C$-stack, we write $S(\C)$ for the set of 2-isomorphism classes $[x]$ of 1-morphisms $x\colon \Spec\C\ra S$. Elements of $S(\C)$ are called $\C$-{\it points}, or {\it geometric points}, of $S$. If $\phi\colon S\ra T$ is a 1-morphism then composition with $\phi$ induces a map of sets~$\phi_*\colon S(\C)\ra T(\C)$.

If $S$ is an Artin $\C$-stack, each $\C$-point $x\in S(\C)$ has an {\it isotropy group} $\Iso_S(x)$, an algebraic $\C$-group. We say that $S$ {\it has affine geometric stabilizers} if $\Iso_S(x)$ is an affine algebraic $\C$-group for all~$x\in S(\C)$.

An important class of Artin $\C$-stacks are {\it quotient stacks} $[S/G]$, where $S$ is a $\C$-scheme and $G$ is an algebraic $\C$-group acting on $S$. When $\M=[S/G]$, the $\C$-points are $G(\C)$-orbits $xG(\C)$ for $\C$-points $x\in S(\C)$, and the isotropy groups are~$\Iso_{[S/G]}(xG(\C))=\Stab_{G(\C)}(x)$.

As in Simpson \cite{Simp} and Blanc \cite[Section~3.1]{Blan}, any Artin $\C$-stack or higher $\C$-stack $\M$ has a {\it topological realization} $\M^\top$, which is a topological space (in fact, a CW-complex) natural up to homotopy equivalence. Topological realization gives a functor $(-)^\top\colon \Ho(\HSta_\C)\ra\Top^{\bf ho}$ from the homotopy category $\Ho(\HSta_\C)$ to the category $\Top^{\bf ho}$ of topological spaces with morphisms homotopy classes of continuous maps.

Let $\M$ be an Artin $\C$-stack, or higher $\C$-stack, and $R$ be a commutative $\Q$-algebra, such as $R=\Q,\R$ or $\C$. We define the {\it homology $H_*(\M)=H_*(\M,R)$ of $M$ with coefficients in} $R$ to be $H_*(\M)=H_*(\M^\top,R)$, the usual homology of the topological space $\M^\top$. Similarly we define the cohomology $H^*(\M)=H^*(\M,R)=H^*(\M^\top,R)$. These are sometimes called the {\it Betti (co)homology}, to distinguish them from other (co)homology theories of stacks. We usually omit the coefficient $\Q$-algebra $R$ from our notation $H_*(\M)$, $H^*(\M)$.

The following properties of $H_*(\M)$, $H^*(\M)$ will be important later:
\begin{itemize}\itemsep=0pt
\item[(a)] Let $S$ be a $\C$-scheme, and $S^\ran$ the underlying complex analytic space. Then $H_*(S)\cong H_*(S^\ran)$ and $H^*(S)\cong H^*(S^\ran)$.
\item[(b)] If $\M$ is a quotient stack $[S/G]$, we have a homotopy equivalence
\begin{equation*}
\M^\top\simeq (S^\ran\t EG^\ran)/G^\ran,
\end{equation*}
where $EG^\ran\ra BG^\ran$ is a classifying space for the complex analytic topological group $G^\ran=G(\C)$. If $S$ is contractible (e.g., if $S$ is a point $*$ or an affine space $\bA^n$) this implies that $\M^\top\simeq BG^\ran$.
\item[(c)] For a disjoint union $\M=\coprod_{i\in I}\M_i$ we have $H_*(\M)\cong\bigop_{i\in I}H_*(\M_i)$ and $H^*(\M)\cong\prod_{i\in I}H^*(\M_i)$.
\item[(d)] $H_*(-)$ is covariantly functorial, and $H^*(-)$ is contravariantly functorial, under morphisms of stacks, in the obvious way.
\item[(e)] Cap and cup products $\cap$, $\cup$ are defined and have the usual properties, e.g., if $f\colon S\ra T$ is a morphism of stacks and $\al \in H_*(S)$, $\be\in H^*(T)$ then
\begin{equation*}
H_*(f)(\al \cap H^*(f)(\be)) = H_*(f)(\al) \cap \be.
\end{equation*}
\item[(f)] The K\"unneth theorem gives isomorphisms $H_*(S\t T)\cong H_*(S)\ot_RH_*(T)$ and $H^*(S\t T)\cong H^*(S)\ot_RH^*(T)$, as $R$ is a $\Q$-algebra.
\item[(g)] Let $\cE^\bu\ra\M$ be a perfect complex (e.g., a vector bundle). Then $\cE^\bu$ corresponds to a~morphism $\phi_{\cE^\bu}\colon \M\ra\Perf_\C$, where $\Perf_\C$ is a higher stack which classifies perfect complexes, as in To\"en and Vezzosi \cite[Definition~1.3.7.5]{ToVe2}. The topological realization of $\Perf_\C$ is $B\U\t\Z$, where $B\U=\varinjlim_{n\ra\iy}B\U(n)$ is the stable unitary classifying space, so that $B\U\t\Z$ is the classifying space for topological complex K-theory $K^0(-)$, as in May \cite[Sec\-tions~23--24]{May}. Thus $\phi_{\cE^\bu}^\top\colon \M^\top\ra B\U\t\Z$ defines a K-theory class $[\cE^\bu]\in K^0(\M^\top)$. Hence we may define the {\it Chern classes} $c_i(\cE^\bu)=c_i([\cE^\bu])$ in $H^{2i}(\M)=H^{2i}(\M^\top)$. These have the usual properties of Chern classes, e.g., $c_k(\cE^\bu\op\cF^\bu)=\sum_{i+j=k}c_i(\cE^\bu)\cup c_j(\cF^\bu)$.
\end{itemize}
As $B\bG_m\simeq \CP^\iy$, using (b) above and $\CP^n\hookra\CP^\iy$ we have an isomorphism
\begin{gather}
H^*([*/\bG_m]) \cong R[[z]]  \quad \text{as $R$-algebras, with $\deg z=2$, so that}\nonumber\\
H^{2n}([*/\bG_m]) =\an{z^n}_R,  \quad \text{normalized so that}\quad  \int_{\CP^n}z^n=1.\label{wc2eq1}
\end{gather}

Our conjectures in Section~\ref{wc4} involve fairly general Artin and higher stacks. However, our main results in Sections~\ref{wc5}--\ref{wc6} involve only moduli stacks $\M$ of abelian categories of quiver representations $\modCQ$. These are of a very simple kind: we have $\M=\coprod_{\bs d\in\N^{Q_0}}\M_{\bs d}$, where $\M_{\bs d}$ is a~quotient stack $[V_{\bs d}/\Pi_{v\in Q_0}\GL(\bs d(v),\C)]$ for $V_{\bs d}$ a $\C$-vector space. Then we can compute $H_*(\M)$, $H^*(\M)$ using~(b),~(c) above. So Sections~\ref{wc5}--\ref{wc6} do not need a detailed knowledge of stacks.

There is also a theory of {\it topological stacks} due to Metzler~\cite{Metz} and Noohi~\cite{Nooh1,Nooh2}, which we can use to write the differential-geometric version of our conjectural picture in Section~\ref{wc4}. In this paper, by `topological stacks' we mean {\it hoparacompact topological stacks} in the sense of Noohi \cite[Section~8.3]{Nooh2}, which form a 2-category $\TopSta$. These are a generalization of topological spaces, which as in \cite{Nooh2} have a well-behaved homotopy theory. As for the algebraic case, there is a topological realization functor $(-)^\top\colon \Ho(\TopSta)\ra\Top^{\bf ho}$. If $\B$ is a topological stack we define $H_*(\B)=H_*(\B^\top,R)$ and $H^*(\B)=H^*(\B^\top,R)$, as above.

We will be primarily interested in the following case (see \cite{JTU} for more details). Let $X$ be a compact manifold, $P\ra X$ a principal $\U(n)$-bundle, $\A_P$ the infinite-dimensional affine space of all connections $\nabla_P$ on $P$, and $\G_P=\Aut(P)$ the infinite-dimensional Lie group of gauge transformations of $P$. Then $\G_P$ acts continuously on $\A_P$, and we define $\B_P=\A_P/\G_P$ to be the quotient topological stack. Since~$\A_P$ is contractible, we have
\begin{equation*}
H_*(\B_P):=H_*\big(\B_P^\top,R\big)=H_*((\A_P\t E\G_P)/\G_P,R)\cong H_*(B\G_P,R).
\end{equation*}

\subsection{A geometric construction of vertex algebras}\label{wc23}

In \cite{Joyc12}, amongst other things, the second author defines new graded vertex algebra structures on the homology $H_*(\M)$ of moduli stacks $\M$ of objects in suitable abelian categories $\A$ or triangulated categories $\T$. We explain the construction in a special case. The next assumption sets out the data we need. As in Section~\ref{wc22}, throughout $R$ is a fixed $\Q$-algebra, and $H_*(-)=H_*(-,R)$ and $H^*(-)=H^*(-,R)$ denote (co)homology over~$R$.

\begin{ass}\label{wc2ass1}
Let $\A$ be a $\C$-linear abelian or exact category coming from algebraic geometry or representation theory, e.g., we could take $\A=\coh(X)$ or $\vect(X)$ for $X$ a smooth projective $\C$-scheme, or $\A=\modCQ$ the category of $\C$-representations of a quiver $Q$. Assume:
\begin{itemize}\itemsep=0pt
\item[(a)] We can form a natural moduli stack $\M$ of objects in $\A$, an Artin $\C$-stack, locally of finite type. Then $\C$-points of $\M$ are isomorphism classes $[E]$ of objects $E\in\A$, and the isotropy groups are $\Iso_\M([E])=\Aut(E)$.
\item[(b)] There is a natural morphism of Artin stacks $\Phi\colon \M\t\M\ra\M$ which on $\C$-points acts by $\Phi_*\colon ([E],[F])\mapsto[E\op F]$, for all objects $E,F\in\A$, and on isotropy groups acts by $\Phi_*\colon \Iso_{\M\t\M}([E],[F])\cong\Aut(E)\t\Aut(F)\ra \Iso_\M([E\op F])\cong\Aut(E\op F)$ by $(\la,\mu)\mapsto\bigl(\begin{smallmatrix}\la & 0 \\ 0 & \mu\end{smallmatrix}\bigr)$ for $\la\in\Aut(E)$ and $\mu\in\Aut(F)$, using the obvious matrix notation for \mbox{$\Aut(E\op F)$}. That is, $\Phi$ is the morphism of moduli stacks induced by direct sum in the abelian category~$\A$. It is associative and commutative in~$\Ho(\Art_\C)$.
\item[(c)] There is a natural morphism of Artin stacks $\Psi\colon [*/\bG_m]\t\M\ra\M$ which on $\C$-points acts by $\Psi_*\colon (*,[E])\mapsto[E]$, for all objects $E$ in $\A$, and on isotropy groups acts by $\Psi_*\colon \Iso_{[*/\bG_m]\t\M}(*,[E])\cong\bG_m\t\Aut(E)\ra \Iso_\M([E])\cong\Aut(E)$ by $(\la,\mu)\mapsto \la\mu=(\la\cdot\id_E)\ci\mu$ for $\la\in\bG_m$ and $\mu\in\Aut(E)$. Note that $\Psi$ is {\it not} the same as the projection $\pi_\M\colon [*/\bG_m]\t\M\ra\M$ from the product $[*/\bG_m]\t\M$, which acts on isotropy groups as $(\pi_\M)_*\colon (\la,\mu)\mapsto\mu$. We have identities in $\Ho(\Art_\C)$:
\begin{gather*}
\Psi\ci(\id_{[*/\bG_m]}\t\Phi)=\Phi\ci\bigl(\Psi\ci\Pi_{12},\Psi\ci\Pi_{13}\bigr)\colon \
[*/\bG_m]\t\M^2\longra\M,\\
\Psi\ci(\id_{[*/\bG_m]}\t\Psi)=\Psi\ci(\Om\t\id_\M)\colon \
[*/\bG_m]^2\t\M\longra\M,
\end{gather*}
where $\Pi_{ij}$ projects to the $i^{\rm th}$ and $j^{\rm th}$ factors, and
$\Om\colon [*/\bG_m]^2\ra[*/\bG_m]$ is induced by the morphism $\bG_m\t\bG_m\ra\bG_m$ mapping~$(\la,\mu)\mapsto\la\mu$.
\item[(d)] We are given a surjective quotient of abelian groups $K_0(\A)\twoheadrightarrow K(\A)$ of the Grothendieck group $K_0(\A)$ of $\A$. We write $\lb E\rb\in K(\A)$ for the class of $E\in\A$. We suppose that if $E\in\A$ with $\lb E\rb=0$ in $K(\A)$ then $E=0$.

We require that the map $\M(\C)\ra K(\A)$ mapping $E \mapsto\lb E\rb$ should be locally constant. This gives a decomposition $\M=\coprod_{\al\in K(\A)}\M_\al$ of $\M$ into open and closed $\C$-substacks $\M_\al\subset \M$ of objects in class $\al$, where $\M_0=\{[0]\}$. We write $\M_{\ne 0}=\M\sm\M_0$. We write $\Phi_{\al,\be}=\Phi\vert_{\M_\al\t\M_\be}\colon \M_\al\t\M_\be\ra\M_{\al+\be}$ and $\Psi_\al=\Psi\vert_{[*/\bG_m]\t\M_\al}\colon [*/\bG_m]\t\M_\al\ra\M_\al$.
\item[(e)] We are given a symmetric biadditive form $\chi\colon K(\A)\t K(\A)\ra \Z$.
\item[(f)] We are given signs $\ep_{\al,\be}\in\{\pm 1\}$ for all $\al,\be \in K(\A)$, such that for all $\al,\be,\ga\in K(\A)$ we~have
\begin{gather}
\label{wc2eq3}
\ep_{\al,\be} \cdot \ep_{\be,\al}=(-1)^{\chi(\al,\be) + \chi(\al,\al)\chi(\be,\be)}, \\
\label{wc2eq4}
\ep_{\al,\be}\cdot\ep_{\al+\be,\ga}=\ep_{\al,\be+\ga}\cdot\ep_{\be,\ga}, \\
\label{wc2eq5}
\ep_{\al,0}=\ep_{0,\al}=1.
\end{gather}
\item[(g)] We are given a perfect complex $\Th^\bu$ on $\M\t\M$, such that the restriction $\Th_{\al,\be}^\bu:=\Th^\bu\vert_{\M_\al\t\M_\be}$ to $\M_\al\t\M_\be$ has rank $\chi(\al,\be)$ for all $\al,\be\in K(\A)$, and there are isomorphisms of perfect complexes
\begin{gather}
\si_\M^*(\Th^\bu)\cong(\Th^\bu)^\vee[2n],
\label{wc2eq6}\\
(\Phi\t\id_\M)^*(\Th^\bu)\cong
\Pi_{13}^*(\Th^\bu)\op \Pi_{23}^*(\Th^\bu),
\label{wc2eq7}\\
(\id_\M\t\Phi)^*(\Th^\bu)\cong
\Pi_{12}^*(\Th^\bu)\op \Pi_{13}^*(\Th^\bu),
\label{wc2eq8}\\
(\Psi\t\id_\M)^*(\Th^\bu)\cong \Pi_1^*(L_{[*/\bG_m]})\ot \Pi_{23}^*(\Th^\bu),
\label{wc2eq9}\\
(\Pi_2,\Psi\ci\Pi_{13})^*(\Th^\bu)\cong \Pi_1^*(L_{[*/\bG_m]}^*)\ot\Pi_{23}^*(\Th^\bu).
\label{wc2eq10}
\end{gather}
Here \eq{wc2eq6} is on $\M\t\M$, where $\si_\M\colon \M\t\M\ra\M\t\M$ exchanges the factors, and $n\in\Z$. Equations \eq{wc2eq7}--\eq{wc2eq8} are on $\M\t\M\t\M$, and \eq{wc2eq9}--\eq{wc2eq10} on $[*/\bG_m]\t\M\t\M$. We write $\Pi_i$ for the projection to the $i^{\rm th}$ factor, and $\Pi_{ij}$ for the projection to the product of the $i^{\rm th}$ and $j^{\rm th}$ factors. We write $L_{[*/\bG_m]}\ra[*/\bG_m]$ for the line bundle corresponding to the weight 1 representation of $\bG_m=\C\sm\{0\}$ on~$\C$.
\end{itemize}	
\end{ass}

Although Assumption \ref{wc2ass1}(a)--(g) look like a lot of rather arbitrary data, as explained in \cite{Joyc12,Joyc13} and Section~\ref{wc4}, there are natural choices for all this data in many large classes of interesting examples.

\begin{dfn}\label{wc2def3}
Suppose Assumption \ref{wc2ass1} holds. Given all this data, we define a graded vertex algebra structure on the homology $H_*(\M)$ from Section~\ref{wc22}. The inclusion of the zero object $0\in\A$ gives a morphism $[0]\colon *\hookra\M$ inducing $R\cong H_*(*)
\ra H_*(\M)$, and we define $\boo\in H_*(\M)$ to be the image of $1\in R$ under this map. Taking homology of $\Psi$ gives a map
\begin{equation*}
\xymatrix@C=30pt{ H_*([*/ \bG_m])\!\ot_R\! H_*(\M)  \ar[r]^(0.57){\bt} & H_*([*/ \bG_m] \!\t\! \M) \ar[r]^(0.61){H_*(\Psi)} & H_*(\M). }
\end{equation*}
As $H^*([*/\bG_m])\cong\Hom_R(H_*([*/ \bG_m]),R)$, this is equivalent to a map
\begin{equation*}
\xymatrix@C=24pt{ H_*(\M) \ar[r] & H_*(\M) \hat\ot_R H^*([*/\bG_m]) \ar[r]^(0.61){\eq{wc2eq1}} & H_*(\M)[[z]], }
\end{equation*}
using equation \eq{wc2eq1}, and we denote the composition ${\rm e}^{zD}$.

The decomposition $\M=\coprod_{\al\in K(\A)}\M_\al$ induces an identification
\e
H_*(\M)=\bigop_{\al\in K(\A)}H_*(\M_\al).
\label{wc2eq11}
\e
For $u\in H_a(\M_\al)\subset H_*(\M)$ and $v\in H_b(\M_\be)\subset H_*(\M)$, define
\begin{gather}
Y(u,z)v=Y(z)(u\ot v)=  \ep_{\al, \be}(-1)^{a\chi(\be,\be)} \sum\nolimits_{i\ge 0} z^{\chi(\al, \be)-i}\nonumber\\
\hphantom{Y(u,z)v=Y(z)(u\ot v)= }{} \times H_*(\Phi) \ci \big({\rm e}^{zD}\ot \id\big) ((u \bt v) \cap c_i(\Th^\bu)).
\label{wc2eq12}
\end{gather}
Using \eq{wc2eq11}, for $n\in\Z$ and $\al\in K(\A)$ we write
\e
\hat H_n(\M_\al)= H_{n-\chi(\al,\al)}(\M_\al),\qquad \hat H_n(\M)=\bigop_{\al\in K(\A)}\hat H_n(\M_\al).
\label{wc2eq13}
\e
That is, $\hat H_*(\M)$ is $H_*(\M)$, but with grading shifted by $-\chi(\al,\al)$ in the component $H_*(\M_\al)\subset H_*(\M)$. The second author \cite{Joyc12} proves:	
\end{dfn}

\begin{thm}\label{wc2thm1}
In Definition \eqref{wc2def3}, $\big(\hat H_*(\M), \boo,{\rm e}^{zD},Y\big)$ is a graded vertex algebra over~$R$.
\end{thm}

Many examples of vertex algebras can be found in Frenkel--Ben-Zvi \cite{FrBZ}, Kac \cite{Kac2}, and Lepowsky--Li \cite{LeLi}. They are generally complicated to write down, as all nontrivial (i.e., non-holomorphic) vertex algebras are infinite-dimensional. When we can compute the vertex algebras arising from Theorem~\ref{wc2thm1} explicitly, they are usually minor modifications of ({\it super-}){\it lattice vertex algebras} \cite[Section~5.4]{FrBZ}, \cite[Section~5.4]{Kac2}. See Theorem~\ref{wc4thm1} below for a class of examples for $\A=D^b\coh(X)$ in which we can describe $\hat H_*(\M)$ and its vertex algebra structure explicitly.

Vertex algebras from quiver categories $\A=\modCQ$ or $\A=D^b\modCQ$ can also be written down explicitly \cite{Joyc12}. That for $\A=D^b\modCQ$ is the lattice vertex algebra used to construct the Kac--Moody Lie algebra $\g$ associated to the underlying Dynkin diagram (undirected graph) of $Q$, as in Kac \cite[Sections~1 and~5]{Kac1}, and $\A=\modCQ$ gives the vertex subalgebra yielding the positive part $\mathfrak{n}_+$ of~$\g$.

\subsection{Lie algebras from the vertex algebras of Section~\ref{wc23}}
\label{wc24}

In the situation of Section~\ref{wc23}, Proposition \ref{wc2prop1} makes $\hat H_{*+2}(\M)/D\big(\hat H_*(\M)\big)$ into a graded Lie algebra. We can interpret this as the shifted homology $\check H_*\big(\M^\pl\big)$ of a modification $\M^\pl$ of $\M$, which we now describe.

\begin{dfn}
\label{wc2def4}
Continue in the situation of Assumption \ref{wc2ass1} and Definition \ref{wc2def3}. Then $[*/\bG_m]$ is a group stack, and $\Psi\colon [*/\bG_m]\t\M\ra\M$ is an action of $[*/\bG_m]$ on $\M=\M_0\amalg\M_{\ne 0}$, which is trivial on $\M_0=\{[0]\}$, and free on $\M_{\ne 0}$. As explained in \cite{Joyc12}, we may take the quotient of~$\M$ by~$\Psi$ to get a stack $\M^\pl$, which we call the {\it projective linear moduli stack}, with projection $\Pi^\pl\colon \M\ra\M^\pl$, in a 2-co-Cartesian square in $\Art_\C$:
\begin{equation*}
\xymatrix@C=130pt@R=15pt{ *+[r]{[*/\bG_m]\t\M} \drtwocell_{}\omit^{}\omit{^{}} \ar[r]_(0.65){\Psi} \ar[d]^{\pi_\M} & *+[l]{\M} \ar[d]_{\Pi^\pl} \\
*+[r]{\M} \ar[r]^(0.35){\Pi^\pl} & *+[l]{\M^\pl.\!} }
\end{equation*}

The construction of $\M^\pl$ as $\M/[*/\bG_m]$ is known in the literature as {\it rigidification}, as in Abramovich--Olsson--Vistoli \cite{AOV} and Romagny~\cite{Roma}, and is written $\M^\pl=\M\!\!\fatslash\,\bG_m$ in \cite{AOV,Roma}. It is used, for example, to rigidify the Picard stack $\Pic(X)$ of line bundles on a~projective scheme $X$ to get the Picard scheme $\Pic(X)\!\!\fatslash\,\bG_m$. A typical example is that $\M$ contains a~component $[*/\GL(n,\C)]$, and $\M^\pl$ contains a corresponding component $[*/\PGL(n,\C)]$, where $\PGL(n,\C)=\GL(n,\C)/(\bG_m\cdot\Id_n)$. That is, the passage $\M\ra\M^\pl$ converts general linear isotropy groups $\GL(n,\C)$ to projective linear isotropy groups $\PGL(n,\C)$, which is why we call $\M^\pl$ `projective linear'.

We regard $\M^\pl$ as the moduli stack of all objects in $\A$ `up to projective linear isomorphisms'. Since $\Pi^\pl\colon \M\ra\M^\pl$ is a $[*/\bG_m]$-bundle it is an isomorphism on $\C$-points. Thus, $\C$-points $x\in\M^\pl(\C)$ correspond naturally to isomorphism classes~$[E]$ of objects $E\in\A$, as for $\M(\C)$, and we will write points of $\M^\pl(\C)$ as~$[E]$, and then $\Pi^\pl(\C)$ maps~$[E]\mapsto [E]$.

The isotropy groups of $\M^\pl$ satisfy $\Iso_{\M^\pl}([E])\cong\Iso_\M([E])/\bG_m$ for $E\ne 0$, where the $\bG_m$-subgroup of $\Iso_\M([E])$ is determined by the action of $\Psi$ on isotropy groups. Thus by Assumption~\ref{wc2ass1}(c) we see that
\e
\Iso_{\M^\pl}([E])\cong\Aut(E)/(\bG_m\cdot\id_E).\label{wc2eq14}
\e
The action of $\Pi^\pl$ on isotropy groups is given by the commutative diagram
\begin{equation*}
\xymatrix@C=150pt@R=15pt{ *+[r]{\Iso_\M([E])} \ar[r]_(0.4){\Pi^\pl_*} \ar[d]^\cong & *+[l]{\Iso_{\M^\pl}([E])} \ar[d]^\cong_{\eq{wc2eq14}} \\
*+[r]{\Aut(E)} \ar[r]^(0.4){\ep\,\longmapsto\, \ep\bG_m} & *+[l]{\Aut(E)/(\bG_m\cdot\id_E).\!\!} }	
\end{equation*}

The splitting $\M=\coprod_{\al\in K(\A)}\M_\al$ descends to $\M^\pl=\coprod_{\al\in K(\A)}\M_\al^\pl$, with $\M_\al^\pl=\M_\al\!\!\fatslash\,\bG_m$ for $\al\ne 0$ and $\M_0^\pl=\M_0=*$. Thus as for \eq{wc2eq11} we have
\e
H_*\big(\M^\pl\big)=\bigop_{\al\in K(\A)}H_*\big(\M_\al^\pl\big).
\label{wc2eq15}
\e
In a similar way to \eq{wc2eq13}, using \eq{wc2eq15}, for $n\in\Z$ and $\al\in K(\A)$ we write
\e
\check H_n\big(\M_\al^\pl\big)=H_{n+2-\chi(\al,\al)}\big(\M_\al^\pl\big),\qquad \check H_n\big(\M^\pl\big)=\bigop_{\al\in K(\A)}\check H_n\big(\M_\al^\pl\big).
\label{wc2eq16}
\e
That is, $\check H_*\big(\M^\pl\big)$ is $H_*\big(\M^\pl\big)$, but with grading shifted by $2-\chi(\al,\al)$ in the component $H_*\big(\M_\al^\pl\big)\subset H_*\big(\M^\pl\big)$.
\end{dfn}

\begin{rem}
\label{wc2rem1}
We now have {\it two different versions} $\M$ and $\M^\pl$ of the moduli stack of objects in~$\A$. Most of the literature on moduli stacks focusses on $\M$. However, for enumerative invariants, $\M^\pl$ {\it is often more useful}.

To see why, observe that one often forms enumerative invariants by forming {\it moduli schemes} $\M_\al^\rst(\tau)\subseteq\M_\al^\ss(\tau)$ of $\tau$-(semi)stable objects in $\A$ in class $\al$ in $K(\A)$, for $\tau$ a suitable stability condition. In good cases $\M_\al^\rst(\tau)$ has a perfect obstruction theory, and $\M_\al^\ss(\tau)$ is proper. If $\M_\al^\rst(\tau)=\M_\al^\ss(\tau)$, then by Behrend and Fantechi \cite{BeFa} we have a virtual class $[\M_\al^\ss(\tau)]_\virt$ in~$H_*(\M_\al^\ss(\tau))$.

If $E\in\A$ is $\tau$-stable then $\Aut(E)=\bG_m$, and thus $\Iso_\M([E])=\bG_m$ and $\Iso_{\M^\pl}([E])=\{1\}$. Now $\C$-schemes may be regarded as examples of Artin $\C$-stacks with trivial isotropy groups. Then $\M_\al^\rst(\tau)\subseteq\M^\pl$ is an {\it open substack}, so we can regard $[\M_\al^\ss(\tau)]_\virt$ as lying in $H_*\big(\M^\pl\big)$. But in general there is no natural morphism $\M_\al^\rst(\tau)\ra\M$, so we cannot map $[\M_\al^\ss(\tau)]_\virt$ to~$H_*(\M)$.
\end{rem}

The next theorem, proved in \cite{Joyc12}, gives a geometric interpretation of the graded Lie algebra~$\hat H_{*+2}(\M)/D\big(\hat H_*(\M)\big)$.

\begin{thm}\label{wc2thm2}
Work in the situation of Assumption {\rm\ref{wc2ass1}} and Definitions~{\rm \ref{wc2def3}} and~{\rm \ref{wc2def4}}, and consider the graded Lie algebra $\hat H_{*+2}(\M)/D\big(\hat H_*(\M)\big)$ constructed by combining Proposition~{\rm \ref{wc2prop1}} and Theorem~{\rm \ref{wc2thm1}}. Then  $\Pi^\pl\colon \M\ra\M^\pl$ gives a morphism $H_*\big(\Pi^\pl\big)\colon H_*(\M)\ab\ra H_*\big(\M^\pl\big)$. It is surjective, with kernel $D(H_*(\M))$. This induces an isomorphism
\begin{equation*}
H_*\big(\Pi^\pl\big)_*\colon \ H_*(\M)/D(H_*(\M))\longra H_*\big(\M^\pl\big).
\end{equation*}
Comparing \eq{wc2eq13} and \eq{wc2eq16}, we see this is an isomorphism for all~$n\in\Z$:
\e
\hat H_{n+2}(\M)/D\big(\hat H_n(\M)\big)\longra \check H_n\big(\M^\pl\big).
\label{wc2eq17}
\e
Thus there is a unique Lie bracket $[\,,\,]$ on $\check H_*\big(\M^\pl\big)$ making it into a graded Lie algebra, such that \eq{wc2eq17} is a Lie algebra isomorphism. Hence $\check H_0\big(\M^\pl\big)$ is a Lie algebra.
\end{thm}

\begin{rem}\label{wc2rem2}\quad
\begin{enumerate}\itemsep=0pt
\item[(a)] The fact that \eq{wc2eq17} is an isomorphism depends on the assumptions that $R$ is a $\Q$-algebra, that we are working with an abelian or exact category $\A$ rather than a triangulated category $\T$, and on another assumption we have suppressed that holds in all the cases we will discuss, which imply that the $[*/\bG_m]$-fibration $\Pi^\pl\colon \M\ra\M^\pl$ is `rationally trivial' in the sense of~\cite{Joyc12}.

In the more general cases discussed in \cite{Joyc12}~-- in particular, if either $R$ is not a $\Q$-algebra, or we work with triangulated categories $\T$ such as $D^b\coh(X)$ -- then equation \eq{wc2eq17} need not be an isomorphism, although under mild conditions it is an isomorphism when $n=0$, which is what matters in this paper.

Proving a conjecture in \cite{Joyc12}, Upmeier \cite{Upme} gives a direct construction of the graded Lie bracket $[\,,\,]$ on $\check H_*\big(\M^\pl\big)$, using a new notion of characteristic class of perfect complexes over $[*/\bG_m]$-fibrations $\Pi\colon S\ra T$ of stacks $S$, $T$, called the {\it projective Euler class}, such that \eq{wc2eq17} is a morphism of graded Lie algebras in every case, whether or not \eq{wc2eq17} is an isomorphism.

\item[(b)] As in \cite{Joyc12}, in cases in which we can compute $\check H_*\big(\M^\pl\big)$ explicitly, $\check H_0\big(\M^\pl\big)$ is often (a~minor modification of) a Kac--Moody Lie algebra $\g$ as in Kac \cite[Sections~1 and~5]{Kac1} (when $\A$ is a~derived category $D^b\coh(X)$ or $D^b\modCQ$), or its positive part $\mathfrak{n}_+$ (when $\A=\modCQ$). For example, if $Q$ is a quiver whose underlying graph is an ADE Dynkin diagram and $\A=D^b\modCQ$ then $\check H_0\big(\M^\pl\big)=\g$ is the corresponding finite-dimensional ADE Lie algebra $\mathfrak{sl}(n+1,\C)$, $\so(2n,\C)$, $\mathfrak{e}_6$, $\mathfrak{e}_7$, $\mathfrak{e}_8$, and $\A=\modCQ$ gives its positive part~$\mathfrak{n}_+$.

Vertex algebras were originally introduced in mathematics by Borcherds~\cite{Borc} to better understand the construction of Kac--Moody-type Lie algebras.
\end{enumerate}
\end{rem}

\subsection{Morphisms of the vertex and Lie algebras of Sections~\ref{wc23}--\ref{wc24}}
\label{wc25}

We now construct morphisms between the graded vertex algebras and Lie algebras of Sections~\ref{wc23}--\ref{wc24}, that will be used in the proofs of our main results in~Section~\ref{wc6}.

\begin{dfn}\label{wc2def5}
Let $\A$, $\M$, $\Phi$, $\Psi$, $K(\A)$, $\chi$, $\ep_{\al,\be}$, $\Th^\bu$ and
$\A'$, $\M'$, $\Phi'$, $\ab\Psi'$, $\ab K(\A')$, $\ab\chi'$, $\ab\ep'_{\al,\be}$, $\Th^{\prime\bu}$ be two sets of data satisfying Assumption~\ref{wc2ass1}. Write
$\big(\hat H_*(\M),\ab \boo,\ab {\rm e}^{zD},\ab Y\big)$ and $\big(\hat H_*(\M'), \boo',{\rm e}^{zD'},Y'\big)$ for the corresponding graded vertex algebras from Theorem~\ref{wc2thm1}. Suppose:
\begin{itemize}\itemsep=0pt
\item[(a)] We are given a $\C$-linear exact functor $\Si\colon \A\ra\A'$. This should induce a morphism $\si\colon \M\ra\M'$ of moduli stacks, which acts on $\C$-points as $\si([E])=[\Si(E)]$. The induced morphism $\Si_*\colon K_0(\A)\ra K_0(\A')$ descends to $\Si_*\colon K(\A)\ra K(\A')$, with $\Si_*(\lb E\rb)=\lb \Si(E)\rb$. In $\Ho(\Art_\C)$ we have
\begin{equation*}
\Phi'\ci(\si\t\si)=\si\ci\Phi,\qquad
\Psi'\ci(\id_{[*/\bG_m]}\t\si)=\si\ci\Psi.
\end{equation*}
\item[(b)] We are given a biadditive morphism $\xi\colon K(\A)\t K(\A)\ra\Z$.
\item[(c)] We are given a vector bundle $F\ra\M\t\M$ of mixed rank, such that $F_{\al,\be}:=F\vert_{\M_\al\t\M_\be}$ has rank $\xi(\al,\be)$ for all $\al,\be\in K(\A)$. We also write $G\ra\M$ for the vector bundle $\De_\M^*(F^*)$, where $\De_\M\colon \M\ra\M\t\M$ is the diagonal morphism. Then $G_\al:=G\vert_{\M_\al}$ has rank $\xi(\al,\al)$.
\end{itemize}

All this data should satisfy:
\begin{itemize}
\itemsep=0pt
\setlength{\parsep}{0pt}
\item[(i)] $\chi'(\Si_*(\al),\Si_*(\be))=\chi(\al,\be)+\xi(\al,\be)+\xi(\be,\al)$ for all $\al,\be\in K(\A)$.
\item[(ii)] $\ep'_{\Si_*(\al),\Si_*(\be)}=(-1)^{\xi(\al,\be)}\ep_{\al,\be}$ for all $\al,\be\in K(\A)$.
\item[(iii)] As for \eq{wc2eq7}--\eq{wc2eq10}, there are isomorphisms of vector bundles
\begin{gather}
(\Phi\t\id_\M)^*(F) \cong
\Pi_{13}^*(F)\op \Pi_{23}^*(F),
\label{wc2eq18}\\
(\id_\M\t\Phi)^*(F) \cong
\Pi_{12}^*(F)\op \Pi_{13}^*(F),
\label{wc2eq19}\\
(\Psi\t\id_\M)^*(F) \cong \Pi_1^*(L_{[*/\bG_m]})\ot \Pi_{23}^*(F),
\label{wc2eq20}\\
(\Pi_2,\Psi\ci\Pi_{13})^*(F) \cong \Pi_1^*(L_{[*/\bG_m]}^*)\ot\Pi_{23}^*(F).
\label{wc2eq21}
\end{gather}
\item[(iv)] In $K_0(\Perf(\M\t\M))$ we have
\begin{equation*}
(\si\t\si)^*([\Th^{\prime\bu}])=[\Th^\bu]+[F]+[\si_\M^*(F^*)],
\end{equation*}
where $\si_\M\colon \M\t\M\ra\M\t\M$ exchanges the factors. (This implies~(i).)
\item[(v)] There is a vector bundle $G^\pl\ra\M^\pl$ with~$G\cong\big(\Pi^\pl\big)^*\big(G^\pl\big)$.
\end{itemize}

Write $c_\top(G)\in H^*(\M)$ for the top Chern class $c_{\rank G}(G)$ of $G$. Note that as $\rank G$ depends on the component $\M_\al\subset\M$ as in~(c), we have $c_\top(G)\vert_{\M_\al}=c_{\xi(\al,\al)}(G_\al)$. Define an $R$-module morphism
\e
\Om\colon \ H_*(\M)\longra H_*(\M')\qquad\text{by}\quad \Om(u)=H_*(\si)\bigl(u\cap c_\top(G)\bigr).
\label{wc2eq22}
\e
Here if $u\in H_a(\M_\al)$ then $\Om(u)\in H_{a-2\xi(\al,\al)}(\M'_{\Si_*(\al)})$. Combining this with \eq{wc2eq13} and~(i), we see that $\Om\colon \hat H_*(\M)\ra\hat H_*(\M')$ is grading-preserving. The next theorem is proved in~\cite{Joyc12}:
\end{dfn}

\begin{thm}\label{wc2thm3}
In Definition {\rm \ref{wc2def5}}, $\Om\colon \big(\hat H_*(\M), \boo,{\rm e}^{zD},Y\big)\ra
\big(\hat H_*(\M'),\ab \boo',\ab {\rm e}^{zD'},\ab Y'\big)$ is a morphism of graded vertex algebras.
\end{thm}

\begin{rem}
\label{wc2rem3}
To prove Theorem \ref{wc2thm3}, it is essential that $G\ra\M$ should be a {\it vector bundle}, not just a perfect complex. This ensures that $\xi(\al,\al)=\rank G_\al\ge 0$ and $c_i(G_\al)=0$ for $i>\xi(\al,\al)$, which are used in the proof.

One consequence is that we cannot make the analogues of Definition \ref{wc2def5} and Theorem \ref{wc2thm3} work when we replace $\A$, $\A'$ by {\it triangulated categories} $\T$, $\T'$, such as $\T=D^b\coh(X)$, as allowed in \cite{Joyc12}, because we cannot satisfy the conditions of Definition~\ref{wc2def5} with nonzero vector bundles~$F$, $G$ in the triangulated case -- for example, unless $\xi\equiv 0$ we would have $\xi(\al,\be)<0$ for some~$\al$,~$\be$ with $\M_\al,\M_\be\ne\es$, and then Definition~\ref{wc2def5}(c) gives $\rank F_{\al,\be}<0$, a contradiction.
\end{rem}

\begin{dfn}\label{wc2def6}
Work in the situation of Definition \ref{wc2def5}. Then $\si\colon \M\ra\M'$ descends to a~morphism $\si^\pl\colon \M^\pl\ra\M^{\prime\pl}$ with $\si^\pl\ci\Pi^\pl=\Pi^{\prime\pl}\ci\si$ in $\Ho(\Art_\C)$. As for \eq{wc2eq22}, define an $R$-module morphism
\e
\Om^\pl\colon \ H_*\big(\M^\pl\big)\longra H_*\big(\M^{\prime\pl}\big)\qquad\text{by}\quad \Om^\pl(u)=H_*\big(\si^\pl\big)\big(u\cap c_\top\big(G^\pl\big)\big).
\label{wc2eq23}
\e
By \eq{wc2eq16} we see that $\Om^\pl\colon \check H_*\big(\M^\pl\big)\longra\check H_*\big(\M^{\prime\pl}\big)$ is grading preserving. From $\si^\pl\ci\Pi^\pl=\Pi^{\prime\pl}\ci\si$, $G\cong\big(\Pi^\pl\big)^*\big(G^\pl\big)$, and \eq{wc2eq22}--\eq{wc2eq23} we see that the following diagram commutes:
\begin{equation*}
\xymatrix@C=120pt@R=15pt{
*+[r]{H_*(\M)} \ar[r]_{\Om\vert_{H_*(\M)}} \ar[d]^{H_*(\Pi^\pl)} & *+[l]{H_*(\M')} \ar[d]_{H_*(\ti\Pi^\pl)} \\
*+[r]{H_*\big(\M^\pl\big)} \ar[r]^{\Om^\pl} & *+[l]{H_*\big(\M^{\prime\pl}\big).\!} }	
\end{equation*}
Hence Theorems \ref{wc2thm1}, \ref{wc2thm2}, and \ref{wc2thm3} imply:
\end{dfn}

\begin{cor}\label{wc2cor1}
In Definition {\rm \ref{wc2def6}}, $\Om^\pl\colon \check H_*\big(\M^\pl\big)\ra \check H_*\big(\M^{\prime\pl}\big)$ is a morphism of the graded Lie algebras in Theorem~{\rm \ref{wc2thm2}}.
\end{cor}

\section[Background on stability conditions, wall-crossing formulae, and pair invariants]{Background on stability conditions, wall-crossing formulae,\\ and pair invariants}
\label{wc3}

Next we explain parts of the second author's series \cite{Joyc4,Joyc2,Joyc5,Joyc6,Joyc3,Joyc7}, focussing in particular on Ringel--Hall algebras and Lie algebras, motivic invariants, and their wall-crossing formulae. Almost all of this section will not be used later, but appears only for motivation, because of the comparison between Theorem~\ref{wc3thm5} and Conjecture~\ref{wc4conj1}. So we omit details in places.

\subsection{Constructible functions and stack functions}\label{wc31}

The theory of constructible functions on Artin stacks was developed by the second author \cite{Joyc2}. We use the notation on stacks from Section~\ref{wc22}. All stacks in Section~\ref{wc3} are assumed to have affine geometric stabilizers.

\begin{dfn}\label{wc3def1}
Let $X$, $Y$ be Artin $\C$-stacks. We call $C\subseteq X(\C)$ {\it constructible} if $C=\bigcup_{i\in I}X_i(\C)$, where $\{X_i\colon i\in I\}$ is a~finite collection of finite type Artin $\C$-substacks $X_i$ of~$X$. Let $R$ be a~commutative $\Q$-algebra. A function $f\colon X(\C)\ra R$ is called {\it constructible} if $f(X(\C))$ is finite and $f^{-1}(c)$ is a constructible set in $X(\C)$ for each $c\in f(X(\C))\sm\{0\}$. Write $\CF(X)$ for the $R$-module of $R$-valued constructible functions on $X$. If $C\subseteq X(\C)$ is constructible we write $\de_C\in\CF(X)$ for its characteristic function. If $\phi\colon X\ra Y$ is a representable 1-morphism then \cite[Section~4.3]{Joyc2} defines the $R$-linear {\it pushforward} $\phi_*\colon \CF(X) \ra\CF(Y)$, using Euler characteristics. If $\th\colon X\ra Y$ is a finite type 1-morphism then \cite[Section~5.2]{Joyc2} defines the $R$-linear {\it pullback} $\th^*\colon \CF(Y)\ra\CF(X)$.
\end{dfn}

Here \cite[Theorems~5.4, 5.6 and Definition~5.5]{Joyc2} are some properties of
these.

\begin{thm}\label{wc3thm1}
Let $W$, $X$, $Y$, $Z$ be Artin $\C$-stacks, and $\be\colon X\ra Y$, $\ga\colon Y\ra Z$ be $1$-morphisms. Then
\begin{gather}
(\ga\ci\be)_* =\ga_*\ci\be_*\colon \ \CF(X)\longra\CF(Z),
\label{wc3eq1}\\
(\ga\ci\be)^* =\be^*\ci\ga^*\colon \ \CF(Z)\longra\CF(X),
\label{wc3eq2}
\end{gather}
supposing $\be,\ga$ representable in \eq{wc3eq1}, and of finite type
in~\eq{wc3eq2}. If
\e
\begin{gathered}
\xymatrix{
W \ar[r]_\eta \ar[d]^\th & Y \ar[d]_\psi \\
X \ar[r]^\phi & Z }
\end{gathered}
\qquad
\begin{gathered}
\text{is a $2$-Cartesian square with}\\
\text{$\eta$, $\phi$ representable and}\\
\text{$\th$, $\psi$ of finite type, then}\\
\text{the following commutes:}
\end{gathered}
\qquad
\begin{gathered}
\xymatrix@C=35pt{
\CF(W) \ar[r]_{\eta_*} & \CF(Y) \\
\CF(X) \ar[r]^{\phi_*} \ar[u]_{\th^*} & \CF(Z).
\ar[u]^{\psi^*} }
\end{gathered}
\label{wc3eq3}
\e
\end{thm}

The second author \cite{Joyc3} also introduced `stack functions', a universal generalization of constructible functions on stacks. To each Artin $\C$-stack $X$ we assign an $R$-module $\SF(X)$ of `stack functions', which is generated by morphisms $f\colon S\ra X$ with $S$ a finite type Artin $\C$-stack and~$f$ a~representable 1-morphism, subject to some relations we will not give. They have the same package of properties as constructible functions above:
\begin{itemize}\itemsep=0pt
\item[(i)] A constructible set $C \subset X(\C)$ has a `characteristic function' $\bar\de_C\in\SF(X)$.
\item[(ii)] There are $R$-linear {\it pushforwards} $\phi_*\colon \SF(X)\ra\SF(Y)$ by representable 1-morphisms $\phi\colon X\ra Y$.
\item[(iii)] There are $R$-linear {\it pullbacks} $\th^*\colon \SF(Y) \ra\SF(X)$ by finite type 1-morphisms $\th\colon X\ra Y$.
\item[(iv)] The analogues of \eq{wc3eq1}--\eq{wc3eq3} hold. There are natural transformations $\SF(X)\ra\CF(X)$, defined using fibrewise Euler characteristics, which commute with pushforwards and pullbacks, and map $\bar\de_C\mapsto\de_C$ in~(i).
\end{itemize}
As in \cite{Joyc4,Joyc2,Joyc5,Joyc6,Joyc3,Joyc7}, stack functions are useful for studying {\it motivic invariants}.

\subsection{Ringel--Hall algebras and Lie algebras}\label{wc32}

Following \cite{Joyc5}, we now explain how to define the {\it Ringel--Hall algebra} of a $\C$-linear abelian category, using either constructible functions or stack functions. The next assumption sets out the data we need. It is similar to Assumption \ref{wc2ass1}, and often both hold at once.

\begin{ass}
\label{wc3ass1}
Let $\A$ be a $\C$-linear abelian category coming from algebraic geometry or representation theory, e.g., we could take $\A=\coh(X)$ for $X$ a smooth projective $\C$-scheme, or $\A=\modCQ$ the category of $\C$-representations of a quiver $Q$. Assume:
\begin{itemize}\itemsep=0pt
\item[(a)] We can form a natural moduli stack $\M$ of objects $E$ in $\A$, an Artin $\C$-stack, locally of finite type. Then $\C$-points of $\M$ are isomorphism classes $[E]$ of objects $E\in\A$, and the isotropy groups are $\Iso_\M([E])=\Aut(E)$.
\item[(b)] We can form a natural moduli stack $\fExact$ of short exact sequences $E_\bu=(0\ra E_1\ra E_2\ra E_3\ra 0)$ in $\A$, an Artin $\C$-stack, locally of finite type. There are 1-morphisms $\pi_i\colon \fExact\ra\M$ for $i=1,2,3$ acting by $\pi_i\colon E_\bu\mapsto E_i$ on $\C$-points. Here $\pi_2\colon \fExact\ra\M$ is representable, and $(\pi_1,\pi_3)\colon \fExact\ra\M\t\M$ is finite type.
\item[(c)] We are given a surjective quotient $K_0(\A)\twoheadrightarrow K(\A)$ of the Grothendieck group $K_0(\A)$ of~$\A$. We write $\lb E\rb\in K(\A)$ for the class of $E\in\A$. We require that the map $\M(\C)\ra K(\A)$ mapping $E \mapsto\lb E\rb$ should be locally constant. This gives a decomposition $\M=\coprod_{\al\in K(\A)}\M_\al$ of $\M$ into open and closed $\C$-substacks $\M_\al\subset \M$ of objects in class $\al$. We also suppose that $\M_0=\{[0]\}$, that is, $0\in\A$ is the only object in class $0\in K(\A)$.

Define the {\it positive cone} $C(\A)\subset K(\A)$ by~$C(\A)=\bigl\{\lb E\rb\colon 0\ne E\in\A\bigr\}$.
\item[(d)] For all $\al,\be,\ga\in K(\A)$ there is a 2-commutative diagram in $\Art_\C$ with all squares 2-Cartesian:
\begin{gather}
\begin{gathered}
\xymatrix@!0@C=175pt@R=32pt{
*+[r]{\M_\al\!\t\!\M_\be\!\t\!\M_\ga} & \fExact_{\al,\be}\!\t\M_\ga \ar[l]^(0.33){(\pi_1,\pi_3)\t\id_{\M_\ga}} \ar[r]_(0.43){\pi_2\t\id_{\M_\ga}} & *+[l]{\M_{\al+\be}\!\t\!\M_\ga}
\\
*+[r]{\M_\al\!\t\!\fExact_{\be,\ga}} \ar[u]_{\id_{\M_\al}\t(\pi_1,\pi_3)} \ar[d]^{\id_{\M_\al}\t\pi_2} & \cN_{\al,\be,\ga} \ar[l]_(0.3){\Pi_{\al,(\be,\ga)}} \ar[r]^(0.4){\Pi_{(\al+\be,\ga)}} \ar[u]_{\Pi_{(\al,\be),\ga}} \ar[d]^{\Pi_{(\al,\be+\ga)}}\drtwocell_{}\omit^{}\omit{^{}} \urtwocell_{}\omit^{}\omit{^{}} \dltwocell_{}\omit^{}\omit{^{}} \ultwocell_{}\omit^{}\omit{^{}} & *+[l]{\fExact_{\al+\be,\ga}} \ar[u]^{(\pi_1,\pi_3)} \ar[d]_{\pi_2} \\
*+[r]{\M_\al\!\t\!\M_{\be+\ga}} & \fExact_{\al,\be+\ga} \ar[l]_(0.4){(\pi_1,\pi_3)} \ar[r]^(0.4){\pi_2} & *+[l]{\M_{\al+\be+\ga}.\!}}\!\!\!\!\!\!
\end{gathered}
\label{wc3eq4}
\end{gather}
\item[] Here $\fExact_{\al,\be}$ is the moduli stack of short exact sequences $0\ra E_\al\ra E_{\al+\be}\ra E_\be\ra 0$ in $\A$, where $\lb E_\de\rb=\de$, and $\cN_{\al,\be,\ga}$ is the moduli stack of flags $E_\al\subset E_{\al+\be}\subset E_{\al+\be+\ga}$ of subobjects in $\A$. In the language of \cite{Joyc4}, $\cN_{\al,\be,\ga}$ is a moduli stack of {\it configurations} in~$\A$.
\end{itemize}
\end{ass}

Assumption \ref{wc3ass1} holds in large classes of interesting examples, as in~\cite{Joyc4,Joyc5,Joyc6,Joyc7}.

\begin{dfn}\label{wc3def2}
Suppose Assumption \ref{wc3ass1} holds. Following \cite[Section~4]{Joyc5}, define an $R$-bilinear product $*\colon \CF(\M)\t\CF(\M)\ra\CF(\M)$ by
\begin{gather*}
f*g=(\pi_2)_*\ci(\pi_1,\pi_3)^*(f\bt g),
\end{gather*}
using $\pi_2\colon \fExact\ra\M$, $(\pi_1,\pi_3)\colon \fExact\ra\M\t\M$ and the notation of Section~\ref{wc31}. Considering the commutative diagram from \eq{wc3eq3} and \eq{wc3eq4}
\begin{equation*}
\xymatrix@!0@C=180pt@R=38pt{
*+[r]{\CF(\M_\al\!\t\!\M_\be\!\t\!\M_\ga)} \ar[d]^{(\id_{\M_\al}\t(\pi_1,\pi_3))^*} \ar[r]_(0.67){\raisebox{-10pt}{$\st ((\pi_1,\pi_3)\t\id_{\M_\ga})^*$}} & \CF(\fExact_{\al,\be}\!\t\M_\ga) \ar[d]^{(\Pi_{(\al,\be),\ga})^*} \ar[r]_(0.43){(\pi_2\t\id_{\M_\ga})_*} & *+[l]{\CF(\M_{\al+\be}\!\t\!\M_\ga)} \ar[d]_{(\pi_1,\pi_3)^*} \\
*+[r]{\CF(\M_\al\!\t\!\fExact_{\be,\ga})} \ar[r]^(0.7){\Pi_{\al,(\be,\ga)}^*} \ar[d]^{(\id_{\M_\al}\t\pi_2)_*} & \CF(\cN_{\al,\be,\ga})  \ar[r]^(0.4){(\Pi_{(\al+\be,\ga)})_*}  \ar[d]^{(\Pi_{(\al,\be+\ga)})_*} & *+[l]{\CF(\fExact_{\al+\be,\ga})}  \ar[d]_{(\pi_2)_*} \\
*+[r]{\CF(\M_\al\!\t\!\M_{\be+\ga})} \ar[r]^(0.6){(\pi_1,\pi_3)^*} & \CF(\fExact_{\al,\be+\ga})  \ar[r]^(0.4){(\pi_2)_*} & *+[l]{\CF(\M_{\al+\be+\ga}),\!}}
\end{equation*}
we find that $(f_\al*f_\be)* f_\ga=f_\al*(f_\be*f_\ga)$ for $f_\al,f_\be,f_\ga\in \CF(\M)$. Hence $\CF(\M)$ is an {\it associative $R$-algebra}, with product $*$ and identity $\de_{[0]}$, where $\de_{[0]}\colon \M(\C)\ra R$ is given by $\de_{[0]}([E])=1$ if $E=0$ and $\de_{[0]}([E])=0$ otherwise.

The same definition makes the stack functions $\SF(\M)$ into an associative $R$-algebra. We call $\CF(\M)$, $\SF(\M)$ {\it Ringel--Hall algebras}, as the construction originated with the work of Ringel on Hall algebras~\cite{Ring}. Observe that as $*$ is associative, $\CF(\M)$, $\SF(\M)$ are also {\it Lie algebras} over~$R$, with Lie bracket
\e
[f,g]=f*g-g*f.\label{wc3eq6}
\e

An object $E\in\A$ is called {\it indecomposable} if we cannot write $E\cong E_1\op E_2$ for $E_1,E_2\not\cong 0$ in~$\A$. A constructible function $f\in\CF(\M)$ is called {\it supported on indecomposables} if $f([E])\ne 0$ for $[E]\in\M(\C)$ implies that $E$ is indecomposable. Write $\CFi(\M)\subset\CF(\M)$ for the $R$-submodule of $f\in\CF(\M)$ supported on indecomposables. Then \cite[Theorem~4.9]{Joyc5} says that $\CFi(\M)$ is closed under the Lie bracket~\eq{wc3eq6}, and so is a {\it Lie subalgebra} of $\CF(\M)$. Note that $\CFi(\M)$ is generally not closed under~$*$.

In \cite[Section~5]{Joyc5}, the second author defines an $R$-subalgebra $\SFa(\M)\subset\SF(\M)$ of stack functions {\it with algebra stabilizers}, and a Lie subalgebra $\SFai(\M)\subset \SFa(\M)$ of stack functions {\it supported on virtual indecomposables}, which is an analogue of the Lie subalgebra~$\CFi(\M)\subset\CF(\M)$.
\end{dfn}

\subsection{(Weak) stability conditions on abelian categories}\label{wc33}

Next we summarize some material from \cite{Joyc6}. See also Rudakov~\cite{Ruda}.

\begin{dfn}\label{wc3def3}
Suppose Assumption \ref{wc3ass1} holds. Let $(T,\leq)$ be a totally ordered set and $\tau\colon C(\A)\ra T$ be a map. We call $(\tau,T,\leq)$ a {\it weak stability condition} on $\A$ if for all $\al,\be,\ga \in C(\A)$ with $\be=\al+\ga$, either
$\tau(\al) \leq \tau(\be) \leq \tau(\ga)$, or~$\tau(\al) \geq \tau(\be) \geq \tau(\ga)$.

We call $(\tau,T,\leq)$ a {\it stability condition} if for all such $\al$, $\be$, $\ga$, either
$\tau(\al)<\tau(\be)<\tau(\ga)$, or $\tau(\al)>\tau(\be)>\tau(\ga)$, or~$\tau(\al)=\tau(\be)=\tau(\ga)$.

Let $(\tau,T,\leq)$ be a weak stability condition. An object $E$ of $\A$ is called
\begin{itemize}\itemsep=0pt
\item[(i)] {\it $\tau$-stable} if $\tau([E'])<\tau([E/E'])$ for all subobjects $E'\subset E$ with $E'\ne 0,E$.
\item[(ii)] {\it $\tau$-semistable} if $\tau([E'])\!\leq\!\tau([E/E'])$ for all $E'\subset E$ with $E'\ne 0,E$.
\item[(iii)] {\it $\tau$-unstable} if it is not $\tau$-semistable.
\item[(iv)] {\it strictly $\tau$-semistable} if it is $\tau$-semistable but not $\tau$-stable.
\end{itemize}

If $(\tau,T,\le)$, $\big(\ti\tau,\ti T,{\le}\big)$ are weak stability conditions on $\A$, we say that $\big(\ti\tau,\ti T,{\le}\big)$ {\it dominates} $(\tau,T,\leq)$ if $\tau (\al)\leq\tau(\be)$ implies $\ti\tau(\al)\leq\ti\tau(\be)$ for all~$\al,\be\in C(\A)$.

We call a weak stability condition $(\tau,T,\leq)$ on $\A$ {\it permissible} if
\begin{itemize}\itemsep=0pt
\item[(a)] $\A$ is $\tau$-{\it artinian}, that is, there are no infinite chains $\cdots\subsetneq E_3\subsetneq  E_2\subsetneq E_1$ of subobjects  in~$\A$ with $\tau([E_{n+1} ])\leq\tau([E_n/E_{n+1}])$ for all $n=1,2,\dots$.
\item[(b)] For each $\al\in K(\A)$, write $\M^\ss_\al(\tau)=\bigl\{[E]\in\M(\C)\colon E$ is $\tau$-semistable, $\lb E\rb=\al\bigr\}$. Then $\M^\ss_\al(\tau)$ is a constructible set, as in Definition~\ref{wc3def1}.
\end{itemize}
We also write $\M^\rst_\al(\tau)=\bigl\{[E]\in\M(\C)\colon E$ is $\tau$-stable, $\lb E\rb=\al\bigr\}$.
\end{dfn}

\begin{rem}\label{wc3rem1}\quad
\begin{enumerate}\itemsep=0pt
\item[(a)] As in \cite[Section~4.1]{Joyc6}, there are some theorems which hold for stability conditions but are false for weak stability conditions. However, we will not use any of these in this paper, so we will work with weak stability conditions.

\item[(b)] In the examples we will be interested in, $\M^\rst_\al(\tau)$ and $\M^\ss_\al(\tau)$ will actually be {\it finite type open substacks} in $\M$, but for the theory of \cite{Joyc6,Joyc7} it is sufficient to suppose only that~$\M^\ss_\al(\tau)$ is a constructible set.	
\end{enumerate}
\end{rem}

\begin{ex}\label{wc3ex1}
For any $\A$ satisfying Assumption \ref{wc3ass1}, we may take $T=*$ to be the point, $\le$ the unique order on $*$, and $\tau_{\rm tr}\colon C(\A)\ra *$ to be the projection. Then $(\tau_{\rm tr},*,\le)$ is a~stability condition on $\A$, the {\it trivial stability condition}. Every $E\in\A$ is $\tau_{\rm tr}$-semistable, so $\M_\al^\ss(\tau_{\rm tr})=\M_\al(\C)$ for $\al\in C(\A)$. Thus $(\tau_{\rm tr},*,\le)$ is permissible if and only if $\M_\al$ is of finite type for all $\al\in C(\A)$. The trivial stability condition dominates all weak stability conditions on~$\A$.	
\end{ex}

The next two examples are taken from \cite[Sections~4.3--4.4]{Joyc6}.

\begin{ex}\label{wc3ex2}
As explained in more detail in Section~\ref{wc5}, let $Q=(Q_0,Q_1,h,t)$ be a quiver, and $\A=\modCQ$ be the abelian category of $\C$-representations of $Q$. Write objects of $\A$ as $(\bs V,\bs\rho)=((V_v)_{v\in Q_0},(\rho_e)_{e\in Q_1})$, where $V_v$ is a finite-dimensional $\C$-vector space and $\rho_e\colon V_{t(e)}\ra V_{h(e)}$ a linear map. Define the {\it dimension vector} of $(\bs V,\bs\rho)$ to be $\bs d=\bdim(\bs V,\bs\rho)\in\N^{Q_0}\subset\Z^{Q_0}$, where $\bs d(v)=\dim_\C V_v$ for $v\in Q_0$. Define $K(\A)=\Z^{Q_0}$ to be the lattice of dimension vectors, with $\lb(\bs V,\bs\rho)\rb=\bdim(\bs V,\bs\rho)$. Then the positive cone is~$C(\A)=\N^{Q_0}\sm\{0\}$.

Choose real numbers $\mu_v\in\R$ for all $v\in Q_0$. Define a map $\mu\colon C(\A)\ra\R$ by
\begin{equation*}
\mu(\bs d)=\frac{\sum_{v\in Q_0}\bs d(v)\mu_v}{\sum_{v\in Q_0}\bs d(v)}.
\end{equation*}
Then \cite[Example~4.14]{Joyc6} shows that $(\mu,\R,\le)$ is a permissible stability condition on $\modCQ$, called {\it slope stability}. We call $\mu$ a {\it slope function}.
\end{ex}

\begin{ex}\label{wc3ex3}
Let $X$ be a smooth projective $\C$-scheme of dimension $m$, and $\A=\coh(X)$ the
abelian category of coherent sheaves on $X$. Let $K(\A)=K^\num(\coh(X))$ be the numerical Grothendieck group. Define $G$ to be the set of monic rational polynomials in $t$ of degree at most~$m$:
\begin{equation*}
G=\bigl\{p(t)=t^d+a_{d-1}t^{d-1}+\cdots+a_0\colon d=0,1,\dots,m,\,
a_0,\dots,a_{d-1}\in\Q\bigr\}.
\end{equation*}
Define a total order `$\le$' on $G$ by $p\le p'$ for $p,p'\in G$ if
either
\begin{itemize}\itemsep=0pt
\item[(a)] $\deg p>\deg p'$, or
\item[(b)] $\deg p=\deg p'$ and $p(t)\le p'(t)$ for all $t\gg 0$.
\end{itemize}
We write $p<q$ if $p\le q$ and $p\ne q$. Note that $\deg p>\deg p'$
in (a) implies that $p(t)>p'(t)$ for all $t\gg 0$, which is the
opposite to $p(t)\le p'(t)$ for $t\gg 0$ in (b).

Fix a very ample line bundle $\O_X(1)$ on $X$. For $E\in\coh(X)$, the {\it Hilbert polynomial} $P_E$ is the unique polynomial in $\Q[t]$ such that $P_E(n)=\dim H^0(E(n))$ for all $n\gg 0$. Equivalently, $P_E(n)=\bar\chi\bigl([\O_X(-n)],[E]\bigr)$ for all $n\in\Z$. Thus, $P_E$ depends only on the class $\al\in
K^\num(\coh(X))$ of $E$, and we may write $P_\al$ instead of $P_E$. Define $\tau\colon C(\coh(X))\ra G$ by $\tau(\al)=P_\al/r_\al$, where~$P_\al$ is the Hilbert polynomial of $\al$, and $r_\al$ is the
leading coefficient of $P_\al$, which must be positive. Then as in
\cite[Example~4.16]{Joyc6}, $(\tau,G,\le)$ is a permissible stability condition on
$\coh(X)$, called {\it Gieseker stability}. Gieseker stability
is studied in~\cite[Section~1.2]{HuLe}.
\end{ex}

\subsection[\texorpdfstring{Constructible and stack functions $\ep_\al(\tau)$, $\bar\ep_\al(\tau)$}{Constructible and stack functions εₐ(τ),‾εₐ(τ)}]{Constructible and stack functions $\boldsymbol{\ep_\al(\tau)}$, $\boldsymbol{\bar\ep_\al(\tau)}$}\label{wc34}

\begin{dfn}\label{wc3def4}
Suppose Assumption \ref{wc3ass1} holds, and let $(\tau,T,\le)$ be a permissible weak stability condition on $\A$. Then for each $\al\in C(\A)$ we have a constructible set $\M_\al^\ss(\tau)\subset\M(\C)$, so as in Section~\ref{wc31} we have $\de_{\M_\al^\ss(\tau)}\in\CF(\M)$ and $\bar\de_{\M_\al^\ss(\tau)}\in\SF(\M)$, where in fact $\bar\de_{\M_\al^\ss(\tau)}\in\SFa(\M)$. Following \cite[Definitions~7.6 and 8.1]{Joyc6}, define elements $\ep^\al(\tau)\in\CF(\M)$ and $\bar\ep{}^\al(\tau)\in\SFa(\M)$ for $\al\in C(\A)$ by
\begin{gather}
\ep^\al(\tau)=
\sum_{\substack{n\ge 1,\,\al_1,\dots,\al_n\in C(\A)\colon\\
\al_1+\cdots+\al_n=\al,\, \tau(\al_i)=\tau(\al),\text{ all
$i$}}}
\frac{(-1)^{n-1}}{n} \de_{\M_{\al_1}^\ss(\tau)}*\de_{\M_{\al_2}^\ss(\tau)}
*\cdots*\de_{\M_{\al_n}^\ss(\tau)},
\label{wc3eq7}\\
\bar\ep{}^\al(\tau)=
\sum_{\substack{n\ge 1,\,\al_1,\dots,\al_n\in C(\A)\colon \\
\al_1+\cdots+\al_n=\al,\, \tau(\al_i)=\tau(\al),\text{ all
$i$}}}
\frac{(-1)^{n-1}}{n} \bar\de_{\M_{\al_1}^\ss(\tau)}*\bar\de_{\M_{\al_2}^\ss(\tau)}
*\cdots*\bar\de_{\M_{\al_n}^\ss(\tau)},
\label{wc3eq8}
\end{gather}
where $*$ is the Ringel--Hall multiplication, and there are only finitely many nonzero terms in \eq{wc3eq7}--\eq{wc3eq8} by  \cite[Definition~7.6]{Joyc6}. Then \cite[Theorems~7.7 and~8.2]{Joyc6} prove
\begin{gather}
\de_{\M_\al^\ss(\tau)} =
\sum_{\substack{n\ge 1,\,\al_1,\dots,\al_n\in C(\A)\colon\\
\al_1+\cdots+\al_n=\al,\, \tau(\al_i)=\tau(\al),\text{ all
$i$}}}
\frac{1}{n!} \ep^{\al_1}(\tau)*\ep^{\al_2}(\tau)*
\cdots*\ep^{\al_n}(\tau),\label{wc3eq9}\\
\bar\de_{\M_\al^\ss(\tau)} =
\sum_{\substack{n\ge 1,\,\al_1,\dots,\al_n\in C(\A)\colon \\
\al_1+\cdots+\al_n=\al,\, \tau(\al_i)=\tau(\al),\text{ all
$i$}}}
\frac{1}{n!} \bar\ep{}^{\al_1}(\tau)*\bar\ep{}^{\al_2}(\tau)*
\cdots*\bar\ep{}^{\al_n}(\tau),
\label{wc3eq10}
\end{gather}
with only finitely many nonzero terms in \eq{wc3eq9}--\eq{wc3eq10}. Formally, for each $t\in T$ we may rewrite~\eq{wc3eq7} and~\eq{wc3eq9} as
\begin{gather*}
\sum_{\al\in C(\A)\colon \tau(\al)=t}\ep^\al(\tau)=\log\Biggl[1+ \sum_{\al\in C(\A)\colon \tau(\al)=t}\de_{\M_\al^\ss(\tau)}\Biggr],\\
 1+ \sum_{\al\in C(\A)\colon \tau(\al)=t}\de_{\M_\al^\ss(\tau)}=\exp\Biggl[\sum_{\al\in C(\A)\colon \tau(\al)=t}\ep^\al(\tau)\Biggr].
\end{gather*}
\end{dfn}

Then \cite[Theorems~7.8 and 8.7]{Joyc6} prove:

\begin{thm}\label{wc3thm2}
In Definition {\rm\ref{wc3def4}}, $\ep^\al(\tau)$ lies in the Lie subalgebra $\CFi(\M)\!\subset\CF(\M)$ from Section~{\rm \ref{wc32}}, and $\bar\ep{}^\al(\tau)$ lies in the Lie subalgebra $\SFai(\M)\subset\SFa(\M)$.
\end{thm}

Here is a way to think about the point of all this:
\begin{itemize}\itemsep=0pt
\item Knowing the $\de_{\M_\al^\ss(\tau)}$, $\bar\de_{\M_\al^\ss(\tau)}$ for $\al\in C(\A)$ is equivalent to knowing the subsets of $\tau$-semistable objects $\M_\al^\ss(\tau)\subset\M(\C)$.
\item By \eq{wc3eq7}--\eq{wc3eq10}, knowing the $\de_{\M_\al^\ss(\tau)}$, $\bar\de_{\M_\al^\ss(\tau)}$ for all $\al\in C(\A)$ is equivalent to knowing the $\ep^\al(\tau)$, $\bar\ep{}^\al(\tau)$ for all $\al\in C(\A)$.
\item The $\de_{\M_\al^\ss(\tau)}$, $\bar\de_{\M_\al^\ss(\tau)}$ tend to satisfy identities in the algebras $(\CF(\M),*)$ and $(\SFa(\M),*)$, but the $\ep^\al(\tau)$, $\bar\ep{}^\al(\tau)$ satisfy corresponding identities in the {\it Lie algebras} $(\CFi(\M),[\,,\,])$ and $(\SFai(\M),[\,,\,])$.
\item Thus, working with the $\ep^\al(\tau)$, $\bar\ep{}^\al(\tau)$ may be useful if we want to use a construction which works for Lie algebras, but not for associative algebras.

This occurs in Donaldson--Thomas theory of Calabi--Yau 3-folds \cite[Section~5.3]{JoSo}, where there is a Lie algebra morphism from $\SFai(\M)$ which does not extend to $\SFa(\M)$. It also occurred in Section~\ref{wc23}, when we had a Lie algebra structure on $\check H_*\big(\M^\pl\big)$, but no corresponding associative algebra.
\end{itemize}

\subsection{Wall-crossing under change of stability condition}\label{wc35}

In \cite{Joyc7} the second author proved transformation laws for the
$\de_{\M_\al^\ss(\tau)}$, $\bar\de_{\M_\al^\ss(\tau)}$ and $\ep^\al(\tau)$, $\bar\ep{}^\al(\tau)$ under change of stability condition. These involve combinatorial coefficients
$S(*;\tau,\ti\tau)\in\Z$ and $U(*;\tau, \ti\tau)\in\Q$ defined in
\cite[Section~4.1]{Joyc7}. Following \cite[Section~3.3]{JoSo}, we have changed some notation
from~\cite{Joyc7}.

\begin{dfn}\label{wc3def5}Suppose Assumption \ref{wc3ass1} holds, and let
$(\tau,T,\le)$, $\big(\ti\tau,\ti T,{\le}\big)$ be weak stability conditions on~$\A$.

Let $n\ge 1$ and $\al_1,\dots,\al_n\in C(\A)$. If for all
$i=1,\dots,n-1$ we have either
\begin{itemize}
\itemsep=0pt
\setlength{\parsep}{0pt}
\item[(a)] $\tau(\al_i)\le\tau(\al_{i+1})$ and
$\ti\tau(\al_1+\cdots+\al_i)>\ti\tau(\al_{i+1}+\cdots+\al_n)$, or
\item[(b)] $\tau(\al_i)>\tau(\al_{i+1})$ and~$\ti\tau(\al_1+\cdots+\al_i)\le\ti\tau(\al_{i+1}+\cdots+\al_n)$,
\end{itemize}
then define $S(\al_1,\dots,\al_n;\tau,\ti\tau)=(-1)^r$,
 where $r$ is the number
of $i=1,\dots,n-1$ satisfying (a). Otherwise define
$S(\al_1,\dots,\al_n;\tau,\ti\tau)=0$. Now
define
\begin{gather}
 U(\al_1,\dots,\al_n;\tau,\ti\tau)=
\label{wc3eq11}\\
 \sum_{\substack{\phantom{wiggle}\\
1\le l\le m\le n,\, 0=a_0<a_1<\cdots<a_m=n,\,
0=b_0<b_1<\cdots<b_l=m\colon \\
\text{Define $\be_1,\dots,\be_m\in C(\A)$ by
$\be_i=\al_{a_{i-1}+1}+\cdots+\al_{a_i}$}.\\
\text{Define $\ga_1,\dots,\ga_l\in C(\A)$ by
$\ga_i=\be_{b_{i-1}+1}+\cdots+\be_{b_i}$}.\\
\text{We require $\tau(\be_i)=\tau(\al_j)$, $i=1,\dots,m$,
$a_{i-1}<j\le a_i$},\\
\text{and $\ti\tau(\ga_i)=\ti\tau(\al_1+\cdots+\al_n)$,
$i=1,\dots,l$}}
\!\!\!\!\!\!\!\!\!\!\!\!\!\!\!\!\!\!\!\!\!\!\!\!\!\!\!\!\!\!\!\!\!
\!\!\!\!\!\!\!\!\!\!\!\!\!\!\!\!\!\!\!}
\begin{aligned}[t]
\frac{(-1)^{l-1}}{l} \prod\limits_{i=1}^lS(\be_{b_{i-1}+1},
\be_{b_{i-1}+2},\dots,\be_{b_i}; \tau,\ti\tau)\prod_{i=1}^m\frac{1}{(a_i-a_{i-1})!}& .
\end{aligned}
\nonumber
\end{gather}
\end{dfn}

Here are some properties of the coefficients $U(-)$. Equations \eq{wc3eq12}--\eq{wc3eq13} come from \cite[Theorem~4.8]{Joyc7}, and \eq{wc3eq14} follows from \eq{wc3eq11} and \cite[equation~(60)]{Joyc7}.

\begin{thm}\label{wc3thm3}
Let Assumption {\rm \ref{wc3ass1} hold}, and
$(\tau,T,\le)$, $\big(\hat\tau,\hat T,{\le}\big)$, $\big(\ti\tau,\ti T,{\le}\big)$ be weak stability conditions on~$\A$. Then for all $\al_1,\dots,\al_n\in C(\A)$ we have
\begin{gather}
U(\al_1,\dots,\al_n;\tau,\tau)
=\begin{cases} 1, & n=1, \\
0, & \text{otherwise}, \end{cases}
\label{wc3eq12}
\\
\begin{aligned}
\sum_{\substack{m,\, a_0,\dots,a_m\colon m=1,\dots,n,\\ 0=a_0<a_1<\cdots<a_m=n, \\
\text{set $\be_k=\al_{a_{k-1}+1}+\cdots+\al_{a_k}$}, \\
k=1,\dots,m}}\,
\begin{aligned}[t]
&U(\be_1,\dots,\be_m;\hat\tau,\ti\tau) \prod_{k=1}^mU (\al_{a_{k-1}+1},\al_{a_{k-1}+2},\dots,\al_{a_k};\tau,\hat\tau)\\
&\qquad =U(\al_1,\dots,\al_n;\tau,\ti\tau).
\end{aligned}
\end{aligned}
\label{wc3eq13}
\end{gather}

If also $\big(\ti\tau,\ti T,{\le}\big)$ dominates $(\tau,T,\le)$, as in Definition~{\rm \ref{wc3def3}}, then
\e
U(\al_1,\dots,\al_n;\tau,\ti\tau) = U(\al_1,\dots,\al_n;\ti\tau,\tau) = 0\qquad \text{unless $\ti\tau(\al_1) = \cdots = \ti\tau(\al_n)$}.
\label{wc3eq14}
\e
\end{thm}

Then in \cite[Theorem~5.2]{Joyc7} the second author derives wall-crossing formulae under change of stability condition from $(\tau,T,\le)$ to~$\big(\ti\tau,\ti T,{\le}\big)$:

\begin{thm}\label{wc3thm4}
Let Assumption {\rm \ref{wc3ass1}} hold, and $(\tau,T,\le)$,
$\big(\ti\tau,\ti T,{\le}\big)$ be permissible weak stability conditions on $\A$. Suppose also that there exists a permissible weak stability condition $\big(\hat\tau,\hat T,{\le}\big)$ on $\A$ which dominates both $(\tau,T,\le)$, $\big(\ti\tau,\ti T,{\le}\big)$ in the sense of Definition {\rm \ref{wc3def3}}. Then for all $\al\in C(\A)$ we have
\begin{gather}
\de_{\M_\al^\ss(\ti\tau)} =
\sum_{\substack{n\ge 1,\,\al_1,\dots,\al_n\in
C(\A)\colon \\ \al_1+\cdots+\al_n=\al}}
S(\al_1,\dots,\al_n;\tau,\ti\tau)\cdot\de_{\M_{\al_1}^\ss(\tau)}*\de_{\M_{\al_2}^\ss(\tau)}
*\cdots*\de_{\M_{\al_n}^\ss(\tau)},
\label{wc3eq15}\\
\bar\de_{\M_\al^\ss(\ti\tau)} =
\sum_{\substack{n\ge 1,\,\al_1,\dots,\al_n\in
C(\A)\colon \\ \al_1+\cdots+\al_n=\al }}
S(\al_1,\dots,\al_n;\tau,\ti\tau)\cdot
\bar\de_{\M_{\al_1}^\ss(\tau)}*\bar\de_{\M_{\al_2}^\ss(\tau)}
*\cdots*\bar\de_{\M_{\al_n}^\ss(\tau)},
\label{wc3eq16}\\
\ep^\al(\ti\tau) =
\sum_{\substack{n\ge 1,\,\al_1,\dots,\al_n\in
C(\A)\colon \\ \al_1+\cdots+\al_n=\al}}
U(\al_1,\dots,\al_n;\tau,\ti\tau)\cdot\ep^{\al_1}(\tau)*\ep^{\al_2}(\tau)*\cdots* \ep^{\al_n}(\tau),
\label{wc3eq17}\\
\bar\ep{}^\al(\ti\tau) =
\sum_{\substack{n\ge 1,\,\al_1,\dots,\al_n\in
C(\A)\colon\\ \al_1+\cdots+\al_n=\al }}
U(\al_1,\dots,\al_n;\tau,\ti\tau)\cdot\bar\ep{}^{\al_1}(\tau)*\bar\ep{}^{\al_2}(\tau)*\cdots* \bar\ep{}^{\al_n}(\tau),
\label{wc3eq18}
\end{gather}
where there are only finitely many nonzero terms in
\eqref{wc3eq15}--\eqref{wc3eq18}.
\end{thm}

Here the third weak stability condition $\big(\hat\tau,\hat T,{\le}\big)$ does not enter \eq{wc3eq15}--\eq{wc3eq18}, but is used to prove that there are only finitely many nonzero terms. For the case of quivers in Sections~\mbox{\ref{wc5}--\ref{wc6}}, we may take $\big(\hat\tau,\hat T,{\le}\big)$ to be the trivial stability condition $(\tau_{\rm tr},*,\le)$ in Example~\ref{wc3ex1}, which dominates any~$(\tau,T,\le)$, $\big(\ti\tau,\ti T,{\le}\big)$.

Theorem \ref{wc3thm3} implies that using \eq{wc3eq17}--\eq{wc3eq18} to transform from $\ep^*(\tau)$, $\bar\ep{}^*(\tau)$ to $\ep^*(\hat\tau)$, $\bar\ep{}^*(\hat\tau)$, and then to transform from $\ep^*(\hat\tau)$, $\bar\ep{}^*(\hat\tau)$ to $\ep^*(\ti\tau)$, $\bar\ep{}^*(\ti\tau)$, is equivalent to transforming from $\ep^*(\tau)$, $\bar\ep{}^*(\tau)$ to $\ep^*(\ti\tau)$, $\bar\ep{}^*(\ti\tau)$, as you would expect.

The next result is proved in \cite[Theorem~5.4]{Joyc7}. We have no explicit definition for $\ti U(\al_1,\dots,\al_n;\allowbreak \tau,\ti\tau)$, we only show that \eq{wc3eq18} can be rewritten in the form~\eq{wc3eq19}.

\begin{thm}\label{wc3thm5}
In Theorem {\rm \ref{wc3thm4}}, equations \eqref{wc3eq17}--\eqref{wc3eq18} may be rewritten as equations in the Lie algebras $\CFi(\M)$, $\SFai(\M)$ using the Lie brackets $[\,,\,]$, rather than as equations in $\CF(\M)$, $\SFa(\M)$ using the Ringel--Hall product~$*$. That is, we
may rewrite~\eqref{wc3eq18}, and similarly~\eqref{wc3eq17}, in the form
\e
\bar\ep{}^\al(\ti\tau)=
\sum_{\substack{n\ge 1,\,\al_1,\dots,\al_n\in
C(\A)\colon \\ \al_1+\cdots+\al_n=\al}\!\!\!\!\!\!\!\!\!\!\!\!\!\!\!\!\!\!}
\ti U(\al_1,\dots,\al_n;\tau,\ti\tau)\,\cdot[[\cdots[[\bar\ep{}^{\al_1}(\tau),\bar\ep{}^{\al_2}(\tau)],\bar\ep{}^{\al_3}(\tau)],
\dots],\bar\ep{}^{\al_n}(\tau)],
\label{wc3eq19}
\e
for some system of combinatorial coefficients $\ti U(\al_1,\dots,\al_n;\tau,\ti\tau)\in\Q$, with only finitely many nonzero terms, such that if we expand $[f,g]=f*g-g*f$ then \eq{wc3eq19} becomes \eq{wc3eq18}.

Alternatively, we may interpret {\rm\eq{wc3eq17}--\eq{wc3eq18}} as holding in the universal enveloping algebras~$U(\CFi(\M))$, $U(\SFai(\M))$.
\end{thm}

It will be very important later that \eq{wc3eq17}--\eq{wc3eq19} are {\it universal wall-crossing formulae in a~Lie algebra}. So we can make sense of the same wall-crossing formulae in the Lie algebras~$\check H_0(\M)$ of~Section~\ref{wc24}.

These results are applied in \cite[Sections~6.4--6.5]{Joyc7} and \cite{JoSo} in the following way. Suppose we can define a Lie algebra morphism $\Psi\colon \SFai(\M)\ra L_{K(\A)}$, where $L_{K(\A)}$ is an explicit Lie algebra, often of the form $\an{\la^\al\colon \al\in K(\A)}_R$ for some commutative ring $R$, with Lie bracket $[\la^\al,\la^\be]=c_{\al,\be}\la^{\al+\be}$ for coefficients $c_{\al,\be}$ in $R$. Then we may define motivic invariants $J^\al(\tau)\in R$ by $\Psi(\bar\ep{}^\al(\tau))=J^\al(\tau)\la^\al$. Applying $\Psi$ to \eq{wc3eq18}, interpreted using Theorem~\ref{wc3thm5}, then gives a~wall-crossing formula for the invariants $J^\al(\tau)$. This is used in \cite[Theorem~3.14]{JoSo} to prove a~wall-crossing formula for Donaldson--Thomas invariants of Calabi--Yau 3-folds.

\section[A conjectural picture of universal structures for enumerative invariant theories]{A conjectural picture of universal structures\\ for enumerative invariant theories}\label{wc4}

\subsection{General statement of the conjecture}\label{wc41}

We first describe our conjectural picture, in a way that tries to unify several rather different contexts in algebraic geometry and differential geometry. Sections \ref{wc43}--\ref{wc46} will explain these contexts. Our initial statement will be imprecise, with details added in Sections~\ref{wc43}--\ref{wc46}. We write it as Assumption~\ref{wc4ass1} followed by Conjecture~\ref{wc4conj1}. The assumption covers material which is basically already known, and we can provide explicit definitions and proofs for (much of them in~\cite{Joyc12}).

\begin{ass}\label{wc4ass1}
Let $\A$ be a $\C$-linear additive category. Some examples we have in mind are, in algebraic geometry, $\A=\coh(X)$ or $\A=D^b\coh(X)$ for $X$ a smooth projective $\C$-scheme; in representation theory, $\A=\modCQ$ or $\A=D^b\modCQ$ for $Q$ a quiver without oriented cycles; and in differential geometry, $\A$ could be the category of pairs $(E,\nabla_E)$ for $E\ra X$ a unitary complex vector bundle on a fixed compact manifold $X$, and $\nabla_E$ a connection on $E$. Assume:
\begin{itemize}\itemsep=0pt
\item[(a)] We can form two natural moduli spaces $\M$ and $\M^\pl$ of objects $E$ in $\A$, where the usual moduli space $\M$ parametrizes such $E$ up to isomorphism, and the `projective linear' moduli space $\M^\pl$ parametrizes $E$  up to isomorphisms modulo $\bG_m\cdot\id_E$. There is a morphism $\Pi\colon \M\ra\M^\pl$.

Here $\M$, $\M^\pl$ may be Artin $\C$-stacks, or higher $\C$-stacks, or derived $\C$-stacks, or topological stacks, or topological spaces up to homotopy equivalence, depending on the context.

There is a morphism $\Phi\colon \M\t\M\ra\M$ mapping~$([E],[F])\mapsto[E\op F]$.
\item[(b)] We are given a surjective quotient $K_0(\A)\twoheadrightarrow K(\A)$ of the Grothendieck group~$K_0(\A)$ of~$\A$. We write $\lb E\rb\in K(\A)$ for the class of $E\in\A$. There should be splittings $\M=\coprod_{\al\in K(\A)}\M_\al$, $\M^\pl=\coprod_{\al\in K(\A)}\M_\al^\pl$ such that $\M_\al$, $\M_\al^\pl$ parametrize objects in class $\al\in K(\A)$. We write $\Phi_{\al,\be}=\Phi\vert_{\M_\al\t\M_\be}\colon \M_\al\t\M_\be\ra\M_{\al+\be}$.
\item[(c)] We are given a symmetric biadditive form $\chi\colon K(\A)\t K(\A)\ra \Z$.
\item[(d)] We are given some natural additional geometric structures $\cG$, $\cG^\pl$ on $\M$, $\ab\M^\pl$. In the case of abelian categories in algebraic geometry, for $\cG$ we mean the data $\Psi$, $\Th^\bu$ in Assumption~\ref{wc2ass1}.

Using $\cG$, $\cG^\pl$, there are notions of {\it real virtual dimension} of $\M$, $\M^\pl$, with $\vdim_\R\M_\al=-\chi(\al,\al)$ and $\vdim_\R\M_\al^\pl=2-\chi(\al,\al)$ for $\al\in K(\A)$.

\item[(e)] There is a notion of {\it orientation} on $\M$, $\M^\pl$, defined using the geometric structures $\G$, $\G^\pl$. In some contexts the complex geometry induces a natural orientation, just as a complex manifold has a natural orientation considered as a real manifold, so issues of orientations can largely be ignored. In other contexts there is no natural choice. The morphism $\Pi\colon \M\ra\M^\pl$ identifies orientations on $\M$ and~$\M^\pl$.

We suppose that $\M$ is orientable (this can often be proved), and that an orientation has been chosen for $\M$, the natural one if this is defined. We write $o_\al$, $o_\al^\pl$ for the orientations on $\M_\al$, $\M_\al^\pl$.
\item[(f)] The morphism $\Phi\colon \M\t\M\ra\M$ has a natural relative orientation, so orientations on $\M$ pull back under $\Phi$ to orientations on $\M\t\M$. Using this, there are signs $\ep_{\al,\be}\in\{\pm 1\}$ for all $\al,\be\in K(\A)$ with $\M_\al, \M_\be\ne\es$ such that $o_\al\bt o_\be=\ep_{\al,\be}\cdot \Phi^*(o_{\al+\be})$. These $\ep_{\al,\be}$ satisfy \eq{wc2eq3}--\eq{wc2eq5}.
\item[(g)] Define $\hat H_*(\M)$ to be $H_*(\M)$ with grading shifted as in \eq{wc2eq13}. Then, as described in \cite{Joyc12}, and in Section~\ref{wc23} for abelian categories $\A$ in algebraic geometry, using the data $\cG$ and signs $\ep_{\al,\be}$ in (d),(f) we can make $\hat H_*(\M)$ into a {\it graded vertex algebra over}~$R$.
\item[(h)] Define $\check H_*\big(\M^\pl\big)$ to be $H_*\big(\M^\pl\big)$ with grading shifted as in \eq{wc2eq16}. Then, as described in~\cite{Joyc12}, and in Section~\ref{wc24} for abelian categories $\A$ in algebraic geometry, using the data~$\cG^\pl$ and signs $\ep_{\al,\be}$ in (d), (f) we can make $\check H_*\big(\M^\pl\big)$ into a~{\it graded Lie algebra over}~$R$. Thus $\check H_0\big(\M^\pl\big)$ is a Lie algebra.

Note that in many situations one can compute $\hat H_*(\M)$ and $\check H_*\big(\M^\pl\big)$ very explicitly (often~$\hat H_*(\M)$ is a~lattice vertex algebra), and they are not difficult to work with in examples.
\item[(i)] There is a notion of {\it stability condition} $\tau$ on $\A$, which induces notions of when objects $E\in\A$ are $\tau$-{\it stable} or $\tau$-{\it semistable}. For example:
\begin{itemize}\itemsep=0pt
\item[$\bu$] When $\A$ is an abelian category in algebraic geometry, we take $\tau$ to be a permissible weak stability condition as in Section~\ref{wc33}.
\item[$\bu$] When $\A$ is a triangulated category such as $D^b\coh(X)$ or $D^b\modCQ$, we may take~$\tau$ to be a suitable Bridgeland stability condition.
\item[$\bu$] When $X$ is a compact oriented 4-manifold with $b^2_+(X)=1$, and we wish to study Donaldson theory on~$X$, and~$\A$ is the category of unitary complex vector bundles $E\ra X$ with connections~$\nabla_E$, a stability condition is a Riemannian metric~$g$ on~$X$, which induces a splitting $H^2_{\rm dR}(X,\R)=\cH^2_+\op\cH^2_-$ of the de Rham 2-cohomology into harmonic self-dual and anti-self-dual forms, with $\cH^2_+=\an{\om_+}$. Then $(E,\nabla)$ is $\tau$-{\it semistable} if $F_{\nabla_E}^+=ic\cdot \id_E\ot\om_+$ for $c\in\R$. We call $(E,\nabla)$ $\tau$-{\it stable} if it is $\tau$-semistable and irreducible.
\end{itemize}
\item[(j)] If $\tau$ is a stability condition on $\A$, we form moduli spaces $\M_\al^\rst(\tau)\subseteq\M_\al^\ss(\tau)\subseteq\M_\al^\pl$ of $\tau$-(semi)stable objects $E\in\A$ in each class $\al\in K(\A)$.

Here $\M_\al^\ss(\tau)$ is compact (or has a compact `coarse moduli space'), and $\M_\al^\rst(\tau)$ has the structure of a `virtual manifold' (for example, $\M_\al^\rst(\tau)$ may be a smooth $\C$-scheme, or a~$\C$-scheme with perfect obstruction theory~\cite{BeFa}, or a smooth manifold, or a derived smooth manifold \cite{Joyc10,Joyc8,Joyc9,Joyc11}). The orientation on $\M_\al^\pl$ in (e) induces an orientation on $\M_\al^\rst(\tau)$.

Thus, if $\M_\al^\rst(\tau)=\M_\al^\ss(\tau)$ (that is, if there are no strictly $\tau$-semistables in class $\al$) then $\M_\al^\ss(\tau)$ is a compact, oriented virtual manifold of virtual dimension $\vdim_\R\M_\al^\ss(\tau)=2-\chi(\al,\al)$, so it has {\it virtual class} $[\M_\al^\ss(\tau)]_\virt$ in homology $H_{2-\chi(\al,\al)}\big(\M_\al^\pl\big)$, which is $\check H_0\big(\M_\al^\pl\big)$ by \eq{wc2eq16}.

In fact $[\M_\al^\ss(\tau)]_\virt$ may be defined in $\Z$-homology $H_{2-\chi(\al,\al)}(\M_\al,\Z)$.
\end{itemize}
\end{ass}

As in Sections~\ref{wc43}--\ref{wc46} and \cite{Joyc12,Joyc13}, there are many interesting enumerative invariant theories in which we can define all the data of Assumption~\ref{wc4ass1}.

The following conjecture should not be regarded as a precise statement, nor are we claiming that it should hold in every theory for which some version of Assumption \ref{wc4ass1} holds. It is intended as an outline of a general structure we expect to see in many enumerative invariant theories, which may need modification in particular cases. We make some more exact conjectures in Sections~\ref{wc43}--\ref{wc46}.

\begin{conj}
\label{wc4conj1}
Suppose Assumption \ref{wc4ass1} holds. Then:
\begin{itemize}
\itemsep=0pt
\setlength{\parsep}{0pt}
\item[(i)] For all $\al\in K(\A)$ we may define invariants $[\M_\al^\ss(\tau)]_\inv\in \check H_0\big(\M_\al^\pl\big)$. If $\M_\al^\rst(\tau)=\M_\al^\ss(\tau)$ then Assumption \ref{wc4ass1}(j) gives a virtual class $[\M_\al^\ss(\tau)]_\virt$ in $\check H_0\big(\M_\al^\pl\big)$, and then $[\M_\al^\ss(\tau)]_\inv\ab =[\M_\al^\ss(\tau)]_\virt$. It is crucial that $\check H_0\big(\M_\al^\pl\big)=H_{2-\chi(\al,\al)}\big(\M_\al^\pl\big)$ is homology over a $\Q$-{\it algebra} $R$, as $[\M_\al^\ss(\tau)]_\inv$ may exist in $\Q$-homology but not in $\Z$-homology.
\item[(ii)] Let $\tau$, $\ti\tau$ be stability conditions on $\A$ as in Assumption \ref{wc4ass1}(i). It may be necessary to impose a condition on $\tau$, $\ti\tau$ to ensure that \eq{wc4eq1}--\eq{wc4eq2} below have only finitely many nonzero terms~-- see the notion of `globally finite change of stability condition' in \cite[Definition~5.1]{Joyc7}. Roughly, this says that $\tau$, $\ti\tau$ are sufficiently close in the space of stability conditions on $\A$.

Then for all $\al\in K(\A)$ (or in a `positive cone' $C(\A)\subset K(\A)$), the analogue of \eq{wc3eq19} holds in the Lie algebra $\check H_0\big(\M^\pl\big)$ from Assumption~\ref{wc4ass1}(h):
\e
\begin{gathered}
{}
[\M_\al^\ss(\ti\tau)]_\inv=
\sum_{\substack{n\ge 1,\,\al_1,\dots,\al_n\in
C(\A)\colon \\ \al_1+\cdots+\al_n=\al}}
\begin{aligned}[t]
\ti U(\al_1,&\dots,\al_n;\tau,\ti\tau)\cdot\bigl[\bigl[\cdots\bigl[[\M_{\al_1}^\ss(\tau)]_\inv,\\
&
[\M_{\al_2}^\ss(\tau)]_\inv\bigr],\dots\bigr],[\M_{\al_n}^\ss(\tau)]_\inv\bigr].
\end{aligned}
\end{gathered}
\label{wc4eq1}
\e
 Equivalently, as in Theorem \ref{wc3thm5}, the analogue of \eq{wc3eq18} holds in the universal enveloping algebra $\bigl(U\big(\check H_0\big(\M^\pl\big)\big),*\bigr)$:
\e
\begin{gathered}
{}
[\M_\al^\ss(\ti\tau)]_\inv=
\sum_{\substack{n\ge 1,\,\al_1,\dots,\al_n\in
C(\A)\colon \\ \al_1+\cdots+\al_n=\al }}
\begin{aligned}[t]
U(\al_1,&\dots,\al_n;\tau,\ti\tau)\cdot[\M_{\al_1}^\ss(\tau)]_\inv *\\
&
[\M_{\al_2}^\ss(\tau)]_\inv*\cdots *[\M_{\al_n}^\ss(\tau)]_\inv.
\end{aligned}
\end{gathered}
\label{wc4eq2}
\e
We call \eq{wc4eq1}--\eq{wc4eq2} {\it wall-crossing formulae} for the invariants~$[\M_\al^\ss(\tau)]_\inv$.

We have only defined the $U(\al_1,\dots,\al_n;\tau,\ti\tau)$, $\ti U(\al_1,\dots,\al_n;\tau,\ti\tau)$ in Section~\ref{wc35} in the abelian category case, in which the `positive cone' $C(\A)$ makes sense. There should be a~way to extend the theory to triangulated categories such as $D^b\coh(X),D^b\modCQ$. Then we no longer have a good definition of `positive cone' $C(\A)\subset K(\A)$, so it is more difficult both to define the $U(\al_1,\dots,\al_n;\tau,\ti\tau)$, and to ensure that \eq{wc4eq1}--\eq{wc4eq2} have only finitely many nonzero terms. This extension was discussed in~\cite[Problem~7.1]{Joyc7}.
\item[(iii)] We currently have no {\it direct} definition of the invariants $[\M_\al^\ss(\tau)]_\inv$ in~(i) when $\M_\al^\rst(\tau)\ne\M_\al^\ss(\tau)$, starting from the moduli space $\M_\al^\ss(\tau)\subseteq\M_\al^\pl$ and the geometric structures $\G_\al^\pl$,~$o_\al^\pl$ upon it. However, in many cases there is an {\it indirect} definition, by making use of the wall-crossing formulae \eq{wc4eq1}--\eq{wc4eq2} in an auxiliary category $\B$. We call this the {\it method of pair invariants}, and it is used in \cite[Sections~5.4 and~13.1]{JoSo} to define Donaldson--Thomas invariants of Calabi--Yau 3-folds. (These are used by Tanaka--Thomas \cite{TaTh2} to study Vafa--Witten invariants. See also Mochizuki's use of `$L$-Bradlow pairs' in \cite[Section~7.3]{Moch}.) When it applies, it shows that classes~$[\M_\al^\ss(\tau)]_\inv$ satisfying (i),(ii) must be unique. We explain it in the next definition.
\end{itemize}	
\end{conj}

\begin{dfn}\label{wc4def1}
Let Assumption \ref{wc4ass1} and Conjecture~\ref{wc4conj1}(i),~(ii) hold for~$\A$. For clarity, suppose also that $\A$ is a $\C$-linear abelian category in algebraic geometry as in Assumption~\ref{wc3def1}, and that `stability conditions' $\tau$ in Assumption~\ref{wc4ass1}(i) mean permissible weak stability conditions $(\tau,T,\le)$ in~Section~\ref{wc33}.

To use the {\it method of pair invariants}, we should construct a second abelian category $\B$ satisfying Assumption \ref{wc4ass1} and Conjecture~\ref{wc4conj1}(i),~(ii), such that:
\begin{itemize}\itemsep=0pt
\item[(i)] There is a $\C$-linear inclusion $i\colon \A\hookra\B$ as a full abelian subcategory.
\item[(ii)] There is a distinguished object $I\in\B$, with $\Hom_\B(I,I)=\C$. Every object $B\in\B$ fits into an exact sequence in $\B$, unique up to isomorphism:
\e
\xymatrix@C=30pt{ 0 \ar[r] & i(A) \ar[r] & B \ar[r] & V\ot_\C I \ar[r] & 0, }
\label{wc4eq3}
\e
where $A\in\A$ and $V$ is a finite-dimensional $\C$-vector space, so that $V\ot_\C I\cong {\buildrel
{\ulcorner\quad\text{$n$ copies }\quad\urcorner} \over
{I\op I\op\cdots\op I}}$ if $\dim_\C V=n$.
\item[(iii)] There is an identification $K(\B)=K(\A)\op\Z$ with $\lb B\rb=\bigl(\lb A\rb,\dim_\C V\bigr)$ in \eq{wc4eq3}. Then $C(\B)=\bigl((C(\A)\amalg\{0\})\t\N\bigr)\sm\{(0,0)\}$.
\item[(iv)] Writing $\chi$, $\ti\chi$ for $\chi$ in Assumption \ref{wc4ass1}(c) for $\A,\B$, we have
$\chi(\al,\be)=\ti\chi((\al,0),(\be,0))$ for $\al,\be\in K(\A)$ and $\ti\chi((0,m),(0,n))=2mn$ for $m,n\in\Z$.
\item[(v)] Write $\M$, $\M^\pl$ for the moduli spaces for $\A$, and $\tiM$, $\tiM^\pl$ for the moduli spaces for $\B$. Then the inclusion $i\colon \A\hookra\B$ in (i) induces morphisms $i\colon \M\ra\tiM$, $i^\pl\colon \M^\pl\ra\tiM^\pl$, which are isomorphisms $i\colon \M_\al\ra\tiM_{(\al,0)}$ and $i^\pl\colon \M_\al^\pl\ra\tiM_{(\al,0)}^\pl$ for $\al\in K(\A)$. These isomorphisms identify the extra data $\G_\al$, $\G_\al^\pl$ and $\ti\G{}_{(\al,0)}$, $\ti\G{}_{(\al,0)}^\pl$ in Assumption~\ref{wc4ass1}(d) for~$\A$,~$\B$, and the orientations $o_\al$, $o_\al^\pl$ and $\ti o{}_{(\al,0)}$, $\ti o{}_{(\al,0)}^\pl$ in Assumption~\ref{wc4ass1}(e). The signs $\ep_{\al,\be}$,~$\ti\ep_{\ti\al,\ti\be}$ in Assumption~\ref{wc4ass1}(f) satisfy~$\ep_{\al,\be}=\ti\ep_{(\al,0),(\be,0)}$.
\item[(vi)] $i_*\colon \hat H_*(\M)\ra\hat H_*\big(\tiM\big)$ and $i_*^\pl\colon \check H_*\big(\M^\pl\big)\ra\check H_*\big(\tiM^\pl\big)$ should be morphisms of the graded vertex and Lie algebras in Assumption~\ref{wc4ass1}(g),~(h).
\end{itemize}	
For example, as in \cite[Section~5.4]{JoSo}, if $\A=\coh(X)$ for $X$ a smooth projective $\C$-scheme, we may define $\B$ to be the category of morphisms $\phi\colon V\ot_\C\O_X(n)\ra A$ for $n\in\Z$ fixed, and $V$ a~finite-dimensional $\C$-vector space, and~$A\in\coh(X)$.

Now suppose $(\tau,T,\le)$ is a permissible weak stability condition on $\A$. Define permissible weak stability conditions $(\tau_+,T_+,\le)$ and $(\tau_-,T_-,\le)$ on $\B$, such that $T_+=(T\t\{0,1\})\amalg\{+\iy\}$ and $T_-=(T\t\{-1,0\})\amalg\{-\iy\}$, where:
\begin{itemize}\itemsep=0pt
\item The order $\le$ on $T_+$ is given by $(t_1,n_1)\le(t_2,n_2)$ for $(t_1,n_1)$, $(t_2,n_2)$ in $T\t\{0,1\}$ if either $t_1<t_2$ in $T$, or $t_1=t_2$ and $n_1\le n_2$, and $(t,n)<+\iy$ for all $(t,n)\in T\t\{0,1\}$.
\item The order $\le$ on $T_-$ is given by $(t_1,n_1)\le(t_2,n_2)$ for $(t_1,n_1)$, $(t_2,n_2)$ in $T\t\{-1,0\}$ if either $t_1<t_2$ in $T$, or $t_1=t_2$ and $n_1\le n_2$, and $-\iy<(t,n)$ for all $(t,n)\in T\t\{-1,0\}$.
\item If $\al\in C(\A)$ then $\tau_\pm(\al,0)=(\tau(\al),0)$ and $\tau_\pm(\al,n)=(\tau(\al),\pm 1)$ for $n>0$.
\item If $n>0$ then $\tau_\pm((0,n))=\pm\iy\in T_\pm$.
\end{itemize}
We should then be able to easily prove that:
\begin{itemize}\itemsep=0pt
\item[(a)] If $\al\in C(\A)$ then $i^\pl$ identifies $\M_\al^\ss(\tau)$ with $\tiM_{(\al,0)}^\ss(\tau_\pm)$. Hence $i_*^\pl$ maps $[\M_\al^\ss(\tau)]_\inv\mapsto\big[\tiM_{(\al,0)}^\ss(\tau_\pm)\big]_\inv$.
\item[(b)] $\tiM_{(0,1)}^\rst(\tau_\pm)=\tiM_{(0,1)}^\pl=\{[I]\}$ is a point, and $\big[\tiM_{(0,1)}^\ss(\tau_\pm)\big]_\inv=1_{H_0(\tiM_{(0,1)}^\pl)}$ in $\check H_0\big(\tiM_{(0,1)}^\pl\big)=H_0\big(\tiM_{(0,1)}^\pl\big)=R$.
\item[(c)] If $\al\in C(\A)$ then $\tiM_{(\al,1)}^\ss(\tau_-)=\es$, as if $B\in\B$ with $\lb B\rb=(\al,1)$ then \eq{wc4eq3} gives a~subobject $i(A)\subset B$ with $B/i(A)\cong I$, and $\tau_-(i(A))>-\iy=\tau_-(I)$, so $B$ is $\tau_-$-unstable by Definition \ref{wc3def3}. Hence $\big[\tiM_{(\al,1)}^\ss(\tau_-)\big]_\virt=0$.
\item[(d)] If $\al\in C(\A)$ then $\tiM_{(\al,1)}^\rst(\tau_+)=\tiM_{(\al,1)}^\ss(\tau_+)$, as if $B\in\B$ with $\lb B\rb=(\al,1)$ then there can exist no subobjects $0\ne B'\subsetneq B$ with $\tau_+(B')=\tau_+(B/B')$. Hence the virtual class $\big[\tiM_{(\al,1)}^\ss(\tau_+)\big]_\virt$ is defined in Assumption~\ref{wc4ass1}(j), without using Conjecture~\ref{wc4conj1}.
\end{itemize}

Now consider the wall-crossing formula \eq{wc4eq1} for $\B$ with $(\al,1)$, $\ab\tau_-$, $\ab\tau_+$ in place of $\al$, $\tau$, $\ti\tau$, for $\al\in C(\A)$. It turns out that this may be written
\begin{gather}
 \big[\tiM_{(\al,1)}^\ss(\tau_+)\big]_\virt=
\label{wc4eq4}\\
 \sum_{\substack{n\ge 1,\;\al_1,\dots,\al_n\in
C(\A)\colon \\ \al_1+\cdots+\al_n=\al, \\
\tau(\al_i)=\tau(\al),\; i=1,\dots,n}}\!
\frac{(-1)^n}{n!}\cdot\bigl[\bigl[\cdots\bigl[\big[\tiM_{(0,1)}^\ss(\tau_-)\big]_\inv,
\big[\tiM_{(\al_1,0)}^\ss(\tau_-)\big]_\inv\bigr],\dots\bigr],\big[\tiM_{(\al_n,0)}^\ss(\tau_-)\big]_\inv\bigr].
\nonumber
\end{gather}
This holds because using \cite[Proposition~13.8]{JoSo} one can show that
\begin{gather*}
 U\bigl((\al_1,0),\dots,(\al_{k-1},0),(0,1),(\al_k,0),\dots,(\al_n,0);\tau_-,\tau_+)\\
\qquad {} =\begin{cases} \dfrac{(-1)^{n-k}}{(k-1)!(n-k)!}, & \text{if $\tau(\al_i)=\tau(\al)$, $i=1,\dots,n$}, \\ 0 & \text{otherwise}, \end{cases}
\end{gather*}
and then equation \eq{wc4eq4} follows from \eq{wc4eq2} as in \cite[Proposition~13.10]{JoSo}.

Using (a),(b) above, let us rewrite \eq{wc4eq4} as
\e
\big[\tiM_{(\al,1)}^\ss(\tau_+)\big]_\virt=\bigl[i_*^\pl\bigl([\M_\al^\ss(\tau)]_\inv\bigr),1_{H_0(\tiM_{(0,1)}^\pl)}\bigr]+\text{lower order terms},
\label{wc4eq5}
\e
where the `lower order terms' are those with $n\ge 2$ in the sum in \eq{wc4eq4}. Our goal is to define $[\M_\al^\ss(\tau)]_\inv$ uniquely in the case when $\M_\al^\rst(\tau)\ne \M_\al^\ss(\tau)$. We do this by induction on $\al$ in some order in $C(\A)$ compatible with addition, for instance by induction on $\rank\al=1,2,\dots$ if this is defined. Then by induction we can suppose the `lower order terms' are uniquely defined. The left-hand side of \eq{wc4eq5} is determined by (d) above. Hence in the inductive step, $\bigl[i_*^\pl\bigl([\M_\al^\ss(\tau)]_\inv\bigr),1_{H_0(\tiM_{(0,1)}^\pl)}\bigr]$ is uniquely defined. If we have chosen $\B$ such that $\big[{-},1_{H_0(\tiM_{(0,1)}^\pl)}\big]$ is injective, then as $i_*^\pl$ is an isomorphism, this shows $[\M_\al^\ss(\tau)]_\inv$ is uniquely determined.

There is an alternative version of the method of pair invariants in which we replace \eq{wc4eq3} by the exact sequence
\begin{equation*}
\xymatrix@C=30pt{ 0 \ar[r] & V\ot_\C I \ar[r] & B \ar[r] & i(A) \ar[r] & 0, }
\end{equation*}
and exchange $(\tau_+,T_+,\le)$ and $(\tau_-,T_-,\le)$ in the argument above.
\end{dfn}

\begin{rem}
\label{wc4rem1}
In defining enumerative invariants in algebraic or differential geometry, there are usually three main difficulties:
\begin{itemize}\itemsep=0pt
\item[(a)] {\it transversality}, whether we can make the moduli spaces smooth.
\item[(b)] {\it compactness}, whether the moduli spaces are compact, or can be compactified by including singular solutions.
\item[(c)] {\it strictly semistable} or {\it reducible} points in the moduli spaces, which cause problems with the definition of virtual classes when~$\M_\al^\rst(\tau)\ne\M_\al^\ss(\tau)$.
\end{itemize}
An important aspect of Conjecture \ref{wc4conj1} is that it offers a new, universal, systematic approach to problem~(c).

With a proper understanding of obstruction theories \cite{BeFa}, derived algebraic geometry \cite{Toen1,Toen2,ToVa,ToVe1,ToVe2}, and derived differential geometry \cite{Joyc10,Joyc8,Joyc9,Joyc11}, part (a) is really not a~problem (though people still make a fuss about it): virtual classes are well defined and well behaved even when moduli spaces are not transverse.

For (b), in algebraic geometry we usually get compactness for free by considering moduli spaces of the right kind of objects (for example, torsion-free sheaves rather than vector bundles). In differential geometry, compactifying moduli spaces by singular solutions usually involves difficult analytic issues. We have nothing new to say about this, we will just assume it works.

A common approach to problem (c) is to avoid it, by only considering classes $\al\in K(\A)$ for which $\M_\al^\rst(\tau)=\M_\al^\ss(\tau)$, or by restricting to $\al$ for which there are very few strictly $\tau$-semistable points in $\M_\al^\ss(\tau)$, which can be understood and dealt with `by hand'. Some examples of this: Thomas \cite{Thom1} originally defined Donaldson--Thomas invariants of Calabi--Yau 3-folds only when $\M_\al^\rst(\tau)=\M_\al^\ss(\tau)$. In Donaldson theory of 4-manifolds \cite{DoKr}, almost all work focusses on $\SU(2)$- or $\SO(3)$-instantons, rather than $G$-instantons for other Lie groups $G$, and on simply-connected 4-manifolds. This is because reducibles for $\SU(2)$- or $\SO(3)$-instantons come from line bundles, and are easily understood. Similar comments apply to Casson invariants of 3-manifolds \cite{AkMc,Taub}, in which restricting to $\SU(2)$ and to homology 3-spheres is used to control reducibles.

However, if we wish to study wall-crossing for enumerative invariants, we cannot restrict to $\al$ with $\M_\al^\rst(\tau)=\M_\al^\ss(\tau)$, as the wall-crossing may involve other moduli spaces $\M_\be^\ss(\tau)$ with $\M_\be^\rst(\tau)\ne\M_\be^\ss(\tau)$. Conjecture \ref{wc4conj1} aims to solve problem (c) of defining invariants in the presence of strictly semistables/reducibles, and understanding their wall-crossing, simultaneously, and Conjecture \ref{wc4conj1}(iii) uses the wall-crossing formula to define the invariants.	
\end{rem}

\begin{rem}[on orientations]\label{wc4rem2}
In contexts in which we are free to choose the orientation~$o_\al^\pl$ on~$\M_\al^\pl$ in Assumption~\ref{wc4ass1}(e), we can change the sign of each invariant $[\M_\al^\ss(\tau)]_\inv$ arbitrarily by changing the sign of~$o_\al^\pl$.

This may seem to contradict the identities \eq{wc4eq1}, \eq{wc4eq2}, \eq{wc4eq4} which mix $[\M_\al^\ss(\tau)]_\inv$ for different~$\al$. However, changing the $o_\al^\pl$ changes the $\ep_{\al,\be}$ in Assumption~\ref{wc4ass1}(f), and the Lie bracket~$[\,,\,]$ on~$\check H_*\big(\M^\pl\big)$ and product $*$ on $U\big(\check H_*\big(\M^\pl\big)\big)$ used in \eq{wc4eq1}, \eq{wc4eq2}, \eq{wc4eq4} depend on the $\ep_{\al,\be}$. The combined effect of the sign changes on $[\M_\al^\ss(\tau)]_\inv$ and $[\,,\,],*$ cancels out.

As in \cite{Joyc12}, we can set the theory up in an orientation-independent way, by replacing $\check H_*\big(\M^\pl\big)$ by homology $\check H_*\big(\M^\pl,O^\pl\big)$ twisted by a principal $\Z_2$-bundle $O^\pl\ra\M^\pl$ of orientations on~$\M^\pl$, which is assumed to be trivializable (though not canonically trivial) in Assumption~\ref{wc4ass1}(e). Then $[\M_\al^\ss(\tau)]_\inv,[\,,\,],*$ exist in and on $\check H_*\big(\M^\pl,O^\pl\big)$ canonically, without having to choose orientations.
\end{rem}

\subsection{Rewriting the wall-crossing formula in the style of Kontsevich--Soibelman}\label{wc42}

The second author developed the wall-crossing story for motivic invariants in Section~\ref{wc3} and \cite{Joyc4,Joyc2,Joyc5,Joyc6,Joyc3,Joyc7} in 2003--2005, and applied it to Donaldson--Thomas invariants of Calabi--Yau 3-folds in 2008~\cite{JoSo}. Independently, Kontsevich and Soibelman \cite[Section~1.4]{KoSo} in 2008 wrote down their own wall-crossing formula for motivic Donaldson--Thomas invariants, which is equivalent to a special case of \eq{wc3eq15}--\eq{wc3eq16} in a suitable associative algebra. Kontsevich and Soibelman's version has proved more popular with subsequent authors, possibly because the coefficients $S(-)$, $U(-)$ in \eq{wc3eq15}--\eq{wc3eq18} are not easy to understand and compute.

We now explain how to rewrite our conjectured wall-crossing formulae \eq{wc4eq1}--\eq{wc4eq2} in the style of Kontsevich--Soibelman \cite[Section~1.4]{KoSo}. We work in the universal enveloping algebra $U\big(\check H_0\big(\M^\pl\big)\big)$, an associative $R$-algebra with product $*$. As $\check H_0\big(\M^\pl\big)\subset U\big(\check H_0\big(\M^\pl\big)\big)$, the invariants $[\M_\al^\ss(\tau)]_\inv$ in Conjecture~\ref{wc4conj1}(i) lie in $U\big(\check H_0\big(\M^\pl\big)\big)$, as we used in equation~\eq{wc4eq2}.

In the situation of Assumption \ref{wc4ass1} and Conjecture~\ref{wc4conj1}(i), in an analogue of \eq{wc3eq9}--\eq{wc3eq10}, define elements $\de^\ss_\al(\tau)\in U\big(\check H_0\big(\M^\pl\big)\big)$ for $\al\in C(\A)$ by
\e
\de^\ss_\al(\tau)=
\sum_{\substack{n\ge 1,\,\al_1,\dots,\al_n\in C(\A)\colon\\
\al_1+\cdots+\al_n=\al,\, \tau(\al_i)=\tau(\al),\text{ all
$i$}}}
\frac{1}{n!} [\M_{\al_1}^\ss(\tau)]_\inv*
\cdots*[\M_{\al_n}^\ss(\tau)]_\inv.
\label{wc4eq6}
\e
As for \eq{wc3eq7}--\eq{wc3eq8}, this equation can be inverted to give
\begin{equation*}
[\M_\al^\ss(\tau)]_\inv=
\sum_{\substack{n\ge 1,\,\al_1,\dots,\al_n\in C(\A)\colon\\
\al_1+\cdots+\al_n=\al,\; \tau(\al_i)=\tau(\al),\text{ all
$i$} }}
\frac{(-1)^{n-1}}{n} \de_{\al_1}^\ss(\tau)*\cdots*\de_{\al_n}^\ss(\tau).
\end{equation*}

Let $(\tau,\R,\le)$ and $(\ti\tau,\R,\le)$ be permissible slope stability conditions on $\A$ in the sense of Section~\ref{wc33}. (See Example~\ref{wc3ex2} and Section~\ref{wc52} on slope stability.) Suppose that $a<b$ in $\R$ are such that the quasi-abelian subcategories ${\mathcal A}_{[a,b]},\ti{\mathcal A}_{[a,b]}\subset\A$ generated by $\tau$- and $\ti\tau$-semistable objects $E,\ti E\in\A$ with $\tau(E),\ti\tau(\ti E)\in[a,b]$ satisfy ${\mathcal A}_{[a,b]}=\ti{\mathcal A}_{[a,b]}$. Then the Kontsevich--Soibelman style analogue of \eq{wc4eq2}~is
\begin{gather*}
1+\sum_{\substack{n\ge 1,\,\al_1,\dots,\al_n\in C(\A)\colon\\
a\le\tau(\al_1)<\tau(\al_2)<\cdots<\tau(\al_n)\le b } }\de_{\al_1}^\ss(\tau)*\cdots*\de_{\al_n}^\ss(\tau)\nonumber\\
\qquad{} =1+\sum_{\substack{n\ge 1,\,\al_1,\dots,\al_n\in C(\A)\colon\\
a\le\ti\tau(\al_1)<\ti\tau(\al_2)<\cdots<\ti\tau(\al_n)\le b} }\de_{\al_1}^\ss(\ti\tau)*\cdots*\de_{\al_n}^\ss(\ti\tau),
\label{wc4eq7}
\end{gather*}
when this makes sense (we do not claim that it always does). As the sums in \eq{wc4eq7} are infinite, they should be interpreted as lying in a completion $\bar U\big(\check H_0\big(\M^\pl\big)\big)$ of $U\big(\check H_0\big(\M^\pl\big)\big)$ with respect to a suitable ideal. Morally speaking, both sides of \eq{wc4eq7} count all objects in~${\mathcal A}_{[a,b]}=\ti{\mathcal A}_{[a,b]}$.

\begin{rem}\label{wc4rem3}
As in Zhu \cite{Zhu}, the graded vertex algebra $\hat H_*(\M)$ has a {\it Zhu algebra} $A\big(\hat H_*(\M)\big)=\hat H_*(\M)/I$, an associative $R$-algebra, with a natural morphism $U\big(\check H_0\big(\M^\pl\big)\big)\ra A\big(\hat H_*(\M)\big)$. It might be interesting to interpret \eq{wc4eq2} and \eq{wc4eq6}--\eq{wc4eq7} as equations in $A\big(\hat H_*(\M)\big)$ or its completion $\bar A\big(\hat H_*(\M)\big)$.
\end{rem}

\subsection{The conjecture in algebraic geometry and representation theory}\label{wc43}

We now give more details on the ideas of Section~\ref{wc41} for $\C$-linear enumerative invariant theories coming from algebraic geometry and representation theory. But we exclude Donaldson--Thomas type invariants of Calabi--Yau 4-folds \cite{BoJo,CaLe,OhTh}, which will be discussed in Section~\ref{wc44}. In the sequel \cite[Sections~6--8]{Joyc13} we prove Conjecture~\ref{wc4conj1} for most of the situations discussed in this section.

\subsubsection{General discussion for abelian categories}\label{wc431}

We first discuss the case of $\C$-linear abelian categories $\A$, such as $\coh(X)$ for $X$ a smooth projective $\C$-scheme, or compactly-supported coherent sheaves $\coh_\cs(X)$ for $X$ smooth and quasi-projective, or $\modCQ$ for $Q$ a quiver, or $\modCQI$ for $(Q,I)$ a quiver with relations.

Then we have {\it Ext groups} $\Ext^k(E,F)$ for $E,F\in\A$ and $k=0,1,\dots$, with $\Ext^0(E,F)=\Hom_\A(E,F)$, which are finite-dimensional $\C$-vector spaces. We say that the category $\A$ has {\it dimension} $\dim\A=m$ if $\Ext^k(E,F)=0$ for all $E$, $F$ and $k>m$, and $\Ext^m(E,F)\ne 0$ for some~$E$,~$F$. For example, if $X$ is a smooth (quasi-)projective $m$-fold then $\coh(X)$ (or $\coh_\cs(X)$) has dimension~$m$, and if~$Q$ is a quiver then $\modCQ$ has dimension~1 (or~0 if~$Q$ has no edges).

The {\it Euler form} of $\A$ is the biadditive map $\chi_\A\colon K_0(\A)\t K_0(\A)\ra\Z$ with
\begin{equation*}
\chi_\A([E],[F])=\sum_{k=0}^{\dim\A}(-1)^k\dim_\C\Ext^k(E,F).
\end{equation*}
It need not be symmetric. The {\it numerical Grothendieck group} is $K^\num(\A)=K_0(\A)/\Ker\chi_\A$. It is usually a good choice for $K(\A)$ in Assumption \ref{wc4ass1}(b).

In all these cases Assumption \ref{wc2ass1} applies, and the vertex algebra $\hat H_*(\M)$ and Lie algebra $\check H_*\big(\M^\pl\big)$ in Assumption~\ref{wc4ass1}(g),~(h) are constructed as in Sections~\ref{wc23}--\ref{wc24}. There is a~natural perfect complex $\cExt^\bu$ on $\M\t\M$ called the {\it Ext complex}, whose cohomology at a~$\C$-point $([E],[F])\in(\M\t\M)(\C)$ is $H^k(\cExt^\bu\vert_{([E],[F])})=\Ext^k(E,F)$, with $\rank\bigl(\cExt^\bu\vert_{\M_\al\t\M_\be}\bigr)=\chi_\A(\al,\be)$ for~$\al,\be\in K(\A)$.

As in \cite{Joyc12}, in Assumption \ref{wc2ass1}(g) we define $\Th^\bu=(\cExt^\bu)^\vee\op\si_\M^*(\cExt^\bu)[2n]$ for some $n\in\Z$, where $\si_\M\colon \M\t\M\ra\M\t\M$ swaps the factors. Then \eq{wc2eq6}--\eq{wc2eq10} hold. We set $\chi(\al,\be)=\chi_\A(\al,\be)+\chi_\A(\be,\al)$, so that $\chi$ is symmetric with $\rank\bigl(\Th^\bu\vert_{\M_\al\t\M_\be}\bigr)=\chi(\al,\be)$. In all these cases, there are canonical orientations on $\M$, $\M^\pl$ in Assumption \ref{wc4ass1}(e) coming from the complex geometry. For these orientations, the signs in Assumption~\ref{wc4ass1}(f) are~$\ep_{\al,\be}=(-1)^{\chi_\A(\al,\be)}$.

All the material of Assumption \ref{wc4ass1}(a)--(i) works for abelian categories $\A$ of any dimension $m\ge 0$, e.g., when $\A=\coh(X)$ for $X$ a smooth projective $m$-fold. However, in Assumption~\ref{wc4ass1}(j), the formation of virtual classes $[\M_\al^\ss(\tau)]_\virt$ depends critically on $\dim\A$:
\begin{itemize}\itemsep=0pt
\item[(i)] If $\dim\A=1$ then moduli spaces $\M_\al^\rst(\tau)$ are smooth $\C$-schemes. If $\M_\al^\rst(\tau)=\M_\al^\ss(\tau)$ then $\M_\al^\ss(\tau)$ is a compact complex manifold, and thus has a fundamental class $[\M_\al^\ss(\tau)]_\fund$ in homology. This case includes $\A=\coh(X)$ for $X$ a projective curve, and $\A=\modCQ$ for~$Q$ a~quiver.
\item[(ii)] If $\dim\A=2$ then moduli spaces $\M_\al^\rst(\tau)$ have perfect obstruction theories in the sense of Behrend--Fantechi~\cite{BeFa}. If $\M_\al^\rst(\tau)=\M_\al^\ss(\tau)$ then $\M_\al^\ss(\tau)$ is proper, and thus has a virtual class $[\M_\al^\ss(\tau)]_\virt$ by~\cite{BeFa}. This case includes $\A=\coh(X)$ for $X$ a projective surface, and $\A=\modCQI$ for $(Q,I)$ a quiver with relations (though see Remark~\ref{wc4rem4}).
\end{itemize}

In general if $\dim\A\ge 3$ there is no way to define virtual classes $[\M_\al^\ss(\tau)]_\virt$, and enumerative invariants do not exist. However, there are three special cases in which $[\M_\al^\ss(\tau)]_\virt$ can be defined by a mathematical trick:
\begin{itemize}
\itemsep=0pt
\setlength{\parsep}{0pt}
\item[(iii)] If $\A=\coh(X)$ for $X$ a Calabi--Yau 3-fold, as in Donaldson--Thomas \cite{DoTh}, Thomas \cite{Thom1} and Joyce--Song \cite{JoSo}. (Essentially the same idea is used for Vafa--Witten invariants of surfaces in Tanaka--Thomas~\cite{TaTh1,TaTh2}.)
\item[(iv)] If $\A=\coh(X)$ for $X$ a Fano 3-fold, as in Thomas \cite{Thom1}.
\item[(v)] If $\A=\coh(X)$ for $X$ a Calabi--Yau 4-fold, as in Borisov--Joyce \cite{BoJo}, Oh--Thomas \cite{OhTh} and Cao--Leung~\cite{CaLe}.
\end{itemize}
We will discuss (iii), (iv) in Section~\ref{wc436}, and (v) in Section~\ref{wc44}.

\begin{rem}\label{wc4rem4}
A well behaved abelian category $\A$ has an inclusion $A\hookra D^b\A$, and then $\Ext^k(E,F)=\Hom_{D^b\A}(E,F[k])$ for $E,F\in\A$ and~$k\in\Z$.

For abelian categories $\modCQI$ of representations of a quiver with relations $(Q,I)$, this definition of Ext groups may have $\Ext^k(E,F)\ne 0$ for infinitely many $k$, and $\cExt^\bu$ is not perfect, so the theory above does not work. However, there is a way to fix this. One can define a triangulated category $\T$, similar to $D^b\A$, with an inclusion $\A\hookra\T$ as the heart of a t-structure, such that $\Ext_{\sst\T}^k(E,F)=\Hom_{\sst\T}(E,F[k])$ has the properties we need, and then we replace $\Ext^k(E,F)$ above by $\Ext_{\sst\T}^k(E,F)$. See \cite[Remark~7.10]{JoSo} for discussion of this.	
\end{rem}

\subsubsection{General discussion for triangulated categories}\label{wc432}

Next we discuss our programme for triangulated categories, such as derived categories $D^b\coh(X)$ and $D^b\modCQ$. The moduli spaces $\M$, $\M^\pl$ of objects in such categories are higher stacks or derived stacks in the sense of~\cite{Toen1,Toen2,ToVa,ToVe1,ToVe2}.

One surprising fact is that the homology $H_*(\M),H_*\big(\M^\pl\big)$ is usually simpler and easier to compute than for abelian categories. For abelian categories, the direct sum $\op$ makes $\M$ into a commutative monoid in stacks, and $\M^\top$ into an H-space. In triangulated categories, the shift operator~$[1]$ acts as an inverse for $\op$ up to homotopy, making $\M$ into (roughly) an abelian group in stacks, and $\M^\top$ into a grouplike H-space, which are nicer than general H-spaces.

In the triangulated case, as in \cite{Joyc12} it is generally no longer true that $\check H_k\big(\M^\pl\big)\ab\cong \hat H_{k+2}(\M)/\ab D\big(\hat H_k(\M)\big)$, but this still holds under reasonable conditions when $k=0$, and the Lie algebra $\check H_0\big(\M^\pl\big)$ is what we need for applications to enumerative invariants. Here is the main result of the first author~\cite[Theorem~1.1]{Gros}:

\begin{thm}\label{wc4thm1}
Suppose $X$ is a smooth projective $\C$-scheme which is either a curve, a~surface, a~rational $3$- or $4$-fold, a toric variety, a flag variety, or one of some other classes we will not~give.

Write $K^i(X^\ran)$ for $i=0,1$ for the topological K-theory groups of the complex analytic topological space $X^\ran$ of $X$, and $K^0_{\rm sst}(X)$ for the $0^{\rm th}$ semi-topological K-theory group of $X$, in the sense of Friedlander and Walker~{\rm \cite{FrWa}}.

Let $\M$ be the moduli stack of objects in $D^b\coh(X)$, a higher $\C$-stack, which exists by {\rm\cite{ToVa}}. Then there is a canonical isomorphism of graded $R$-modules:
\begin{gather}
\hat H_*(\M)\cong R\big[K^0_{\rm sst}(X)\big]\ot_R \mathop{\rm Sym}\nolimits^*\bigl(K^0(X^\ran)\ot_\Z t^2R\big[t^2\big]\bigr)\ot_R\bigwedge\vphantom{\bigl(}^*\bigl(K^1(X^\ran)\ot_\Z tR\big[t^2\big]\bigr).\!\!\!
\label{wc4eq8}
\end{gather}
Here $\hat H_*(\M)$ is $H_*(\M,R)$ with grading shifted as in \eqref{wc2eq13}. The group ring $R[K^0_{\rm sst}(X)]:=\an{e^\al\colon \al\in K^0_{\rm sst}(X)}_R$ is graded by $\deg e^\al=-\chi(\al,\al)$. In the symmetric and exterior products $\mathop{\rm Sym}^*(\cdots),\bigwedge^*(\cdots)$, we take $K^i(X^\ran)$ to have degree $0$, and $t$ to be a formal variable with~\hbox{$\deg t=1$}.	

The vertex algebra structure on $\hat H_*(\M)$ is also equally explicit. By {\rm~\cite{Joyc12}}, this gives an explicit description of the Lie algebra $\check H_0\big(\M^\pl\big)=\hat H_2(\M)/D\big(\hat H_0(\M)\big)$.
\end{thm}

Even if we only want to study enumerative invariants for the abelian category $\coh(X)$, it may still be helpful to work with $\hat H_*(\M)$, $\check H_0\big(\M^\pl\big)$ for $\M$, $\M^\pl$ the moduli of objects in $D^b\coh(X)$, as these are easy to write down.

It is an important problem to extend Conjecture~\ref{wc4conj1} to Bridgeland stability conditions~\cite{Brid1} on triangulated categories such as $D^b\coh(X)$ and $D^b\modCQ$. This raises several difficult issues. There is already a significant literature for related questions on derived category versions of Donaldson--Thomas invariants of Calabi--Yau 3-folds. We make some remarks:
\begin{itemize}\itemsep=0pt
\item[(a)] The authors expect the issues of defining invariants $[\M_\al^\ss(\tau)]_\inv$ counting strictly $\tau$-semi\-sta\-bles in Conjecture \ref{wc4conj1}(i), and their characterization via pair invariants in Conjecture~\ref{wc4conj1}(iii), to extend to Bridgeland stability conditions $\tau=(Z,\cP)$ on triangulated categories $\T$ with essentially no changes. This is because both these issues can be written in the abelian subcategory $\T_\phi\subset\T$ of $\tau$-semistable objects with phase $\phi\in\R$.
\item[(b)] The authors only expect \eq{wc4eq1}--\eq{wc4eq2} to make sense in a triangulated category $\T$, and to have finitely many nonzero terms, if $\tau$, $\ti\tau$ are sufficiently close in the moduli space of Bridgeland stability conditions on $\T$.

If $\tau=(Z,\cP)$ and $\ti\tau=\big(\ti Z,\ti\cP\big)$ are close, then there should exist a unique third Bridgeland stability condition $\hat\tau=\big(\hat Z,\hat\cP\big)$ on $\T$ such that:
\begin{itemize}\itemsep=0pt
\item[(i)] The central charges $Z$, $\ti Z$, $\hat Z$ satisfy $\Re\hat Z=\Re Z$ and $\Im\hat Z=\Im\ti Z$.
\item[(ii)] Write $\A_{[\phi,\phi+\pi)},\ti{\mathcal A}_{[\phi,\phi+\pi)},\hat{\mathcal A}_{[\phi,\phi+\pi)}\subset\T$ for the abelian subcategories generated by $\tau$-, $\ti\tau$- and $\hat\tau$-semistable objects with phases in $[\phi,\phi+\pi)$, for $\phi\in\R$. Then $\hat{\mathcal A}_{[-\pi/2,\pi/2)}=\A_{[-\pi/2,\pi/2)}$ and $\hat{\mathcal A}_{[0,\pi)}=\ti{\mathcal A}_{[0,\pi)}$.
\end{itemize}

Then wall-crossing from $\tau$ to $\hat\tau$ is equivalent to wall-crossing in the abelian category $\hat{\mathcal A}_{[-\pi/2,\pi/2)}$, and wall-crossing from $\hat\tau$ to $\ti\tau$ is equivalent to wall-crossing in $\hat{\mathcal A}_{[0,\pi)}$, as explained in \cite[Section~7]{Joyc7}. Hence Conjecture \ref{wc4conj1}(ii) for triangulated categories reduces to Conjecture \ref{wc4conj1}(ii) for abelian categories.
\item[(c)] Bridgeland \cite{Brid2} has some interesting work on encoding Donaldson--Thomas invariants of a~3-Calabi--Yau triangulated category $\T$ into geometric structures (attractively called `Joyce structures') on the space of Bridgeland stability conditions $\Stab(\T)$ on $\T$. The Kontsevich--Soibelman wall-crossing formula \cite{KoSo} discussed in Section~\ref{wc42} is an ingredient in \cite[Definition~5.3]{Brid2}. By replacing this by \eq{wc4eq7} and basing the definition on $U\big(\check H_0\big(\M^\pl\big)\big)$, it may be possible to generalize \cite{Brid2} to enumerative invariants in other triangulated categories, for example $\T=D^b\coh(X)$ for $X$ a~projective surface.
\item[(d)] Work by Halpern-Leistner and coauthors (see~\cite{Halp} and references therein) generalizes the requirement in Assumption~\ref{wc4ass1}(j) that $\M_\al^\ss(\tau)$ should have a compact `coarse moduli space' to triangulated categories.
\end{itemize}

\subsubsection{Representations of quivers, and quivers with relations}
\label{wc433}

We will discuss these at length in Sections~\ref{wc5}--\ref{wc6} and \cite{Joyc13}, so we say only a few words here. Sections~\ref{wc5}--\ref{wc6} will prove Conjecture~\ref{wc4conj1} when $\A=\modCQ$ for $Q$ a quiver without oriented cycles (`without oriented cycles' makes $\M_\al^\ss(\tau)$ compact). This will be generalized in~\cite{Joyc13} to $\A=\modCQI$ for $(Q,I)$ a quiver with relations. There may be other categories of interest in representation theory that we can study using similar techniques.

\subsubsection{Vector bundles and coherent sheaves on curves}\label{wc434}

Let $X$ be a projective complex curve, that is, a compact Riemann surface. Then we can study moduli spaces $\M_{r,d}=\M^\ss_{(r,d)}(\mu)$ of $\mu$-semistable vector bundles (for $r>0$) or coherent sheaves (for $r=0$) on $X$ with rank $r$ and degree $d$, where $\mu$ is slope stability. This was done by Harder--Narasimhan \cite{HaNa} in algebraic geometry and Atiyah--Bott \cite{AtBo} in differential geometry, and many subsequent authors. It is common to restrict to $r,d$ coprime, as in \cite[Section~9]{AtBo}, as then $\M^\rst_{(r,d)}(\mu)=\M^\ss_{(r,d)}(\mu)$ is nonsingular, and a compact complex manifold.

There is only really one nice stability condition on $\coh(X)$, namely slope stability $\mu$. So wall-crossing as in Conjecture \ref{wc4conj1}(ii) is not interesting for curves. However, Conjecture \ref{wc4conj1}(i),~(iii) may still have something new to say. They predict that there are classes $[\M_{r,d}]_\inv$ in $H_*\big(\M^\pl,\Q\big)$ for all~$(r,d)$, equal to the fundamental class of $\M_{r,d}$ in $H_*\big(\M^\pl,\Z\big)$ when $r,d$ are coprime, which could be computed using Definition~\ref{wc4def1}. For example, one could ask whether these classes have any interesting number-theoretic properties, as in~\cite{HaNa}.

\subsubsection{Coherent sheaves on surfaces}\label{wc435}

Next consider our programme when $\A=\coh(X)$ for $X$ a~projective complex surface. We take stability conditions on $\A$ to be Gieseker stability for an ample line bundle $L\ra X$. Invariants counting Gieseker semistable sheaves on~$X$ have been studied in depth by Mochizuki~\cite{Moch}, and should be understood as an algebro-geometric analogue of Donaldson invariants of 4-manifolds~\cite{DoKr}.

When $\M_\al^\rst(\tau)=\M_\al^\ss(\tau)$, the virtual classes $[\M_\al^\ss(\tau)]_\virt$ in Assumption~\ref{wc4ass1}(i) are defined using Behrend--Fantechi~\cite{BeFa} perfect obstruction theories $\phi\colon \cF^\bu\ra\bL_{\M^\pl}$ as in \cite[Sections~5 and~6.1]{Moch}. If the geometric genus $p_g=h^{2,0}(X)$ has $p_g>0$, and $\rank\al>0$, there is a~trivial factor in $h^{-1}(\cF^\bu)$ which causes $[\M_\al^\ss(\tau)]_\virt=0$, see \cite[Proposition~6.2.2]{Moch}. As in \cite[Section~6.2]{Moch}, we can obtain nonzero invariants by fixing determinants, but that takes us outside the framework of Conjecture \ref{wc4conj1}. So here we restrict to the case~$p_g=0$.

\begin{conj}\label{wc4conj2}
Conjecture~\ref{wc4conj1} holds for $\A=\coh(X)$ with $X$ a projective complex surface with $p_g=0$, with `stability condition $\tau$' meaning Gieseker stability for an ample line bundle on~$X$, and other details as in~Section~\ref{wc431}.
\end{conj}

This is proved in the sequel \cite[Section~7.7]{Joyc13}, which also covers the case $p_g>0$. It builds on previous work by Mochizuki \cite{Moch} and others.

\subsubsection{Donaldson--Thomas theory for Calabi--Yau and Fano 3-folds}\label{wc436}

Donaldson--Thomas invariants counting semistable coherent sheaves on Calabi--Yau 3-folds were proposed by Donaldson--Thomas \cite{DoTh} and defined by Thomas \cite{Thom1}, who also gave a version for Fano 3-folds. The Calabi--Yau 3-fold version was extended in Joyce--Song \cite{JoSo} and Kontsevich--Soibelman~\cite{KoSo}.

Let $X$ be a Calabi--Yau or Fano 3-fold, and $\tau$ be Gieseker stability on $\A=\coh(X)$, and $\al\in K(\A)$ with $\M_\al^\rst(\tau)=\M_\al^\ss(\tau)$. Since $\dim\A=3$ in Section~\ref{wc431}, the natural obstruction theory $\cF^\bu\ra\bL_{\M^\pl}$ on $\M_\al^\ss(\tau)$ is perfect in $[-2,0]$, and it is perfect in $[-1,0]$ (which is necessary for the virtual class $[\M_\al^\ss(\tau)]_\virt$ from \cite{BeFa} to be defined) if $h^2\big((\cF^\bu)^\vee\big)=0$. For a~$\C$-point $[E]$ in~$\M_\al^\ss(\tau)$ we have
\begin{equation*}
h^2\big((\cF^\bu)^\vee\vert_{[E]}\big)=\Ext^3(E,E)\cong\Hom(E,E\ot K_X)^*.	
\end{equation*}
\begin{itemize}\itemsep=0pt
\item[(a)] If $X$ is a Calabi--Yau 3-fold then $K_X\cong\O_X$, and as $E$ is stable we have $\Hom(E,E\ot K_X)=\C$, so that $h^2\big((\cF^\bu)^\vee\big)=\O_{\M_\al^\ss(\tau)}$. As in~\cite{Thom1} we may modify $\cF^\bu\ra\bL_{\M^\pl}$ by deleting the line bundle $h^2\big((\cF^\bu)^\vee\big)$, and then~\cite{BeFa} gives a virtual class $[\M_\al^\ss(\tau)]_\virt$ of real dimension~$-2\chi_\A(\al,\al)=0$.

Donaldson--Thomas theory of Calabi--Yau 3-folds {\it does not fit into the set-up of Assumption}~\ref{wc4ass1} {\it and Conjecture}~\ref{wc4conj1}, as deleting $h^2\big((\cF^\bu)^\vee\big)$ means that $[\M_\al^\ss(\tau)]_\virt\in\check H_{-2}\big(\M^\pl\big)$, which is the wrong dimension.

In the sequel \cite[Section~1.5]{Joyc13}, the second author explains a modification of Conjecture~\ref{wc4conj1} adapted to Donaldson--Thomas theory of Calabi--Yau 3-folds. The graded vertex algebra structure on $\hat H_*(\M)$ in Assumption~\ref{wc4ass1}(g) is replaced by a {\it graded vertex Lie algebra}.
\item[(b)] If $X$ is a Fano 3-fold and $\dim\supp E>0$, that is, if $\dim\al>0$, then using $\tau$-semistability of~$E$ and~$K_X$ negative one can show that $\Hom(E,E\ot K_X)=0$, so $\cF^\bu\ra\bL_{\M^\pl}$ is perfect in $[-1,0]$, and \cite{BeFa} gives a virtual class $[\M_\al^\ss(\tau)]_\virt$ of real dimension $2-2\chi_\A(\al,\al)$. If $\dim\al=0$ this argument does not work, and $[\M_\al^\ss(\tau)]_\virt$ is undefined. The sequel \cite[Section~7.8]{Joyc13} proves Conjecture~\ref{wc4conj1} for $\A=\coh(X)$ for classes $\al\in K(\A)$ with $\dim\al>0$ only.
\end{itemize}
Note that Donaldson--Thomas theory for Calabi--Yau 3-folds, and for Fano 3-folds, are significantly different, because of the difference in virtual dimension.

\subsubsection{Theories of vector bundles or sheaves with extra data}\label{wc437}

There are also enumerative invariant theories in the literature, in which the objects to be counted are one or more vector bundles or coherent sheaves on a smooth projective $\C$-scheme $X$, together with some morphisms between these sheaves. Such theories are usually special cases of the following general definition of \'Alvarez-C\'onsul and Garc\'\i a-Prada \cite{AlGa}.

\begin{dfn}
\label{wc4def2}
Let $Q=(Q_0,Q_1,h,t)$ be a quiver, and $X$ be a smooth projective $\C$-scheme. Assign a vector bundle $E_e\ra X$ to each edge $e\in Q_1$. Following \cite{AlGa}, define a $\C$-linear abelian category $\A$ whose objects are tuples $((V_v)_{v\in Q_0},(\phi_e)_{e\in Q_1})$, where $V_v\in\coh(X)$ and $\phi_e\colon V_{t(e)}\ra V_{h(e)}\ot E_e$ is a morphism in $\coh(X)$. Objects of $\A$ are called {\it twisted quiver sheaves}, or {\it twisted quiver bundles} if the $V_v$ are vector bundles. If $E_e=\O_X$ for all $e\in Q_1$ they are called {\it quiver sheaves}, or {\it quiver bundles}.

\'Alvarez-C\'onsul and Garc\'\i a-Prada \cite{AlGa} prove a Hitchin--Kobayashi correspondence for twisted quiver bundles, identifying solutions of a gauge theory equation with $\tau$-polystable objects in $\A$ for a stability condition $\tau$ on $\A$. This extends previously known Hitchin--Kobayashi correspondences. For example:
\begin{itemize}\itemsep=0pt
\item If $X$ is a Riemann surface, and $Q=\overset{v}{\bu}\,\rotatebox[origin=c]{90}{$\circlearrowleft$}\,\raisebox{1pt}{$\scriptstyle e$}$, and $E_e=K_X$, then twisted quiver bundles are {\it Higgs bundles}, as in Hitchin~\cite{Hitc}.
\item If $Q=\overset{v}{\bu}\,{\buildrel e\over \longra}\,\overset{w}{\bu}$ then {\it semistable pairs} as in Bradlow and Daskalopoulos \cite{Brad1,Brad2,BrDa}, Garc\'\i a-Prada \cite{GaPr1,GaPr2}, Thaddeus \cite{Thad} (who studies wall-crossing under change of stability condition), and others, are examples of $\tau$-semi\-stable quiver bundles with $V_v=\O_X$.
\end{itemize}
\end{dfn}

We usually have $\dim\A=\dim X+1$ (or $\dim\A=\dim X$ for some moduli spaces, under extra conditions). To fit such categories $\A$ into the enumerative invariant set-up of Section~\ref{wc43}, we need $\dim\A\le 2$ for virtual classes $[\M_\al^\ss(\tau)]_\virt$ to be defined, as in Section~\ref{wc431}, so we take $X$ to be a curve (or possibly a surface, under extra conditions). We also require the quiver $Q$ to have no oriented cycles, to ensure that $\M_\al^\ss(\tau)$ is compact. Under these conditions we expect Conjecture \ref{wc4conj1} to hold. Note that in contrast to Section~\ref{wc434}, there are usually many suitable stability conditions on $\A$, so wall-crossing in
Conjecture \ref{wc4conj1}(ii) is nontrivial.

\subsubsection{Theories using equivariant homology}\label{wc438}

Several interesting enumerative invariant theories involve moduli spaces $\M$, $\M^\pl$ with the action of an algebraic torus $T\cong\bG_m^k$, and we form virtual classes $[\M_\al^\ss(\tau)]_\virt$ in the {\it equivariant} homology $H_*^T\big(\M^\pl\big)$. See for example Tanaka--Thomas \cite{TaTh1,TaTh2} for the case of Vafa--Witten invariants. Then an equivariant version of Conjecture~\ref{wc4conj1} may apply. This is discussed in the sequel~\cite{Joyc13}.

One useful property of the equivariant setting, as used in \cite{TaTh1,TaTh2} for instance, is that to define virtual classes $[\M_\al^\ss(\tau)]_\virt$ in $H_*^T\big(\M^\pl\big)$ we do not need $\M_\al^\ss(\tau)$ to be compact, but only that the $T$-fixed locus $\M_\al^\ss(\tau)^T$ should be compact, which is often easier to satisfy.

\subsection{The conjecture for Calabi--Yau 4-fold DT4 invariants}
\label{wc44}

We summarize some ideas from derived algebraic geometry \cite{PTVV,Toen1,Toen2,ToVa,ToVe1,ToVe2} and Donaldson--Thomas type invariants of Calabi--Yau 4-folds~\cite{BoJo,CGJ,CaLe,OhTh}:
\begin{itemize}\itemsep=0pt
\item[(a)] Let $X$ be a smooth projective $\C$-scheme. Then To\"en and Vaqui\'e \cite{ToVa} construct a {\it derived moduli stack} $\bs\M$ of objects in $\coh(X)$ or in $D^b\coh(X)$, as a locally finitely presented derived $\C$-stack in the sense of To\"en and Vezzosi \cite{Toen1,Toen2,ToVe1,ToVe2}. It has a {\it virtual dimension} $\vdim_\C\bs\M$, a locally constant map $\bs\M\ra\Z$. The classical truncation $\M=t_0(\bs\M)$ is the usual moduli stack, as an Artin $\C$-stack or higher $\C$-stack.
\item[(b)] Pantev, To\"en, Vaqui\'e and Vezzosi \cite{PTVV} introduced a theory of shifted symplectic derived algebraic geometry, defining $k$-{\it shifted symplectic structures} $\om$ on a derived stack $\bs\cS$ for $k\in\Z$. If $X$ is a Calabi--Yau $m$-fold and $\bs\M$ is a derived moduli stack of objects in~$\coh(X)$ or~$D^b\coh(X)$ then $\bs\M$ has a $(2-m)$-shifted symplectic structure,~\cite[Corollary~2.13]{PTVV}.
\item[(c)] If $(\bs\cS,\om)$ is a $k$-shifted symplectic derived stack for $k$ even, Borisov--Joyce \cite[Section~2.4]{BoJo} define a notion of {\it orientation} on $(\bs\cS,\om)$.
\item[(d)] Let $(\bs\cS,\om)$ be a proper, oriented $-2$-shifted symplectic derived scheme with $\cS=t_0(\bs\cS)$. Then Borisov--Joyce \cite[Corollary~1.2]{BoJo} construct a {\it virtual class} $[\bs\cS]_\virt$ in $H_*(\cS,\Z)$ using derived differential geometry \cite{Joyc10,Joyc8,Joyc9,Joyc11}, of real dimension $\vdim_\C\bs\cS=\ha\vdim_\R\bs\cS$. Note that this is {\it half the expected dimension}. Oh--Thomas \cite{OhTh} provide an alternative definition of $[\bs\cS]_\virt$ in the style of Behrend--Fantechi~\cite{BeFa}.
\item[(e)] Let $X$ be a Calabi--Yau 4-fold, and $(\bs\M,\om)$ the $-2$-shifted symplectic derived moduli stack of objects in $\coh(X)$ or $D^b\coh(X)$ from (a),(b). Then Cao--Gross--Joyce \cite[Corollary~1.17]{CGJ} prove that $(\bs\M,\om)$ is orientable in the sense of~(c). By taking a shifted symplectic quotient by $[*/\bG_m]$, one can show that $\big(\bs\M^\pl,\om\big)$ is also $-2$-shifted symplectic and orientable. Choose an orientation on $\big(\bs\M^\pl,\om\big)$.
\item[(f)] Suppose $\al\in K(\coh(X))$ with $\M_\al^\rst(\tau)=\M_\al^\ss(\tau)$, where $\tau$ is Gieseker stability. Then $\bs\M_\al^\ss(\tau)$ is a proper, oriented $-2$-shifted symplectic derived scheme, and has a virtual class $[\bs\M_\al^\ss(\tau)]_\virt$ in $H_*\big(\M^\pl\big)$. Borisov--Joyce~\cite{BoJo} propose to define Donaldson--Thomas type `DT4 invariants' of~$X$ using these virtual classes. Cao--Leung \cite{CaLe} make a similar proposal using gauge theory rather than derived algebraic geometry.
\end{itemize}

We can now extend Section~\ref{wc41} to $\A=\coh(X)$ for $X$ a Calabi--Yau 4-fold. This works as in Section~\ref{wc431}, with the following important differences:
\begin{itemize}\itemsep=0pt
\item[(i)] In Section~\ref{wc431} we took $\Th^\bu=(\cExt^\bu)^\vee\op\si_\M^*(\cExt^\bu)[2n]$ in Assumption \ref{wc2ass1}(g), where $\cExt^\bu$ is the Ext complex on $\M\t\M$, and $\chi(\al,\be)=\chi_\A(\al,\be)+\chi_\A(\be,\al)$ in Assumption~\ref{wc4ass1}(c). As in~\cite{Joyc12}, in the Calabi--Yau $2m$-fold case we instead set $\Th^\bu=(\cExt^\bu)^\vee$ and $\chi(\al,\be)=\chi_\A(\al,\be)$, and use these to define the vertex and Lie algebra structures on $\hat H_*(\M)$, $\check H_*\big(\M^\pl\big)$ in Assumption~\ref{wc4ass1}(g),~(h) as in Sections~\ref{wc23}--\ref{wc24}. Serre duality for Calabi--Yau 4-folds implies that \eq{wc2eq6} holds with $n=2$, and $\chi$ is symmetric. Note that $\Th^\bu$ and $\chi$ are both (roughly speaking) half of their values in Section~\ref{wc431}. This is parallel to the virtual class~$[\bs\cS]_\virt$ in~(d) above having half the expected dimension.
\item[(ii)] In contrast to Section~\ref{wc431}, in the Calabi--Yau 4-fold case the complex geometry does not determine canonical orientations in Assumption \ref{wc4ass1}(e), but we must instead use Borisov--Joyce orientations~\cite{BoJo} as in~(c) above. As in~(e), orientations exist on $\M$, and we must choose~one.
\item[(iii)] As in~(d), we use Borisov--Joyce or Oh--Thomas virtual classes~\cite{BoJo,OhTh} instead of Behrend--Fantechi virtual classes~\cite{BeFa} (which are undefined in this case) in Assumption~\ref{wc4ass1}(j) when $\M_\al^\rst(\tau)=\M_\al^\ss(\tau)$.
\end{itemize}

\begin{conj}\label{wc4conj3}
Conjecture \ref{wc4conj1} holds for $\A=\coh(X)$ when $X$ is a Calabi--Yau 4-fold, with details in Assumption~\ref{wc4ass1} as above.	
\end{conj}

\subsection[\texorpdfstring{Donaldson theory for 4-manifolds with $b^2_+=1$}{Donaldson theory for 4-manifolds with b²₊=1}]{Donaldson theory for 4-manifolds with $\boldsymbol{b^2_+=1}$}\label{wc45}

Let $(X,g)$ be a compact, oriented Riemannian 4-manifold, which need not be simply-connected. Hodge theory gives a natural isomorphism $H^2_{\rm dR}(X,\R)\cong\cH^2$, where $H^2_{\rm dR}(X,\R)$ is the second de Rham cohomology group, and $\cH^2=\bigl\{\eta\in\Ga^\iy\big(\La^2T^*X\big)\colon \d\eta=\d^*\eta=0\bigr\}$ is the harmonic 2-forms on $X$. The splitting $\La^2T^*X=\La^2_+T^*X\op\La^2_-T^*X$ into self-dual and anti-self-dual 2-forms induces a splitting $\cH^2=\cH^2_+\op\cH^2_-$. Write $H^2_{\rm dR}(X,\R)=H^2_+(X,\R)\op H^2_-(X,\R)$ for the corresponding splitting in de Rham cohomology, and~$b^2_\pm(X)=\dim H^2_\pm(X,\R)$.

{\it Donaldson theory} is the study of enumerative invariants (called {\it Donaldson invariants}) that `count' connections with anti-self-dual curvature (called {\it instantons}) on vector or principal bundles $E\ra X$, as in Donaldson and Kronheimer \cite{DoKr}. They have the amazing property that they can distinguish different smooth structures on the same topological 4-manifold. Most work on Donaldson theory takes $E\ra X$ to be a principal $\SU(2)$- or $\SO(3)$-bundle, but we will consider $\U(m)$-bundles for~$m\ge 0$.

We divide into three cases:
\begin{itemize}\itemsep=0pt
\item[(i)] If $b^2_+(X)=0$ then Donaldson invariants cannot be defined.\footnote{But the moduli spaces can be used to restrict the intersection form of $X$, as in \cite{Dona1}.}
\item[(ii)] If $b^2_+(X)=1$ then Donaldson invariants can be defined. They depend on the splitting $H^2_{\rm dR}(X,\R)=H^2_+(X,\R)\op H^2_-(X,\R)$ induced by $g$, and have wall-crossing behaviour under changes of this splitting.
\item[(iii)] If $b^2_+(X)>1$ then Donaldson invariants can be defined, and are independent of $g$ and the splitting $H^2_{\rm dR}(X,\R)=H^2_+(X,\R)\op H^2_-(X,\R)$.
\end{itemize}

We are interested here in case (ii). We will sketch how Donaldson theory for $\U(m)$-bundles on $X$ when $b^2_+(X)=1$ can be made, conjecturally, into a theory with the structure described in Section~\ref{wc41}. The sequel \cite[Section~7.7]{Joyc13} will study case (iii) for coherent sheaves on surfaces. We have nothing new to say about case~(i).

Some papers on Donaldson theory with $b^2_+=1$ are Kotschick--Morgan \cite{Kots,KoMo}, Ellingsrud--G\"{o}ttsche \cite{ElGo}, Friedman--Qin \cite{FrQi}, G\"ottsche \cite{Gott}, G\"ottsche--Zagier \cite{GoZa}, Moore--Witten \cite{MoWi}, and G\"ottsche--Nakajima--Yoshioka \cite{GNY1}. We define the moduli spaces of instantons we are interested in.

\begin{dfn}\label{wc4def3}
Let $(X,g)$ be a compact, connected, oriented Riemannian 4-manifold, and $E\ra X$ be a {\it unitary vector bundle}, that is, a complex vector bundle $E\ra X$ with a Hermitian metric $h$ on its fibres. A {\it unitary connection} $\nabla_E$ is a connection $\nabla_E$ on $E$ preserving $h$. Write $\A_E$ for the topological space of unitary connections on $E$, with the $C^\iy$-topology. The group $\Aut(E)$ of unitary automorphisms of $E$ acts on $\A_E$, and its normal subgroup $\U(1)=\U(1)\cdot\id_E\subset\Aut(E)$ acts trivially, so that $\Aut(E)/\U(1)$ also acts on $\A_E$. We call a connection $\nabla_E$ {\it irreducible} if $\Stab_{\Aut(E)}(\nabla_E)=\U(1)$, and {\it reducible} otherwise. Write $\A_E^\irr\subseteq\A_E$ for the open subset of irreducible connections. Define topological stacks $\B_E^\irr\subseteq\B_E$ and $\B_E^{\irr,\pl}\subseteq\B_E^\pl$ as in \cite{Metz,Nooh1,Nooh2} by\looseness=-1
\begin{alignat*}{3}
&\B_E^\irr=\A_E^\irr/\Aut(E),\qquad && \B_E=\A_E/\Aut(E),&\\
&\B_E^{\irr,\pl}=\A_E^\irr/(\Aut(E)/\U(1)),\qquad && \B_E^\pl=\A_E/(\Aut(E)/\U(1)),&
\end{alignat*}
Then $\B_E^{\irr,\pl}$ is a topological space, as $\Aut(E)/\U(1)$ acts freely on $\A_E^\irr$.

Let $\nabla_E$ in $\A_E$ have curvature $F^{\nabla_E}$, and split $F^{\nabla_E}=F^{\nabla_E}_+\op F^{\nabla_E}_-$ for $F^{\nabla_E}_\pm$ the components in $\ad(E)\ot_\R\La^2_\pm T^*X$. We call $\nabla_E$ an {\it instanton} if
\e
F^{\nabla_E}_+=i\,\id_E\ot\om\qquad\text{for some $\om\in\cH^2_+$}.
\label{wc4eq9}
\e
Define moduli spaces of instantons $\M_E^\rst\subseteq\M_E^\ss$ by
\begin{gather}
\begin{split}
&\M_E^\rst=\bigl\{[\nabla_E]\in \B_E^{\irr,\pl}\colon \text{$\nabla_E$ is an instanton}\bigr\}\subset\B_E^{\irr,\pl},
\\
&\M_E^\ss=\bigl\{[\nabla_E]\in \B_E^\pl\colon \text{$\nabla_E$ is an instanton}\bigr\}\subset\B_E^\pl.
\end{split}
\label{wc4eq10}
\end{gather}

If $g$ is generic then $\M_E^\rst$ is a smooth manifold of dimension
\e
\dim\M_E^\rst=1+b^2_+(X)-\chi(\lb E\rb,\lb E\rb),
\label{wc4eq11}
\e
where $\chi\colon K^0(X)\t K^0(X)\ra\Z$ is the symmetric biadditive map
\begin{align}
\chi(\al,\be)={}&-\bigl(1-b_1(X)+b^2_+(X)\bigr)\rank\al\rank \be-2 \int_X\ch_1(\al)\ch_1(\be)\nonumber\\
&{}+2\rank\al \int_X\ch_2(\be)+2\rank\be\int_X\ch_2(\al),\label{wc4eq12}
\end{align}
with $\ch_i(-)$ the Chern characters. Here in~\eq{wc4eq11}, the first term $1$ compensates for the quotient by $\U(1)$ in \eq{wc4eq10}, the second $b^2_+(X)$ compensates for $\om\in\cH^2_+$ in~\eq{wc4eq9}, and the third $-\chi(\lb E\rb,\lb E\rb)$ comes from the Atiyah--Singer index theorem as in Kronheimer \cite[equation~(3), p.~64]{Kron}. If $g$ is not generic then $\M_E^\rst$ is a {\it derived manifold} of virtual dimension~\eq{wc4eq11}, in the sense of~\cite{Joyc10,Joyc8,Joyc9,Joyc11}.

As in \cite{DoKr}, by {\it Uhlenbeck compactification} the moduli spaces $\M_E^\rst\subseteq\M_E^\ss$ have completions $\oM_E^\rst\subseteq\oM_E^\ss$ such that $\oM_E^\ss$ is compact, and if $\oM_E^\rst=\oM_E^\ss$ then $\oM_E^\ss$ should have a virtual class $[\oM_E^\ss]_\virt$, of dimension \eq{wc4eq11}. Points of $\oM_E^\ss\sm\M_E^\ss$ are singular instantons with `bubbles' at finitely many points in~$X$.
\end{dfn}

\begin{rem}
\label{wc4rem5}
The usual definition \cite{DoKr} of instantons (generally for $\SU(2)$ connections) has $F^{\nabla_E}_+=0$ rather than \eq{wc4eq9}. In fact $\om$ in \eq{wc4eq9} is determined by $[\om]=\pi_{H^2_+(X,\R)}(2\pi c_1(E))$ in $H^2_+(X,\R)$, so $\om=0$ if $E$ is an $\SU(m)$-bundle.	
\end{rem}

We now restrict to the case $b^2_+(X)=1$. We outline the analogue of the data in Assumption~\ref{wc4ass1}(a)--(j). We can take $\A$ to be the category of pairs $(E,\nabla_E)$ of a unitary bundle $E\ra X$ with a unitary connection $\nabla_E$. Then:
\begin{itemize}\itemsep=0pt\setlength{\leftskip}{0.43cm}
\item[(a)] The obvious choices for $\M$, $\M^\pl$ are the topological stacks
\begin{equation*}
\B=\coprod_{\substack{\text{\quad iso. classes $[E]$ of} \\ \text{unitary bundles $E\ra X$}}}\B_E, \qquad
\B^\pl=\coprod_{\substack{\text{\quad iso. classes $[E]$ of} \\ \text{unitary bundles $E\ra X$}}}\B_E^\pl.
\end{equation*}
These are studied in \cite{JTU}, and their topological realizations have homotopy equivalences
\begin{gather}
\begin{split}
&\B^\top\simeq \Map_{C^0}\bigg(X, \coprod_{m\ge 0}B\U(m)\bigg),\\
&(\B^\pl)^\top\simeq \Map_{C^0}\bigg(X, \coprod_{m\ge 0}B\U(m)\bigg)/B\U(1).
\end{split}
\label{wc4eq13}
\end{gather}

However, there is a problem: {\it it is not clear that the completions $\oM_E^\rst$, $\oM_E^\ss$ of $\M_E^\rst$, $\M_E^\ss$ map naturally to} $\B$, $\B^\pl$, and the authors expect that they do not. So we should not define invariants $[\oM_E^\ss]_\inv$ in $H_*\big(\B^\pl\big)$.

Instead, we define topological spaces
\e
\M= \Map_{C^0}(X,B\U\t\Z),\qquad
\M^\pl= \Map_{C^0}(X,B\U\t\Z)/B\U(1),
\label{wc4eq14}
\e
which are a kind of completion of \eq{wc4eq13} (in fact $\M$ is the {\it H-space completion} of $\B^\top$, and is studied in \cite{CGJ,JTU}). We propose that invariants $[\oM_E^\ss]_\inv$ should lie in $H_*\big(\M^\pl\big)$. See Remark \ref{wc4rem6} below on this.
\item[(b), (c)] We take $K(\A)=K^0(X)$, and $\chi$ as in \eq{wc4eq12}.
\item[(e)] Orientations on $\M$ are explained in Joyce--Tanaka--Upmeier \cite[Section~4.2.3]{JTU} and constructed in \cite[Theorem~4.6]{JTU}, following previous work of Donaldson.
\item[(g), (h)] The vertex algebra structure on $\hat H_*(\M)$ will be constructed in \cite{Joyc12}. It is the lattice vertex algebra on the super-lattice $K^0(X)\op K^1(X)$ with intersection form $\chi$ in \eq{wc4eq12}, so in a similar way to \eq{wc4eq8} we have
\begin{equation*}
\hat H_*(\M)\cong R\big[K^0(X)\big]\ot_R \mathop{\rm Sym}\nolimits^*\bigl(K^0(X)\ot_\Z t^2R\big[t^2\big]\bigr)\ot_R\bigwedge\vphantom{\bigl(}^*\bigl(K^1(X)\ot_\Z tR\big[t^2\big]\bigr).
\end{equation*}
This gives the Lie algebra~$\check H_0\big(\M^\pl\big)=\hat H_2(\M)/D\big(\hat H_0(\M)\big)$.
\item[(i)] By {\it stability condition} we mean either the Riemannian metric $g$, or the orthogonal splitting $H^2_{\rm dR}(X,\R)=H^2_+(X,\R)\op H^2_-(X,\R)$ induced by $g$, depending on your point of view. Here $H^2_+(X,\R)=\an{[\om]}_\R$ as $b^2_+(X)=1$, where $\om\in\cH^2_+$ is a harmonic self-dual 2-form on~$X$.

For comparison, Gieseker and $\mu$-stability conditions on $\coh(X)$ for a projective surface~$X$ with $b^2_+(X)=1$ correspond to K\"ahler classes $[\om]$ in $H^2_{\rm dR}(X,\R)$, where $\cH^2_+=\an{\om}$. For the wall-crossing formula \eq{wc4eq2}, the coefficients $U(-)$ should be defined as for $\mu$-stability for projective surfaces.
\item[(j)] If $\al\in K(\A)$, by $\M_\al^\rst(\tau)$, $\M_\al^\ss(\tau)$ we mean the disjoint union over isomorphism classes of unitary bundles $E\ra X$ with $\lb E\rb=\al$ of the moduli spaces $\oM_E^\rst$, $\oM_E^\ss$ above, considered as mapping to $\M^\pl$ as claimed in (a). Although the moduli spaces themselves depend on $g$, when $\M_\al^\rst(\tau)=\M_\al^\ss(\tau)$ the virtual class $[\M_\al^\ss(\tau)]_\virt\in H_*\big(\M^\pl\big)$ depends only on the splitting $H^2_{\rm dR}(X,\R)=H^2_+(X,\R)\op H^2_-(X,\R)$, which is why we have these two choices for stability conditions in (i). Note too that as $b^2_+(X)=1$, equation \eq{wc4eq11} gives $\vdim \M_\al^\ss(\tau)=2-\chi(\al,\al)$, as required.
\end{itemize}

\begin{rem}\label{wc4rem6}
Here is some justification for the choice of $\M$, $\M^\pl$ in \eq{wc4eq14}. Let $X$ be a~pro\-jec\-tive surface. Then we can form moduli stacks $\M_{\text{vect}}\subset\M_{\text{t-f}}\subset\M_{\text{perf}}$ of vector bundles, and torsion-free sheaves, and perfect complexes on $X$, and projective linear versions $\M_{\text{vect}}^\pl\subset\M_{\text{t-f}}^\pl\subset\M_{\text{perf}}^\pl$ with $\M_{\text{vect}}^\pl=\M_{\text{vect}}/[*/\bG_m]$, and so on. We have mapping stack presentations
\begin{equation*}
\M_{\text{vect}}\cong\Map\bigg(X, \coprod_{m\ge 0}[*/\GL(m,\C)]\bigg),\qquad \M_{\text{perf}}\cong\Map\big(X,\Perf_\C\big).
\end{equation*}

The analogues of $\M^\rst_E$, $\M^\ss_E$ in Definition \ref{wc4def3} are moduli spaces of (semi)\-stable vector bundles,
and are substacks of~$\M_{\text{vect}}^\pl$. But the analogues of the Uhlenbeck compactifications $\oM^\rst_E$,~$\oM^\ss_E$ are moduli spaces of (semi)stable {\it torsion-free sheaves} (thought of as singular vector bundles), so they are substacks of~$\M_{\text{t-f}}^\pl$, and hence of $\M_{\text{perf}}$, but {\it not} of~$\M_{\text{vect}}^\pl$. We have topological realizations
\begin{equation*}
\bigg(\coprod\limits_{m\ge 0}[*/\GL(m,\C)]\bigg)^\top\simeq \coprod\limits_{m\ge 0\!\!}B\U(m),\qquad \Perf_\C^\top\simeq B\U\t\Z,\qquad [*/\bG_m]^\top\simeq B\U(1).
\end{equation*}
Hence $\M_{\text{vect}}$, $\M_{\text{vect}}^\pl$ are analogous to $\B^\top$, $\big(\B^\pl\big)^\top$ in \eq{wc4eq13}, and $\M_{\text{perf}}$, $\M_{\text{perf}}^\pl$ are analogous to $\M$, $\M^\pl$ in \eq{wc4eq14}.
\end{rem}

\begin{conj}
\label{wc4conj4}
With the set up above, Conjecture \ref{wc4conj1} holds for Donaldson theory of 4-mani\-folds $X$ with $b^2_+(X)=1$, not necessarily simply-connected.	
\end{conj}

\subsection{Other gauge-theoretic enumerative invariant theories}\label{wc46}

We comment briefly on other gauge-theoretic invariants in the literature.

\subsubsection*{Casson invariants of 3-manifolds}
Casson invariants count flat connections on compact 3-manifolds, as in Akbulut--McCarthy \cite{AkMc}, Taubes \cite{Taub} and Boden--Herald \cite{BoHe}. They do not fit into the set-up of Section~\ref{wc41}, because the virtual dimension is wrong, as for Donaldson--Thomas invariants of Calabi--Yau 3-folds in Section~\ref{wc436}(a), and for other reasons.

\subsubsection*{$\boldsymbol{G_2}$-instantons on $\boldsymbol{G_2}$-manifolds}
Let $X$ be a compact 7-manifold and $(\vp,g)$ a torsion-free $G_2$-structure on $X$, as in Joyce~\cite{Joyc1}. The $G_2$-structure induces a splitting $\La^2T^*X=\La^2_7\op\La^2_{14}$ into subbundles of ranks~7,~14. A~connection~$\nabla_E$ on a vector or principal bundle $E\ra X$ is a $G_2$-{\it instanton} if $\pi_{\La^2_7}\big(F^{\nabla_E}\big)=0$. Donaldson and Segal \cite{DoSe} propose defining enumerative invariants of $(X,\vp,g)$ by counting $G_2$-instantons. Again, these do not fit into the set-up of Section~\ref{wc41}, because the virtual dimension is wrong.

\subsubsection*{$\boldsymbol{\Spin(7)}$-instantons on $\boldsymbol{\Spin(7)}$-manifolds}
Let $X$ be a compact, simply-connected 8-manifold and $(\Om,g)$ a torsion-free $\Spin(7)$-structure on~$X$, as in Joyce~\cite{Joyc1}. The $\Spin(7)$-structure induces a splitting $\La^2T^*X=\La^2_7\op\La^2_{21}$ into subbundles of ranks~7,~21. A connection $\nabla_E$ on a~vector or principal bundle $E\ra X$ is a~$\Spin(7)$-{\it instanton} if $\pi_{\La^2_7}\big(F^{\nabla_E}\big)=0$. Donaldson and Thomas~\cite{DoTh} propose defining enumerative invariants of $(X,\Om,g)$ by counting $\Spin(7)$-instantons.

There is a strong analogy between counting $\Spin(7)$-instantons and Donaldson theory of 4-manifolds. Under this analogy we compare $\La^2_7\leftrightsquigarrow \La^2_+$ and $\La^2_{21}\leftrightsquigarrow \La^2_-$. There is a splitting $H^2_{\rm dR}(X,\R)=H^2_7(X,\R)\op H^2_{21}(X,\R)$, with $b^2_7=\dim H^2_7(X,\R)$ the analogue of $b^2_+$. As in Section~\ref{wc45}, we divide into three cases:
\begin{itemize}\itemsep=0pt
\item[(i)] $b^2_7=0$, which happens if $g$ has holonomy group $\Hol(g)=\Spin(7)$. Then the analogy with Donaldson theory suggests that enumerative invariants counting $\Spin(7)$ instantons cannot be defined.
\item[(ii)] $b^2_7=1$, which happens if $\Hol(g)=\SU(4)$, and $X$ is a Calabi--Yau 4-fold. Then the set-up of Section~\ref{wc41} may work. Note however that this may just be a gauge-theoretic version of the DT4 invariants of Calabi--Yau 4-folds discussed in Section~\ref{wc44}, as proposed in Cao--Leung~\cite{CaLe}.
\item[(iii)] $b^2_7>1$, which happens if $b^2_7=2$ and $\Hol(g)=\Sp(2)$ or $b^2_7=3$ and $\Hol(g)=\Sp(1)\t\Sp(1)$, and $g$ is hyperk\"ahler. Then we might hope that enumerative invariants can be defined, which are independent of $g$, but which do not fit into the framework of Section~\ref{wc41}, though see Question~\ref{wc4quest4}.
\end{itemize}

\subsection{Questions for future work}\label{wc47}

Here are some questions that seem to the authors to be interesting.

\begin{quest}\label{wc4quest1}
In a $\C$-linear enumerative invariant theory of the kind discussed in Section~{\rm \ref{wc41}}, can we write the family of invariants $[\M_\al^\ss(\tau)]_\inv$ for all $\al$ in $C(\A)$ in terms of a small amount of data by a universal formula?
\end{quest}

Examples of the kind of formula we have in mind are those writing the generating series of Donaldson invariants in terms of a finite collection of Kronheimer--Mrowka basic classes, or Seiberg--Witten invariants, as in Kronheimer--Mrowka \cite{KrMr}, Fintushel--Stern~\cite{FiSt}, Witten \cite{Witt}, Moore--Witten \cite{MoWi}, Mari\~no--Moore \cite{MaMo}, and G\"{o}ttsche--Nakajima--Yoshioka \cite{GNY3}. Using the results of Sections~\ref{wc5}--\ref{wc6}, we can investigate this question for $\A=\modCQ$, and the authors hope to write about this in a future paper in the series.

\begin{quest}\label{wc4quest2}
Can we extend the set up of Section~{\rm \ref{wc41}} to replace $H_*(\M)$, $H_*\big(\M^\pl\big)$ by $E_*(\M)$, $\ab E_*\big(\M^\pl\big)$ for $E_*(-)$ a complex-oriented generalized homology theory over $R$, such as K-homology?
\end{quest}

The vertex algebra and Lie algebra parts of Assumption~\ref{wc4ass1} will be extended to $E_*(\M)$, $E_*\big(\M^\pl\big)$ in \cite{Joyc12}. It turns out that $E_*(\M)$ is a `vertex $F$-algebra', where $F(x,y)$ is the formal group law associated to the complex-oriented cohomology theory~$E^*(-)$. There is an interesting literature on K-theoretic versions of enumerative invariants, which often form generating functions with attractive properties~-- see for example G\"{o}ttsche--Nakajima--Yoshioka~\cite{GNY2}, Okounkov~\cite{Okou}, G\"{o}ttsche--Kool~\cite{GoKo1}, Thomas~\cite{Thom2}, Laarakker~\cite{Laar}, Arbesfeld \cite{Arbe}, and Cao--Kool--Monavari~\cite{CKM}~-- and also on cobordism invariants, as in G\"ottsche--Kool \cite{GoKo2} and Shen \cite{Shen}. In Question~\ref{wc4quest1} over $E_*(-)$, should the universal formula depend on the formal group law~$F$? A~good starting point would be to study invariants $[\M_\al^\ss(\tau)]_\inv\in E_*\big(\M^\pl\big)$ for $\A=\modCQ$, as in Sections~\ref{wc5}--\ref{wc6}.

\begin{quest}\label{wc4quest3}
Can we use the theory of vertex algebras to understand more about the structure of enumerative invariants, given the appearance of vertex algebras in Assumption~{\rm \ref{wc4ass1}?} For example, to explain modular properties of generating functions of enumerative invariants? Remark~{\rm \ref{wc4rem3}} may help.
\end{quest}

Connections between vertex algebras and Donaldson theory, Vafa--Witten theory, or Seiberg--Witten theory of 4-manifolds are suggested by the work of Nakajima \cite{Naka1,Naka2} and Feigin--Gukov~\cite{FeGu}.

As a possible place to start, observe that in Conjecture \ref{wc4conj1}, for $\al\in K(\A)$ and with $R=\Q$ or $\C$, we can consider the graded $R$-vector space
\begin{equation*}
V^\al_*:=[\M_\al^\ss(\tau)]_\inv\cap H^*\big(\M_\al^\pl\big)\subset H_*\big(\M_\al^\pl\big).
\end{equation*}
This is finite-dimensional over $R$, and may be considered an approximation to the homology $H_*(\M_\al^\ss(\tau))$, since if $\M_\al^\ss(\tau)$ is a smooth projective $\C$-scheme and the restriction map $H^*\big(\M^\pl_\al\big)\ra H^*(\M_\al^\ss(\tau))$ is surjective (this is called {\it Kirwan surjectivity}, and can be proved in some situations), then $V^\al_*=H_*(\M_\al^\ss(\tau))$. For example, if $\M_\al^\ss(\tau)$ is a moduli space of rank 1 torsion-free sheaves on a simply-connected projective surface $X$ then $\M_\al^\ss(\tau)$ may be identified with a Hilbert scheme ${\rm Hilb}{}^{(n)}(X)$, and $V^\al_*\cong H_*\big({\rm Hilb}{}^{(n)}(X)\big)$.

We regard $V^\al_*$ as a `categorification' of $[\M_\al^\ss(\tau)]_\inv$. It would be interesting to use vertex algebra ideas to produce representations of interesting algebras on $\bigop_{\al\in S}V_*^\al$ for subsets $S\subset K(\A)$, just as Grojnowski~\cite{Groj} and Nakajima~\cite{Naka3} find representations of Heisenberg algebras on $\bigop_{n\ge 0}H_*\big({\rm Hilb}^{(n)}(X)\big)$.

\begin{quest}\label{wc4quest4}
Donaldson theory of compact, oriented $4$-manifolds $X$ with $b^2_+(X)=1$ fits directly into our theory, as in Section~{\rm \ref{wc45}}. Can we produce a variant of our theory which describes the case~{\rm$b^2_+(X)>1$?}

This should also apply to Mochizuki-style~{\rm \cite{Moch}} counting of coherent sheaves on smooth projective surfaces $X$ with $p_g >0$, and other situations. There should no longer be wall-crossing phenomena under change of stability condition, but counting strictly $\tau$-semistables, as in Conjecture~{\rm \ref{wc4conj1}(i)}, and general structures in the invariants, as in Question~{\rm \ref{wc4quest1}}, may still apply.	
\end{quest}

Question \ref{wc4quest4} is answered in the sequel \cite[Section~7.6]{Joyc13} for invariants counting coherent sheaves on smooth projective surfaces $X$ with~$p_g>0$.

\begin{quest}\label{wc4quest5}
Conjecture {\rm \ref{wc4conj1}} gives invariants $[\M_\al^\ss(\tau)]_\inv$ in $\Q$-homology rather than $\Z$-homology when $\M_\al^\rst(\tau)\ne\M_\al^\ss(\tau)$. The wall-crossing formulae \eq{wc4eq1}, \eq{wc4eq2}, \eq{wc4eq4} also involve coefficients $\ti U(-)$, $U(-)$ in $\Q$ rather than~$\Z$.

Is there a universal way to produce invariants $[\M_\al^\ss(\tau)]_\inv^\Z$ in $\Z$-homology, with an invertible $($modulo torsion$)$ transformation law to the $[\M_\al^\ss(\tau)]_\inv$ similar to \eq{wc3eq7}--\eq{wc3eq10}, satisfying a~different wall-crossing formula under change of stability condition, with coefficients in~$\Z$?
\end{quest}

As an example of what we have in mind, note that Joyce--Song \cite{JoSo} define `generalized Donaldson--Thomas invariants' $\bar{DT}{}^\al(\tau)\in\Q$, the analogue of our $[\M_\al^\ss(\tau)]_\inv$. In \cite[Section~6.2]{JoSo} they also define `BPS invariants' $\hat{DT}{}^\al(\tau)$ by
\begin{equation*}
\bar{DT}{}^\al(\tau)=\sum_{m\ge 1,\, m\mid\al}\frac{1}{m^2}\,
\hat{DT}{}^{\al/m}(\tau),
\end{equation*}
and they conjecture \cite[Conjecture~6.12]{JoSo} that $\hat{DT}{}^\al(\tau)\in\Z$ when $\tau$ is `generic'. The authors expect the answer to Question~\ref{wc4quest5} will involve linear operations mapping $\check H_*\big(\M^\pl_\al\big)\ra\check H_*\big(\M^\pl_{m\al}\big)$ for~$m=2,3,\dots$.

\begin{quest}\label{wc4quest6}
Do our conjectures have an interpretation in string theory?
\end{quest}

\section{An example: representations of quivers}\label{wc5}

We now prove Conjecture~\ref{wc4conj1} when $\A=\modCQ$ for $Q$ a quiver without oriented cycles. The proofs of Theorems~\ref{wc5thm1} and~\ref{wc5thm2} are postponed to Section~\ref{wc6}.

\subsection{Quivers, their moduli stacks, and vertex algebras}\label{wc51}

Here are the basic definitions in quiver theory, as in Benson~\cite[Section~4.1]{Bens}.

\begin{dfn}
\label{wc5def1}
A {\it quiver} $Q$ is a finite
directed graph. That is, $Q$ is a quadruple $(Q_0,Q_1,h,t)$, where
$Q_0$ is a finite set of {\it vertices}, $Q_1$ is a finite set of
{\it edges}, and $h,t\colon Q_1\ra Q_0$ are maps giving the {\it head}
and {\it tail} of each edge.

A closed loop of directed edges $\overset{v_0}{\bu}\,{\buildrel e_1\over \longra}\,\overset{v_1}{\bu}\,{\buildrel e_2\over \longra}\,\cdots \,{\buildrel e_n\over \longra}\,\overset{v_n=v_0}{\bu}$ is an {\it oriented cycle} in $Q$. Later we will restrict to quivers with {\it no oriented cycles}.

A {\it representation} $(\bs V,\bs\rho)=((V_v)_{v\in Q_0},(\rho_e)_{e\in Q_1})$ of $Q$ gives
finite-dim\-en\-sion\-al $\C$-vector spa\-ces~$V_v$ for $v\in Q_0$, and linear maps $\rho_e\colon V_{t(e)}\ra V_{h(e)}$ for $e\in Q_1$. A {\it morphism of representations} $\bs\phi=(\phi_v)_{v\in Q_0}\colon (\bs V,\bs\rho)\ra(\bs W,\bs\si)$ gives linear maps $\phi_v\colon V_v\ra W_v$ for $v\in Q_0$ with $\phi_{h(e)}\ci\rho_e=\si_e\ci\phi_{t(e)}$ for $e\in Q_1$. Write $\modCQ$ for the $\C$-linear abelian category of representations of~$Q$.

Write $\Z^{Q_0}$ for the abelian group of maps $\bs d\colon Q_0\ra\Z$, and $\N^{Q_0}\subset\Z^{Q_0}$ for the subset of maps $\bs d\colon Q_0\ra\N$. The {\it dimension vector} $\bdim(\bs V,\bs\rho)\in\N^{Q_0}$ of a representation $(\bs V,\bs\rho)$ is $\bdim(\bs V,\bs\rho)\colon v\mapsto\dim_\C V_v$. This induces a surjective morphism~$\bdim\colon K_0(\modCQ)\ra\Z^{Q_0}$.
\end{dfn}

We describe the moduli stacks $\M$, $\M^\pl$ for $\A=\modCQ$.

\begin{dfn}
\label{wc5def2}
Let $Q=(Q_0,Q_1,h,t)$ be a quiver. Write $\M$ for the moduli stack of objects $(\bs V,\bs\rho)$ in $\modCQ$. Then there is a natural decomposition
\begin{equation*}
\M= \coprod_{\bs d\in\N^{Q_0}}\M_{\bs d},
\end{equation*}
where $\M_{\bs d}$ is the moduli stack of $(\bs V,\bs\rho)$ with $\bdim(\bs V,\bs\rho)=\bs d$. For any such $(\bs V,\bs\rho)$, by considering isomorphisms $V_v\cong\C^{\bs d(v)}$ for $v\in Q_0$ we see we may write $\M_{\bs d}$ explicitly as a quotient stack
\begin{gather*}
\M_{\bs d}=[R_{\bs d}/\GL_{\bs d}], \qquad\text{where}\quad R_{\bs d}= \prod_{e \in Q_1} \Hom\big(\C^{\bs d(t(e))},\C^{\bs d(h(e))}\big)\\
\text{and} \quad  \GL_{\bs d}=  \prod_{v \in Q_0} \GL(\bs d(v),\C),\quad\text{with group action}\\
(A_v)_{v\in Q_0}\colon \ ((B_e)_{e\in Q_1})\longmapsto \big(A_{h(e)}\ci B_e\ci A_{t(e)}^{-1}\big)_{e\in Q_1}.
\end{gather*}

Write $\cV_v\ra\M$ for $v\in Q_0$ for the tautological vector bundle with
\begin{equation*}
\cV_v\vert_{[((V_v)_{v\in Q_0},(\rho_e)_{e\in Q_1})]}=V_v,
\end{equation*}
and write $\cV_{v,\bs d}=\cV_v\vert_{\M_{\bs d}}$ for $\bs d\in\N^{Q_0}$, so that $\rank \cV_{v,\bs d}=\bs d(v)$. As $R_{\bs d}$ is contractible, we have $\bA^1$-homotopy equivalences
\begin{equation*}
\M_{\bs d}\simeq[*/\GL_{\bs d}]=\prod_{v\in Q_0}[*/\GL(\bs d(v),\C)].	
\end{equation*}
Thus the topological realization of $\M_{\bs d}$ is
\begin{equation*}
\M_{\bs d}^\top\simeq\prod_{v\in Q_0}B\GL(\bs d(v),\C).
\end{equation*}
Let $R$ be a commutative $\Q$-algebra, such as $\Q$, $\R$ or $\C$. As $\GL(r,\C)\simeq\U(r)$, the computation of $H^*(B\U(r))$ by Milnor and Stasheff \cite[Theorem~14.5]{MiSt} implies that the cohomology of $\M_{\bs d}$ over $R$~is
\e
H^*(\M_{\bs d})=H^*\big(\M_{\bs d}^\top,R\big)\cong R\big[c_{v,{\bs d}}^i\colon v\in Q_0,\, i=1,2,\dots,\bs d(v)\big],
\label{wc5eq2}
\e
where $c_{v,{\bs d}}^i$ is a formal variable of degree $2i$, with $c_{v,{\bs d}}^i=c_i(\cV_{v,\bs d})$. The homology $H_*(\M_{\bs d})$ is the $R$-linear dual of \eq{wc5eq2}, and $H_*(\M)=\bigop_{\bs d\in\N^{Q_0}}H_*(\M_{\bs d})$.

Similarly, the projective linear moduli stack $\M^\pl$ from Definition \ref{wc2def4} is
\e
\M^\pl= \coprod_{\bs d\in\N^{Q_0}}\M_{\bs d}^\pl,\qquad\text{where}\quad \M_{\bs d}^\pl=[R_{\bs d}/\PGL_{\bs d}],
\label{wc5eq3}
\e
for $\PGL_{\bs d}=\GL_{\bs d}/\bG_m$ with $\bG_m=\bigl\{(\la\,\id_{\bs d(v)})_{v\in Q_0}\colon 0\ne\la\in\C\bigr\}\subseteq\GL_{\bs d}$.
\end{dfn}

We can describe the Ext groups $\Ext^i(D,E)$, the Euler form $\chi_Q$, and the Ext complex $\cExt^\bu$ explicitly for~$\modCQ$.

\begin{dfn}\label{wc5def3}
Let $Q=(Q_0,Q_1,h,t)$ be a quiver. It is well known that $\Ext^i(D,E)=0$ for all $D,E\in\modCQ$ and $i>1$, and
\begin{equation*}
\dim_\C\Hom(D,E)-\dim_\C\Ext^1(D,E)=\chi_Q(\bdim D,\bdim E),
\end{equation*}
where $\chi_Q\colon \Z^{Q_0}\t\Z^{Q_0}\ra\Z$ is the Euler form of
$\modCQ$, given by
\e
\chi_Q(\bs d,\bs e)= \sum_{v\in Q_0}\bs d(v)\bs e(v)-\sum_{e\in
Q_1}\bs d(t(e))\bs e(h(e)).\label{wc5eq4}
\e
The Ext complex $\cExt^\bu\ra\M\t\M$ may be written explicitly as the two-term complex of vector bundles in degrees~0,~1:
\e
\cExt^\bu=\Bigl[\xymatrix@C=50pt{
\mathop{\bigop_{v\in Q_0} \cV_v^*\bt \cV_v}\limits_0 \ar[r]^{\la} & \mathop{\bigop_{e\in Q_1} \cV_{t(e)}^*\bt \cV_{h(e)}}\limits_1}\Bigr],
\label{wc5eq5}
\e
where `$\bt$' is external tensor product, and the morphism $\la$ depends on the point in $R_{\bs d}$ in~$\M_{\bs d}=[R_{\bs d}/\GL_{\bs d}]$.
\end{dfn}

We then make $\hat H_*(\M)$ into a graded vertex algebra, and $\check H_*\big(\M^\pl\big)$ into a graded Lie algebra, as in Sections~\ref{wc23}--\ref{wc24} with the data in Assumption \ref{wc2ass1} chosen as in Section~\ref{wc431}, so in particular we take
\begin{alignat}{3}
&K(\modCQ)=\Z^{Q_0}, \qquad && \chi(\al,\be)=\chi_Q(\al,\be)+\chi_Q(\be,\al), &\nonumber\\
&\ep_{\al,\be}=(-1)^{\chi_Q(\al,\be)}, \qquad && \Th^\bu=(\cExt^\bu)^\vee\op\si_\M^*(\cExt^\bu).&\label{wc5eq6}
\end{alignat}

\subsection{(Weak) stability conditions on quiver categories}\label{wc52}

{\it Slope stability conditions} on quiver categories are an important class.

\begin{dfn}\label{wc5def4}
Let $Q$ be a quiver, and in the situation of Section~\ref{wc33} with $\A=\modCQ$, take $K(\A)=\Z^{Q_0}$, so that $C(\A)=\N^{Q_0}\sm\{0\}$. Fix $\mu_v\in\R$ for all $v\in Q_0$. Define $\mu\colon C(\A)\ra\R$ by
\begin{equation*}
\mu(\bs d)=\frac{\sum_{v\in Q_0}\mu_v\bs d(v)}{\sum_{v\in Q_0}\bs d(v)} .
\end{equation*}
Then $(\mu,\R,\le)$ is a stability condition on $\modCQ$ in the sense of Definition~\ref{wc3def3}, called {\it slope stability}, which we often write as~$\mu$. We call $\mu$ a {\it slope function}.

For an object $E$ of $\modCQ$ to be $\mu$-stable, or $\mu$-semistable, is an open condition on the point~$[E]$ in~$\M$ or $\M^\pl$. Write $\M^\rst_{\bs d}(\mu)\subseteq\M^\ss_{\bs d}(\mu)\subseteq\M_{\bs d}^\pl$ for the open $\C$-substacks of $\mu$-(semi)stable objects. They are quotient stacks
\begin{equation*}
\M^\rst_{\bs d}(\mu)=[R^\rst_{\bs d}(\mu)/\PGL_{\bs d}],\qquad
\M^\ss_{\bs d}(\mu)=[R^\ss_{\bs d}(\mu)/\PGL_{\bs d}],
\end{equation*}
for $\PGL_{\bs d}$-invariant open subschemes $R^\rst_{\bs d}(\mu)\subseteq R^\ss_{\bs d}(\mu)\subseteq R_{\bs d}$.
\end{dfn}

Here is a class of slope functions for which the moduli spaces $\M^\rst_{\bs d}(\mu)$, $\ab\M^\ss_{\bs d}(\mu)$ are easy to understand.

\begin{dfn}\label{wc5def5}
Let $Q$ be a quiver, and $\mu$ a slope function on $\modCQ$ defined using $\mu_v\in\R$ for $v\in Q_0$. We call $\mu$ {\it increasing} if for all edges $\overset{v}{\bu}\,{\buildrel e\over \longra}\,\overset{w}{\bu}$ in $Q$ we have $\mu_v<\mu_w$. Such $\mu$ exist if and only if $Q$ has {\it no oriented cycles}.
\end{dfn}

\begin{prop}\label{wc5prop1}
Let $Q$ be a quiver with no oriented cycles, and $\mu$ be an increasing slope function on $Q$. Then for each $\bs d\in\N^{Q_0}\sm\{0\}$, either:
\begin{itemize}\itemsep=0pt
\item[{\rm (a)}] $\bs d=\de_v$ for some $v\in Q_0$, that is, $\bs d(v)=1$ and $\bs d(w)=0$ for $w\ne v$. Then $\M^\rst_{\bs d}(\mu)=\M^\ss_{\bs d}(\mu)$ is a single point~$*$.
\item[{\rm (b)}] $\bs d\ne\de_v$ for any $v\in Q_0$, and for some $t\in\R$ we have $\bs d(v)=0$ for all $v\in Q_0$ with $\mu_v\ne t$. Then $\M^\rst_{\bs d}(\mu)=\es$ and $\M^\ss_{\bs d}(\mu)\cong[*/\PGL_{\bs d}]$. Also  $2-\chi(\bs d,\bs d)<0$ in this case.
\item[{\rm (c)}] Neither {\rm (a)} nor {\rm (b)} hold. Then $\M^\rst_{\bs d}(\mu)=\M^\ss_{\bs d}(\mu)=\es$.
\end{itemize}
\end{prop}

\begin{proof} For (a), if $\bs d=\de_v$ then as $Q$ has no oriented cycles there are no edges $\overset{v}{\bu}\,{\buildrel e\over \longra}\,\overset{v}{\bu}$ in $Q$, and in \eq{wc5eq3} we have $R_{\bs d}=0$ and $\PGL_{\bs d}=\{1\}$, so $\M_{\bs d}^\pl=*$. But $\M^\rst_{\bs d}(\mu)=\M^\ss_{\bs d}(\mu)=\M_{\bs d}^\pl$ as any $(\bs V,\bs\rho)$ in class $\bs d$ in $\modCQ$ has no non-trivial subobjects, and so is automatically $\mu$-stable.

For (b), suppose $\bs d\ne\de_v$ for any $v$, and for $t\in\R$ we have $\bs d(v)=0$ for all $v\in Q_0$ with $\mu_v\ne t$. Then if $v_1,v_2\in Q_0$ with $\bs d(v_1),\bs d(v_2)>0$ we have $\mu_{v_1}=\mu_{v_2}=t$, so there are no edges $\overset{v_1}{\bu}\,{\buildrel e\over \longra}\,\overset{v_2}{\bu}$ in $Q$ as $\mu$ is increasing. Thus in \eq{wc5eq3} we have $R_{\bs d}=0$, so $\M_{\bs d}^\pl\cong[*/\PGL_{\bs d}]$.

Let $(\bs V,\bs\rho)$ lie in class $\bs d$ in $\modCQ$. Then $\bs\rho=\bs 0$ as $R_{\bs d}=0$. If $0\ne (\bs W,\bs 0)\subseteq(\bs V,\bs 0)$ is any subobject then $\mu(\lb\bs W,\bs 0\rb)=\mu(\lb\bs V,\bs 0\rb)=t$ as $\bs d(v)=0$ unless $\mu_v=t$. Hence $(\bs V,\bs\rho)$ is $\mu$-semistable and $\M^\ss_{\bs d}(\mu)=\M_{\bs d}^\pl\cong[*/\PGL_{\bs d}]$. Also as $\bs d\ne\de_v$ we have $\sum_{v\in Q_0}\bs d(v)>1$. So there exists a proper subobject $0\ne (\bs W,\bs 0)\subsetneq(\bs V,\bs 0)$, and $\mu(\lb\bs W,\bs 0\rb)=\mu(\lb\bs V,\bs 0\rb)$ shows that $(\bs V,\bs\rho)$ is not $\mu$-stable, giving $\M^\rst_{\bs d}(\mu)=\es$. As above, if $v_1,v_2\in Q_0$ with $\bs d(v_1),\bs d(v_2)>0$ there are no edges $\overset{v_1}{\bu}\,{\buildrel e\over \longra}\,\overset{v_2}{\bu}$ in $Q$. Thus from \eq{wc5eq4} and \eq{wc5eq6} we see that $\chi(\bs d,\bs d)=2\sum_{v\in Q_0}\bs d(v)^2$, so $\chi(\bs d,\bs d)\ge 4$ as $\bs d\ne\de_v$, and~$2-\chi(\bs d,\bs d)<0$.

For (c), as neither (a), (b) hold there exist $v_1,v_2\in Q_0$ with $\bs d(v_1),\bs d(v_2)>0$ and $\mu_{v_1}\ne\mu_{v_2}$. Set $t=\ha(\mu_{v_1}+\mu_{v_2})$. Define a subobject $(\bs W,\bs\si)\subset(\bs V,\bs\rho)$ in $\modCQ$ by $W_v=V_v$ if $\mu_v\ge t$ and $W_v=0\subseteq V_v$ if $\mu_v<t$, and $\bs\si=\bs\rho\vert_{\bs W}$. If $\overset{v}{\bu}\,{\buildrel e\over \longra}\,\overset{w}{\bu}$ is an edge in $Q$ then $\mu$ increasing implies that either (i) $W_v=V_v$, $W_w=V_w$, or (ii) $W_v=W_w=0$, or (iii) $W_v=0$, $W_w=V_w$, and in each case $\rho_e\colon V_v\ra V_w$ maps $W_v\ra W_w$, so $(\bs W,\bs\si)$ is well defined. Both $(\bs W,\bs\si)$ and $(\bs V,\bs\rho)/(\bs W,\bs\si)$ are nonzero, as one contains $V_{v_1}\ne 0$ and the other $V_{v_2}\ne 0$. Also $\mu\bigl([(\bs W,\bs\si)]\bigr)\ge t$ and $\mu\bigl([(\bs V,\bs\rho)/(\bs W,\bs\si)]\bigr)<t$, so $\mu\bigl([(\bs W,\bs\si)]\bigr)>\mu\bigl([(\bs V,\bs\rho)/(\bs W,\bs\si)]\bigr)$. Hence $(\bs V,\bs\rho)$ is $\mu$-unstable by Definition \ref{wc3def3}, for all $(\bs V,\bs\rho)$ with $\bdim(\bs V,\bs\rho)=\bs d$. Thus~$\M^\rst_{\bs d}(\mu)=\M^\ss_{\bs d}(\mu)=\es$.
\end{proof}

\begin{prop}\label{wc5prop2}
Let $Q$ be a quiver with no oriented cycles, and $\mu$ a slope function on $\modCQ$, and $\bs d\in \N^{Q_0}\sm\{0\}$ with $\M^\rst_{\bs d}(\mu)=\M^\ss_{\bs d}(\mu)$. Then $\M^\ss_{\bs d}(\mu)$ is a smooth projective $\C$-scheme.
\end{prop}

\begin{proof}
Suppose first that $\mu_v\in\Z$ for all $v\in Q_0$. Then using geometric invariant theory (GIT) from Mumford--Fogarty--Kirwan \cite{MFK}, King \cite[Theorem~4.1]{King} shows that $R^\rst_{\bs d}(\mu)$, $R^\ss_{\bs d}(\mu)$ are the open subschemes of GIT (semi)stable points for a certain linearization $\th$ of the action of $\PGL_{\bs d}$ on~$R_{\bs d}$ determined by the $\mu_v$. Thus $\M^\rst_{\bs d}(\mu)=[R^\rst_{\bs d}(\mu)/\PGL_{\bs d}]$ is not merely an Artin $\C$-stack, but a smooth quasi-projective $\C$-scheme. Also there exists a GIT quotient $\tiM^\ss_{\bs d}(\mu)=R^\ss_{\bs d}(\mu)/\!/_\th\PGL_{\bs d}$, which is a coarse moduli scheme for $\M^\ss_{\bs d}(\mu)$. As $Q$ has no oriented cycles, there is a $\bG_m$-subgroup of $\PGL_{\bs d}$ acting on the vector space $R_{\bs d}$ with only positive weights, so $\tiM^\ss_{\bs d}(\mu)$ is a projective $\C$-scheme. If $\M^\rst_{\bs d}(\mu)=\M^\ss_{\bs d}(\mu)$ then $\M^\rst_{\bs d}(\mu)=\tiM^\ss_{\bs d}(\mu)$, so $\M^\ss_{\bs d}(\mu)$ is a smooth projective $\C$-scheme.

The condition that $\mu_v\in\Z$ is unnecessary. As the notions of $\mu$-(semi)stability are unchanged by multiplying all $\mu_v$ by a positive number, the result holds for $\mu_v\in\Q$. To allow $\mu_v\in\R$, note that the space $\R^{Q_0}$ of values $(\mu_v)_{v\in Q_0}$ is divided into chambers by finitely many real hyperplanes of the form $\mu(\bs e)=\mu(\bs d-\bs e)$ for $0<\bs e<\bs d$, such that $\M^\rst_{\bs d}(\mu)$, $\M^\ss_{\bs d}(\mu)$ depend only on the codimension $k$ stratum $S$ in the induced stratification of $\R^{Q_0}$ containing $(\mu_v)_{v\in Q_0}$. As the hyperplanes $\mu(\bs e)=\mu(\bs d-\bs e)$ are rational, $S$ contains a rational point $(\ti\mu_v)_{v\in Q_0}$, and then $\M^\rst_{\bs d}(\mu)=\M^\rst_{\bs d}(\ti\mu)$ and $\M^\ss_{\bs d}(\mu)=\M^\ss_{\bs d}(\ti\mu)$. The proposition follows.	
\end{proof}

\subsection{Defining invariants}\label{wc53}

The next theorem, Conjecture~\ref{wc4conj1}(i),~(ii) for $\modCQ$, is proved in Section~\ref{wc6}.

\begin{thm}\label{wc5thm1}
Let $Q$ be a quiver with no oriented cycles, and use the notation of Section~{\rm \ref{wc51}}, so in particular we have a Lie algebra over the $\Q$-algebra~$R$
\begin{equation*}
\check H_0\big(\M^\pl\big)=\bigop_{\bs d\in\N^{Q_0}}\check H_0\big(\M_{\bs d}^\pl\big)=\bigop_{\bs d\in\N^{Q_0}}H_{2-2\chi_Q(\bs d,\bs d)}\big(\M_{\bs d}^\pl\big),
\end{equation*}
for $\chi_Q$ as in \eq{wc5eq4}. Then for all weak stability conditions $(\tau,T,\le)$ on $\modCQ$ in the sense of Section~{\rm \ref{wc33}}, and for all $\bs d\in \N^{Q_0}\sm\{0\}$, there exist unique classes $[\M_{\bs d}^\ss(\tau)]_\inv\in\check H_0\big(\M_{\bs d}^\pl\big)$ with the properties:
\begin{itemize}\itemsep=0pt
\item[{\rm (i)}] Suppose $\mu$ is a slope function on $\modCQ$, and $\bs d\in\N^{Q_0}\sm\{0\}$ with $\M^\rst_{\bs d}(\mu)=\M^\ss_{\bs d}(\mu)$, so that $\M^\ss_{\bs d}(\mu)$ is a smooth projective $\C$-scheme by Proposition~{\rm \ref{wc5prop2}}. Then $\dim_\C\M^\ss_{\bs d}(\mu)=1-\chi_Q(\bs d,\bs d)$, so it has a fundamental class $\big[\M^\ss_{\bs d}(\mu)\big]_\fund$ in $H_{2-2\chi_Q(\bs d,\bs d)}\big(\M^\ss_{\bs d}(\mu)\big)$, and $\big[\M_{\bs d}^\ss(\mu)\big]_\inv$ is the pushforward $\io_*\bigl([\M^\ss_{\bs d}(\mu)]_\fund\bigr)$ under the inclusion $\io\colon \M^\ss_{\bs d}(\mu)\hookra\M_{\bs d}^\pl$.

This includes the case when $\M^\ss_{\bs d}(\mu)=\es$, with~$[\M_{\bs d}^\ss(\mu)]_\inv=0$.
\item[{\rm (ii)}] Let $(\tau,T,\le)$ and $\big(\ti\tau,\ti T,{\le}\big)$ be two weak stability conditions on $\modCQ$. Then as for \eq{wc4eq1}--\eq{wc4eq2}, for all $\bs d\in \N^{Q_0}\sm\{0\}$ we have
\begin{gather}
[\M_{\bs d}^\ss(\ti\tau)]_\inv=
\sum_{\substack{n\ge 1,\,\bs d_1,\dots,\bs d_n\in
\N^{Q_0}\sm\{0\}\colon \\ \bs d_1+\cdots+\bs d_n=\bs d }}  \!\!\!
\begin{aligned}[t]
\ti U(\bs d_1,&\dots,\bs d_n;\tau,\ti\tau)\cdot\bigl[\bigl[\cdots\bigl[\big[\M_{\bs d_1}^\ss(\tau)\big]_\inv,\\
&
\big[\M_{\bs d_2}^\ss(\tau)\big]_\inv\bigr],\dots\bigr],\big[\M_{\bs d_n}^\ss(\tau)\big]_\inv\bigr],
\end{aligned}
\label{wc5eq7}\\
[\M_{\bs d}^\ss(\ti\tau)]_\inv=
\sum_{\substack{n\ge 1,\,\bs d_1,\dots,\bs d_n\in
\N^{Q_0}\sm\{0\}\colon \\ \bs d_1+\cdots+\bs d_n=\bs d }}  \!\!\!
\begin{aligned}[t]
U(\bs d_1,&\dots,\bs d_n;\tau,\ti\tau)\cdot\big[\M_{\bs d_1}^\ss(\tau)\big]_\inv *\\
&
\big[\M_{\bs d_2}^\ss(\tau)\big]_\inv*\cdots *\big[\M_{\bs d_n}^\ss(\tau)\big]_\inv,
\end{aligned}\label{wc5eq8}
\end{gather}
which are equivalent equations, \eq{wc5eq7} in the Lie algebra $\check H_0\big(\M^\pl\big)$ and \eq{wc5eq8} in its universal enveloping algebra $U\big(\check H_0\big(\M^\pl\big)\big)$.
\item[{\rm (iii)}] Let $\mu$ be an increasing slope function on $\modCQ$. Then
\e
[\M_{\bs d}^\ss(\mu)]_\inv=\begin{cases} 1\in H_0\big(\M_{\bs d}^\pl\big)\cong R, & \bs d=\de_v,\  v\in Q_0, \\ 0, & \text{otherwise}. 	
\end{cases}
\label{wc5eq9}
\e
This follows from {\rm (i)} and Proposition {\rm\ref{wc5prop1}}, noting that in Proposition~{\rm \ref{wc5prop1}(b)}, $[\M_{\bs d}^\ss(\mu)]_\inv$ lies in $H_{<0}\big(\M_{\bs d}^\pl\big)=0$ as~$2-\chi(\bs d,\bs d)<0$.
\end{itemize}
\end{thm}

The sequel \cite[Section~6]{Joyc13} gives an alternative proof of Theorem~\ref{wc5thm1}.

\subsection{Morphisms of quivers}\label{wc54}

Here is a new notion of morphisms of quivers, which is designed to be compatible with the morphisms of vertex and Lie algebras in~Section~\ref{wc25}.

\begin{dfn}\label{wc5def6}
Let $Q=(Q_0,Q_1,h,t)$ and $Q'=(Q_0',Q_1',h',t')$ be quivers. A {\it morphism} $\la\colon Q\ra Q'$ is a pair $\la=(\la_0,\la_1)$, where $\la_0\colon Q_0\ra Q_0'$ is a map, and $\la_1\subseteq Q_1\t Q_1'$ a subset satisfying:
\begin{itemize}\itemsep=0pt
\item[(i)] If $(e,e')\in\la_1$ then $\la_0\ci h(e)=h'(e')$ and $\la_0\ci t(e)=t'(e')$.
\item[(ii)] If $v,w\in Q_0$ and $e'\in Q_1'$ with $\la_0(v)=h'(e')$ and $\la_0(w)=t'(e')$, there exists unique $e\in Q_1$ such that $(e,e')\in\la_1$ and $h(e)=v$, $t(e)=w$.
\item[(iii)] The projection $\pi_{Q_1}\colon \la_1\ra Q_1$ mapping $(e,e')\mapsto e$ is injective.
\end{itemize}

If $Q''=(Q_0'',Q_1'',h'',t'')$ is another quiver and $\mu\colon Q'\ra Q''$ is a morphism, the {\it composition} $\mu\ci\la=((\mu\ci\la)_0,(\mu\ci\la)_1)$ is given by $(\mu\ci\la)_0=\mu_0\ci\la_0$ and
\begin{equation*}
(\mu\ci\la)_1=\bigl\{(e,e'')\in Q_1\t Q_1''\colon \text{$\exists e'\in Q_1'$ with $(e,e')\in \la_1$ and $(e',e'')\in\mu_1$}\bigr\}.
\end{equation*}
It is easy to show this is a morphism, and makes quivers into a category.

For $\la$ as above, define a $\C$-linear exact functor $\Si_\la\colon \modCQ\ra\modCQ'$~by
\begin{align*}
&\Si_\la\colon \ ((V_v)_{v\in Q_0},(\rho_e)_{e\in Q_1})\mapsto((V'_{v'})_{v'\in Q_0'},(\rho'_{e'})_{e'\in Q_1'}) \ \text{on objects},\\
&\text{where} \  V'_{v'}=\bigop_{v\in Q_0\colon \la_0(v)=v'}V_v \quad \text{and}\quad  \rho'_{e'}=\sum_{e\in Q_1\colon (e,e')\in\la_1}\rho_e\quad\text{and} \\
&\Si_\la\colon \ (\phi_v)_{v\in Q_0}\mapsto (\phi'_{v'})_{v'\in Q_0'}\ \text{on morphisms, where}\
\phi'_{v'}=\sum_{v\in Q_0\colon \la_0(v)=v'}\phi_v.
\end{align*}
If $\mu\colon Q'\ra Q''$ is another morphism then~$\Si_{\mu\ci\la}=\Si_\mu\ci \Si_\la$.

The induced action $(\Si_\la)_*\colon K_0(\modCQ)\ra K_0(\modCQ')$ descends to
\begin{equation*}
\la_*\colon \ \Z^{Q_0}\ra\Z^{Q_0'}, \qquad \la_*\colon \ \bs d\mapsto \bs d', \qquad\text{where}\quad \bs d'(v')=\sum_{v\in Q_0\colon \la_0(v)=v'}\bs d(v).
\end{equation*}
Then $\la_*$ maps $\N^{Q_0}\ra\N^{Q_0'}$ and $\N^{Q_0}\sm\{0\}\ra\N^{Q_0'}\sm\{0\}$.

If $(\tau',T',\le)$ is a (weak) stability condition on $\modCQ'$, as in Section~\ref{wc33}, it is easy to check that $(\tau'\ci\la_*,T',\le)$ is a (weak) stability condition on $\modCQ$. If $\mu'$ is a slope function on $\modCQ'$ defined using constants $\mu'_{v'}$ for $v'\in Q_0'$ then $\mu=\mu'\ci\la_*$ is the slope function defined using $\mu_v=\mu'_{\la_0(v)}$ for~$v\in Q_0$.
\end{dfn}

\begin{dfn}\label{wc5def7}
Let $\la\colon Q\ra Q'$ be a morphism of quivers, as in Definition \ref{wc5def6}. Write~$\M$,~$\M'$ for the moduli stacks of objects in $\modCQ$, $\modCQ'$, and $\M^\pl$, $\M^{\prime\pl}$ for the projective linear moduli stacks, so that Section~\ref{wc51} and Sections~\ref{wc23}--\ref{wc24} define graded vertex algebras~$\hat H_*(\M)$, $\hat H_*(\M')$ and graded Lie algebras~$\check H_*\big(\M^\pl\big)$,~$\check H_*\big(\M^{\prime\pl}\big)$. We will use the constructions of Section~\ref{wc25} to define morphisms $\Om:\hat H_*(\M)\ra\hat H_*(\M')$ and~$\Om^\pl\colon \check H_*\big(\M^\pl\big)\ra \check H_*\big(\M^{\prime\pl}\big)$.

In Definition \ref{wc2def5}(a)--(c), let $T\colon \A\ra\A'$ be $\Si_\la\colon \modCQ\ra\modCQ'$, and write $\si_\la\colon \M\ra\M'$ and $\si_\la^\pl\colon \M^\pl\ra\M^{\prime\pl}$ for the induced stack morphisms. Define $\xi\colon \Z^{Q_0}\t\Z^{Q_0}\ra\Z$ and vector bundles $F\ra\M\t\M$, $G\ra\M$ by
\begin{gather}
\xi(\bs d,\bs e)=\sum_{v\ne w\in Q_0\colon \la_0(v)=\la_0(w)}\bs d(v)\bs d(w)+\sum_{e\in Q_1\colon \not\exists (e,e')\in \la_1}\bs d(t(e))\bs e(h(e)),
\label{wc5eq10}\\
F=\bigop_{v\ne w\in Q_0\colon \la_0(v)=\la_0(w)}\cV_v\bt\cV_w^*\op \bigop_{e\in Q_1\colon \not\exists (e,e')\in \la_1}\cV_{t(e)}\bt\cV_{h(e)}^*,
\label{wc5eq11}\\
G=\bigop_{v\ne w\in Q_0\colon \la_0(v)=\la_0(w)}\cV_v^*\ot\cV_w\op \bigop_{e\in Q_1\colon \not\exists (e,e')\in \la_1}\cV_{t(e)}^*\ot\cV_{h(e)}.
\label{wc5eq12}
\end{gather}

We now claim that the conditions Definition \ref{wc2def5}(i)--(v) hold. From \eq{wc5eq4} and Definition \ref{wc5def6} (especially (i)--(iii)) we can show that
\begin{equation*}
\chi_{Q'}(\la_*(\bs d),\la_*(\bs e))=\chi_Q(\bs d,\bs e)+\xi(\bs d,\bs e),
\end{equation*}
and then Definition \ref{wc2def5}(i),~(ii) follow from the second and third equations of \eq{wc5eq6}. For (iii), equations \eq{wc2eq18}--\eq{wc2eq21} follow from obvious compatibilities between $\cV_v$ and $\Phi$, $\Psi$. For~(iv), in $K_0(\Perf(\M\t\M))$ we have
\begin{align*}
&(\si_\la\t\si_\la)^*\big(\big[(\cExt^{\prime\bu})^\vee\big]\big)
 =\sum_{v'\in Q_0'}(\si_\la\t\si_\la)^*([\cV_{v'}\bt\cV_{v'}^*])
-\sum_{e'\in Q_1'}(\si_\la\t\si_\la)^*([\cV_{t'(e')}\bt\cV_{h'(e')}^*])\\
&\qquad{} =\sum_{v'\in Q_0'}\sum_{\substack{v,w\in Q_0\colon \\ \la_0(v)=\la_0(w)=v'}}[\cV_v\bt\cV_w^*]
-\sum_{e'\in Q_1'}\sum_{\substack{v,w\in Q_0\colon  \\ \la_0(v)=t'(e'),\, \la_0(w)=h'(e')}}[\cV_v\bt\cV_w^*]\\
&\qquad{} =\bigg(\sum_{v\in Q_0}[\cV_v\bt\cV_v^*]+\sum_{v\ne w\in Q_0\colon \la_0(v)=\la_0(w)}[\cV_v\bt\cV_w^*]\bigg)\\
&\qquad\quad{} -\bigg(\sum_{e\in Q_1}[\cV_{t(e)}\bt\cV_{h(e)}^*]
-\sum_{e\in Q_1\colon \not\exists (e,e')\in \la_1}[\cV_{t(e)}\bt\cV_{h(e)}^*]\bigg)=\big[(\cExt^\bu)^\vee\big]+[F],
\end{align*}
using \eq{wc5eq5} in the first step, and Definition \ref{wc5def6} in the second, and Definition \ref{wc5def6}(i)--(iii) to rewrite the sum $\sum_{e'\in Q_1'}\sum_{v,w}$ in the third, and \eq{wc5eq5} and \eq{wc5eq11} in the fourth. Definition \ref{wc2def5}(iv) then follows from the fourth equation of \eq{wc5eq6}. Part (v) follows as the vector bundles $\cV_v^*\ot\cV_w\ra\M$ descend to $\M^\pl$, as they have weight 0 for the $[*/\bG_m]$-action $\Psi$ on $\M$.

The additional condition in Definition \ref{wc2def6} that if $E\ne 0$ then $\Si_\la(E)\ne 0$ in $\modCQ'$ also holds. Thus Section~\ref{wc25} defines morphisms $\Om$, $\Om^\pl$ in \eq{wc2eq22}--\eq{wc2eq23}.
\end{dfn}

The next theorem relates enumerative invariants of quivers $Q$, $Q'$ linked by a morphism $\la\colon Q\ra Q'$. It will be proved in~Section~\ref{wc62}.

\begin{thm}\label{wc5thm2}
In the situation of Definitions {\rm\ref{wc5def6}} and {\rm\ref{wc5def7}}, let $(\tau',T',\le)$ be a weak stability condition on~$\modCQ'$, so that $(\tau,T,\le):=(\tau'\ci\la_*,T',\le)$ is a weak stability condition on $\modCQ$. Suppose that $Q,Q'$ have no oriented cycles. Then for all $\bs d\in\N^{Q_0}\sm\{0\}$ with $\la_*(\bs d)=\bs d'\in \N^{Q_0'}\sm\{0\}$, the invariants of Theorem~{\rm\ref{wc5thm1}} for $\modCQ,\modCQ'$ satisfy
\e
\prod_{v\in Q_0}\bs d(v)!\cdot \Om^\pl\bigl([\M_{\bs d}^\ss(\tau)]_\inv\bigr)=\prod_{v'\in Q'_0}\bs d'(v')!\cdot [\M_{\bs d'}^{\prime\ss}(\tau')]_\inv.
\label{wc5eq13}
\e
If $\la_0\colon Q_0\!\ra\! Q_0'$ is injective this simplifies to~$\Om^\pl([\M_{\bs d}^\ss(\tau)]_\inv)\!=\![\M_{\bs d'}^{\prime\ss}(\tau')]_\inv$.
\end{thm}

We will use Theorem \ref{wc5thm2} in Sections~\ref{wc64}--\ref{wc65} to prove Theorem~\ref{wc5thm1}(i).

\subsection{Pair invariants for quivers}\label{wc55}

Conjecture \ref{wc4conj1}(iii) and Definition \ref{wc4def1} described the method of pair invariants. We explain how to use this for quivers, proving a version of Conjecture~\ref{wc4conj1}(iii).

\begin{dfn}
\label{wc5def8}
Let $Q=(Q_0,Q_1,h,t)$ be a quiver with no oriented cycles. Choose $n_v\in\N$ for $v\in Q_0$. Define a new quiver $\ti Q=\big(\ti Q_0,\ti Q_1,\ti h,\ti t\big)$ to be $Q$ together with one extra vertex~$\iy$, so that $\ti Q_0=Q_0\amalg \{\iy\}$, and with $n_v$ extra edges $\overset{\iy}{\bu}\longra\overset{v}{\bu}$ for each $v\in Q_0$, so that $\ti Q_1=Q_1\amalg\coprod_{v\in Q_0}\{v\}\t\{1,\dots,n_v\}$. Then $\ti Q$ also has no oriented cycles. There is an obvious inclusion $i\colon \modCQ\hookra\modCtQ$ identifying $\modCQ$ with the full subcategory of $((V_v)_{v\in\ti Q_0},(\rho_e)_{e\in\ti Q_1})$ in $\modCtQ$ with $V_\iy=0$. In fact $i=\Si_\la$ in Definition~\ref{wc5def6} for $\la\colon Q\ra\ti Q$ given by $\la_0\colon v\mapsto v$ and~$\la_1=\{(e,e)\colon e\in Q_1\}$.

We can now apply Definition \ref{wc4def1} with $\A=\modCQ$ and $\B=\modCtQ$. For Definition~\ref{wc4def1}(i)--(vi), in (i) we take $i\colon \modCQ\hookra\modCtQ$ as above, and in (ii) we define an object $I=\big((W_v)_{v\in\ti Q_0},(\si_e)_{e\in\ti Q_1}\big)$ in $\modCtQ$ to have $W_\iy=\C$, and $W_v=0$ for all $v\in Q_0$, and $\si_e=0$ for all $e\in\ti Q_1$. For (iii) we use $\Z^{\ti Q_0}=\Z^{Q_0\amalg\{\iy\}}=\Z^{Q_0}\op\Z^{\{\iy\}}=\Z^{Q_0}\op\Z$. Parts (iv)--(vi) are obvious. The rest of Definition~\ref{wc4def1} then works in a straightforward way. Consider the map
\e
\big[{-},1_{H_0(\tiM_{(0,1)}^\pl)}\big]\colon \ H_*\big(\tiM_{(\al,0)}^\pl\big)\longra H_*\big(\tiM_{(\al,1)}^\pl\big).\label{wc5eq14}
\e

Using the explicit formula for $[\,,\,]$ in \cite{Joyc12} we can prove that \eq{wc5eq14} is injective for all $\al\in \N^{Q_0}\sm 0$ if and only if $n_v>0$ for all $v\in Q_0$. One way to do this is to consider the line bundle
\begin{equation*}
\cV_\iy^{-\sum_{v\in Q_0}\al(v)}\ot \bigot_{v\in Q_0}\det \cV_v\longra\tiM_{(\al,1)}.
\end{equation*}
This has weight 0 for the $[*/\bG_m]$-action on $\tiM_{(\al,1)}$, and so is the pullback of a line bundle $L\ra \tiM_{(\al,1)}^\pl$. The map \eq{wc5eq14} increases homological degree by $2d=2\big(\sum_{v\in Q_0}n_v\al(v)-1\big)$, where $d\ge 0$ as $\al\ne 0$ and $n_v>0$ for all $v\in Q_0$. We can show that if $\la\in H_*\big(\tiM_{(\al,0)}^\pl\big)$ then $\Pi_*\bigl(\big[\la,1_{H_0(\tiM_{(0,1)}^\pl)}\big]\cap c_1(L)^d\bigr)$ is a nonzero multiple of $\la$, where $\Pi\colon \tiM_{(\al,1)}^\pl\ra\tiM_{(\al,0)}^\pl$ is the forgetful morphism omitting $V_\iy$ and its edge maps, so $\big[\la,1_{H_0(\tiM_{(0,1)}^\pl)}\big]\ne 0$ if~$\la\ne 0$.

However, there is a problem: Theorem~\ref{wc5thm1}(i) only gives an explicit geometric definition of $[\M_{\bs d}^\ss(\tau)]_\inv$ when $\M_{\bs d}^\rst(\tau)=\M_{\bs d}^\ss(\tau)$ {\it and $\tau$ is a slope function}. As the weak stability conditions $(\tau_\pm,T_\pm,\le)$ on $\B=\modCtQ$ in Definition~\ref{wc4def1} are clearly not slope functions, Definition~\ref{wc4def1} as written does {\it not} give an inductive definition of $[\M_{\bs d}^\ss(\tau)]_\inv$ in terms of geometrically defined classes.

We now explain a way to get round this, by a variation of Definition~\ref{wc4def1} which uses only slope functions. Instead of a general weak stability condition $(\tau,T,\le)$ on $\A=\modCQ$, we choose a~slope function $\mu$ on $\modCQ$, defined by constants $(\mu_v)_{v\in Q_0}$. The analogue of the weak stability conditions  $(\tau_\pm,T_\pm,\le)$ on $\B=\modCtQ$ in Definition \ref{wc4def1} are slope functions $\ti\mu^{\bs d}_\pm$ on $\modCtQ$ which depend on a choice of $\bs d\in\N^{Q_0}\sm\{0\}$, defined by constants $\ti\mu^{\bs d}_{\pm,v}=\mu_v$ for $v\in Q_0\subset\ti Q_0$, and $\ti\mu^{\bs d}_{\pm,\iy}=\mu(\bs d)\pm\ep$ for $\iy\in\ti Q_0$, where $\ep>0$ is small.

The important properties of $\ti\mu^{\bs d}_\pm$ we need are:
\begin{itemize}\itemsep=0pt
\item[(a)] If $\bs e\in\N^{Q_0}\sm\{0\}$ then $\ti\mu^{\bs d}_\pm(\bs e,0)=\mu(\bs e)$.
\item[(b)] If $\bs e,\bs f\in\N^{Q_0}\sm\{0\}$ with $\bs d=\bs e+\bs f$ and $\mu(\bs e)<\mu(\bs f)$ then $\ti\mu^{\bs d}_+(\bs e,0)<\ti\mu^{\bs d}_+(\bs f,1)$, $\ti\mu^{\bs d}_+(\bs e,1)<\ti\mu^{\bs d}_+(\bs f,0)$, $\ti\mu^{\bs d}_-(\bs e,0)<\ti\mu^{\bs d}_-(\bs f,1)$, $\ti\mu^{\bs d}_-(\bs e,1)<\ti\mu^{\bs d}_-(\bs f,0)$.
\item[(c)] If $\bs e,\bs f\in\N^{Q_0}\sm\{0\}$ with $\bs d=\bs e+\bs f$ and $\mu(\bs e)=\mu(\bs f)$ then $\ti\mu^{\bs d}_+(\bs e,0)<\ti\mu^{\bs d}_+(\bs f,1)$ and $\ti\mu^{\bs d}_-(\bs e,1)<\ti\mu^{\bs d}_-(\bs f,0)$.
\item[(d)] $\ti\mu^{\bs d}_+(0,1)>\mu(\bs d)$ and $\ti\mu^{\bs d}_-(0,1)<\mu(\bs d)$.
\end{itemize}
Here $\ep>0$ needs to be small to ensure $\ti\mu^{\bs d}_+(\bs e,1)<\ti\mu^{\bs d}_+(\bs f,0)$ and $\ti\mu^{\bs d}_-(\bs e,0)<\ti\mu^{\bs d}_-(\bs f,1)$ in (b). Then $\mu$, $\mu_+^{\bs d}$, $\mu_-^{\bs d}$ satisfy all the same inequalities for $\tau$, $\tau_+$, $\tau_-$ used to prove \eq{wc4eq4} with $\al=\bs d$. So combining~\eq{wc4eq4} with Theorem~\ref{wc5thm1}(i) and Definition~\ref{wc4def1}(a),~(b) proves that
\ea
&\ti\io_*\bigl(\big[\tiM_{(\bs d,1)}^\ss\big(\ti\mu^{\bs d}_+\big)\big]_\fund\bigr)=
\label{wc5eq15}\\
&\sum_{\substack{n\ge 1,\,\bs d_1,\dots,\bs d_n\in
\N^{Q_0}\sm\{0\}\colon \\ \bs d_1+\cdots+\bs d_n=\bs d, \\
\mu(\bs d_i)=\mu(\bs d),\, i=1,\dots,n }} \!\!\!\!\!\!\!\!\!\!
\frac{(-1)^n}{n!}\cdot\bigl[\bigl[\cdots\bigl[[1_{H_0(\tiM_{(0,1)}^\pl)},
i_*^\pl\big(\big[\M_{\bs d_1}^\ss(\mu)\big]_\inv\big)\bigr],\dots\bigr],i_*^\pl\big(\big[\M_{\bs d_n}^\ss(\mu)\big]_\inv\big)\bigr].
\nonumber
\ea

Here $\big[\tiM_{(\bs d,1)}^\ss\big(\ti\mu^{\bs d}_+\big)\big]_\fund$ is the fundamental class of the smooth projective $\C$-scheme $\tiM_{(\bs d,1)}^\rst\!\big(\ti\mu^{\bs d}_+\big) \allowbreak =\tiM_{(\bs d,1)}^\ss\big(\ti\mu^{\bs d}_+\big)$, and is geometrically defined. As for \eq{wc4eq5}, we can rewrite \eq{wc5eq15} as
\e
\ti\io_*\bigl(\big[\tiM_{(\bs d,1)}^\ss\big(\ti\mu^{\bs d}_+\big)\big]_\fund\bigr)=\bigl[i_*^\pl\bigl(\big[\M_{\bs d}^\ss(\mu)\big]_\inv\bigr),1_{H_0(\tiM_{(0,1)}^\pl)}\bigr]+\text{lower order terms},
\label{wc5eq16}
\e
where the operation $\big[{-},1_{H_0(\tiM_{(0,1)}^\pl)}\big]$ is injective if $n_v>0$ for all $v$ in $Q_0$, and $i_*^\pl$ is an isomorphism. Then we can use \eq{wc5eq16} to determine $\big[\M_{\bs d}^\ss(\mu)\big]_\inv$ uniquely in terms of $\big[\tiM_{(\bs d,1)}^\ss\big(\ti\mu^{\bs d}_+\big)\big]_\fund$ and $[\M_{\bs e}^\ss(\mu)]_\inv$ for $\md{\bs e}<\md{\bs d}$, by induction on increasing~$\md{\bs d}$.
\end{dfn}

\section{Proofs of Theorems \ref{wc5thm1} and \ref{wc5thm2}}\label{wc6}

We now prove Theorems \ref{wc5thm1} and \ref{wc5thm2}. Firstly, in Section~\ref{wc61}, we prove that there are unique classes $[\M_{\bs d}^\ss(\tau)]_\inv$ satisfying Theorem~\ref{wc5thm1}(ii),~(iii), but without yet showing that they satisfy Theorem~\ref{wc5thm1}(i). Then Section~\ref{wc62} shows that Theorem~\ref{wc5thm2} holds for these classes $\big[\M_{\bs d}^\ss(\tau)\big]_\inv$. Sections \ref{wc63}--\ref{wc66} prove that Theorem~\ref{wc5thm1}(i) also holds, using Theorem~\ref{wc5thm2} as a tool.

\subsection{Proof of Theorem \ref{wc5thm1}(ii), (iii)}\label{wc61}

\begin{prop}\label{wc6prop1}
In the situation of Theorem~{\rm \ref{wc5thm1}}, there exist unique classes $\big[\M_{\bs d}^\ss(\tau)\big]_\inv$ in $\check H_0\big(\M_{\bs d}^\pl\big)$ satisfying Theorem~{\rm\ref{wc5thm1}(ii)},~{\rm (iii)}.
\end{prop}

\begin{proof}Fix an increasing slope function $\mu$ on $\modCQ$, which is possible as in Definition~\ref{wc5def5} since~$Q$ has no oriented cycles. Then Theorem \ref{wc5thm1}(iii) determines the classes $\big[\M_{\bs d}^\ss(\mu)\big]_\inv$. Let $(\tau,T,\le)$ be any weak stability condition on $\modCQ$. Define classes $\big[\M_{\bs d}^\ss(\tau)\big]_\inv\in\check H_0\big(\M_{\bs d}^\pl\big)$ for all $\bs d\in\N^{Q_0}\sm\{0\}$ by the equivalent equations \eq{wc5eq7}--\eq{wc5eq8} with $\mu,\tau$ in place of~$\tau$,~$\ti\tau$. Note that when $\tau=\mu$, equation \eq{wc3eq12} implies that this recovers the classes $\big[\M_{\bs d}^\ss(\mu)\big]_\inv$ in \eq{wc5eq9}, so Theorem~\ref{wc5thm1}(iii) holds for this fixed increasing slope function~$\mu$.

We claim that these classes $\big[\M_{\bs d}^\ss(\tau)\big]_\inv$ satisfy Theorem~\ref{wc5thm1}(ii). To see this, let $(\tau,T,\le)$, $\big(\ti\tau,\ti T,{\le}\big)$ be weak stability conditions on $\modCQ$. Then
\begin{gather}
\big[\M_{\bs d}^\ss(\tau)\big]_\inv=
\sum_{\substack{n\ge 1,\,\bs e_1,\dots,\bs e_n\in
\N^{Q_0}\sm\{0\}\colon \\ \bs e_1+\cdots+\bs e_n=\bs d } }
\begin{aligned}[t]
U(\bs e_1,&\dots,\bs e_n;\mu,\tau)\cdot[\M_{\bs e_1}^\ss(\mu)]_\inv *\\
&
[\M_{\bs e_2}^\ss(\mu)]_\inv*\cdots *[\M_{\bs e_n}^\ss(\mu)]_\inv,
\end{aligned}
\label{wc6eq1}\\
\big[\M_{\bs d}^\ss(\ti\tau)\big]_\inv=
\sum_{\substack{n\ge 1,\,\bs e_1,\dots,\bs e_n\in
\N^{Q_0}\sm\{0\}:\\ \bs e_1+\cdots+\bs e_n=\bs d}}
\begin{aligned}[t]
U(\bs e_1,&\dots,\bs e_n;\mu,\ti\tau)\cdot[\M_{\bs e_1}^\ss(\mu)]_\inv *\\
&
[\M_{\bs e_2}^\ss(\mu)]_\inv*\cdots *[\M_{\bs e_n}^\ss(\mu)]_\inv.
\end{aligned}\label{wc6eq2}
\end{gather}
To verify \eq{wc5eq8} for $(\tau,T,\le)$, $\big(\ti\tau,\ti T,{\le}\big)$, substitute \eq{wc6eq2} into the left-hand side of~\eq{wc5eq8}, and substitute~\eq{wc6eq1} with $\bs d_i$ in place of $\bs d$ into each term $\big[\M_{\bs d_i}^\ss(\tau)\big]_\inv$ on the right-hand side. Multiplying out, we see that the coefficients of $[\M_{\bs e_1}^\ss(\mu)]_\inv*\cdots *[\M_{\bs e_n}^\ss(\mu)]_\inv$ on each side are equal by equation~\eq{wc3eq13} of Theorem~\ref{wc3thm3} with $\mu$, $\tau$, $\ti\tau$, $\bs e_i$ in place of $\tau$, $\hat\tau$, $\ti\tau$, $\al_i$. Thus \eq{wc5eq8} holds, so~\eq{wc5eq7} also holds as it is equivalent by Theorem~\ref{wc3thm5}. This proves Theorem~\ref{wc5thm1}(ii).

Next, suppose for a contradiction that $\ti\mu$ is another increasing slope function, and \eq{wc5eq9} fails for $\ti\mu$ and $\bs{\ti d}\in\N^{Q_0}\sm\{0\}$. Let $\mu$, $\ti\mu$ be defined using constants $(\mu_v)_{v\in Q_0}$, $\ab(\ti\mu_v)_{v\in Q_0}$ in $\R^{Q_0}$. Observe that the condition on $(\mu_v)_{v\in Q_0}$ for $\mu$ to be increasing in Definition \ref{wc5def5} defines an open convex subset $\R^{Q_0}_{>}$ of $\R^{Q_0}$. Choose a generic smooth path $\big(\mu_v^t\big)_{v\in Q_0}$ for $t\in[0,1]$ in $\R^{Q_0}_{>}$ with $\mu_v^0=\mu_v$ and $\mu_v^1=\ti\mu_v$ for $v\in Q_0$. This defines a generic smooth path $\mu^t$ of increasing slope functions on $\modCQ$ with $\mu^0=\mu$ and $\mu^1=\ti\mu$. Define
\begin{equation*}
t_0=\inf\bigl\{t\in[0,1]\colon \text{\eq{wc5eq9} fails for $\big[\M_{\bs d}^\ss\big(\mu^t\big)\big]_\inv$ for some $\bs d$ with $0<\bs d\le\bs{\ti d}$}\bigr\}.
\end{equation*}

Divide into cases:
\begin{itemize}\itemsep=0pt
\item[(a)] \eq{wc5eq9} holds for $\big[\M_{\bs d}^\ss\big(\mu^{t_0}\big)\big]_\inv$ whenever $0<\bs d\le\bs{\ti d}$. Then let $\ep>0$ be small and $0<\bs d\le\bs{\ti d}$ such that \eq{wc5eq9} fails for $\big[\M_{\bs d}^\ss\big(\mu^{t_0+\ep}\big)\big]_\inv$, and set $\mu'=\mu^{t_0}$ and $\mu''=\mu^{t_0+\ep}$.
\item[(b)] \eq{wc5eq9} does not hold for $\big[\M_{\bs d}^\ss\big(\mu^{t_0}\big)\big]_\inv$, some $0<\bs d\le\bs{\ti d}$. Then let $\ep>0$ be small and set $\mu'=\mu^{t_0-\ep}$ and $\mu''=\mu^{t_0}$.
\end{itemize}
In both cases, $\mu'$, $\mu''$ are increasing such that \eq{wc5eq9} holds for $\mu'$ and all $\bs d'\le\bs{\ti d}$, but fails for~$\mu''$,~$\bs d$. Furthermore, as~$\ep$ is small, in case~(a)~$\mu'$ dominates~$\mu''$ in the sense of Definition~\ref{wc3def3}, and in case~(b)~$\mu''$ dominates~$\mu'$. This holds because any slope function~$\mu_1$ on $\modCQ$ dominates all slope functions~$\mu_2$ in a sufficiently small neighbourhood of $\mu_1$ in the space of slope functions~$\R^{Q_0}$.

Consider \eq{wc5eq8} with $\mu'$, $\mu''$ in place of $\tau$, $\ti\tau$. To get a nonzero term on the right-hand side, we must have $\bs d_i=\de_{v_i}$ for $v_i\in Q_0$ and $i=1,\dots,n$, as \eq{wc5eq9} holds for~$\mu'$. Also, if $U(\bs d_1,\dots,\bs d_n;\mu',\mu'') \ne 0$ then~\eq{wc3eq14} gives $\mu'(\de_{v_1})=\cdots=\mu'(\de_{v_n})$ in case~(a), as~$\mu'$ dominates~$\mu''$, and $\mu''(\de_{v_1})=\cdots=\mu''(\de_{v_n})$ in case (b). As $\mu'$, $\mu''$ are increasing, Definition~\ref{wc5def5} implies that there are no edges in $Q$ joining $v_i$, $v_j$ for any $i,j=1,\dots,n$. The definition of the Lie bracket on $\check H_0\big(\M^\pl\big)$ and~\eq{wc5eq9} for~$\mu'$ then implies that
\begin{equation*}
\bigl[\big[\M_{\de_{v_i}}^\ss(\mu')\big]_\inv,\big[\M_{\de_{v_j}}^\ss(\mu')\big]_\inv\bigr]=0.
\end{equation*}
Thus rewriting \eq{wc5eq8} in the form \eq{wc5eq7}, we see that every term on the right-hand side is zero, so $\big[\M_{\bs d}^\ss(\mu'')\big]_\inv=0$, a contradiction, as \eq{wc5eq9} fails for $\mu''$, $\bs d$. Therefore Theorem~\ref{wc5thm1}(iii) holds. This completes the proof.
\end{proof}

\subsection{Proof of Theorem \ref{wc5thm2}}\label{wc62}

Given Sections~\ref{wc63}--\ref{wc66}, the next proposition is equivalent to Theorem~\ref{wc5thm2}.

\begin{prop}\label{wc6prop2}
In the situation of Theorem {\rm\ref{wc5thm2}}, the classes $\big[\M_{\bs d}^\ss(\tau)\big]_\inv$ in $\check H_0\big(\M_{\bs d}^\pl\big)$ defined in Proposition~{\rm \ref{wc6prop1}} satisfy~\eq{wc5eq13}.
\end{prop}

\begin{proof} We divide into three cases, depending on the morphism $\la\colon Q\ra Q'$:
\begin{itemize}\itemsep=0pt
\item[(a)] We have $Q_0=Q_0'$ and $\la_0=\id_{Q_0}$.
\item[(b)] The projection $\pi_{Q_1}\colon \la_1\ra Q_1$ in Definition~\ref{wc5def6}(iii) is a bijection.
\item[(c)] The general case.
\end{itemize}
Now any $\la=(\la_0,\la_1)\colon Q\ra Q'$ may be written $\la=\la''\ci\la'$, for $\la'\colon Q\ra\ti Q$ as in (a) and $\la''\colon \ti Q\ra Q'$ as in (b), where $\ti Q$ is $Q$ with edges $e$ in $Q_1\sm\Im\pi_{Q_1}$ deleted, and $\la'=(\id_{Q_0},\{(e,e)\colon e\in\ti Q_1\})$, $\la''=(\la_0,\la_1)$. Clearly \eq{wc5eq13} for $\la'$, $\la''$ imply \eq{wc5eq13} for $\la=\la''\ci\la'$. Thus~(a),~(b) imply~(c).

For both (a), (b), we will give different constructions of slope functions $\mu'$ and $\mu=\mu'\ci\la_*$ on $\modCQ'$, $\modCQ$ with both $\mu$, $\mu'$ increasing, and prove \eq{wc5eq13} holds for $\mu$, $\mu'$. Then we complete the proofs of (a), (b) together.

For (a), note that the map $\mu'\mapsto\mu=\mu'\ci\la_*$ is a bijection from slope functions on $\modCQ$ to slope functions on $\modCQ'$. Choose $\mu'$ such that $\mu=\mu'\ci\la_*$ is an increasing slope function on $\modCQ$. Then $\mu'$ is also increasing, since edges of $Q'$ are also edges of $Q$. Hence \eq{wc5eq9} holds for $\mu$ on $\modCQ$ and for $\mu'$ on $\modCQ'$. To verify~\eq{wc5eq13} with $\mu$, $\mu'$ in place of $\tau$, $\tau'$, note that if $\bs d=\bs d'=\de_v$ then $\big[\M_{\bs d}^\ss(\tau)\big]_\inv=\big[\M_{\bs d'}^{\prime\ss}(\tau')\big]_\inv=1$ and $\Om^\pl=\id$ on $H_*\big(\M_{\bs d}^\pl\big)$ as $\rank G^\pl_{\bs d}=\xi(\de_v,\de_v)=0$, and if $\bs d\ne \de_v$ then both sides of \eq{wc5eq13} are zero by~\eq{wc5eq9}. So~\eq{wc5eq13} holds for~$\mu$, $\mu'$.

For (b), let $\mu'$ be an increasing slope function on $\modCQ'$ with constants $(\mu'_{v'})_{v'\in Q_0'}$, and set $\mu=\mu'\ci\la_*$. Then $\mu$ is a slope function on $\modCQ$, with constants $(\mu_v)_{v\in Q_0}$ for $\mu_v=\mu'_{\la_0(v)}$. If $\overset{v}{\bu}\,{\buildrel e\over \longra}\,\overset{w}{\bu}$ is an edge in $Q$ there exists unique $e'\in Q_1'$ with $(e,e')\in\la_1$, and then $\overset{\la_0(v)}{\bu}\,{\buildrel e'\over \longra}\,\overset{\la_0(w)}{\bu}$ is an edge in~$Q'$. Hence as $\mu'$ is increasing we have $\mu_v=\mu'_{\la_0(v)}<\mu'_{\la_0(w)}=\mu_w$, so $\mu$ is also increasing.

If $\bs d=\de_v$ for some $v\in Q_0$, so that $\bs d'=\de_{\la_0(v)}$, then $\big[\M_{\bs d}^\ss(\tau)\big]_\inv=\big[\M_{\bs d'}^{\prime\ss}(\tau')\big]_\inv=1$ by \eq{wc5eq9} and we can show \eq{wc5eq13} holds in a similar way to~(a). If $\bs d\ne\de_v$ for any $v\in Q_0$ then $\bs d'\ne\de_{v'}$ for any $v'\in Q_0'$, so $\big[\M_{\bs d}^\ss(\tau)\big]_\inv=\big[\M_{\bs d'}^{\prime\ss}(\tau')\big]_\inv=0$ by \eq{wc5eq9} as $\mu$, $\mu'$ are increasing, and~\eq{wc5eq13} holds trivially.

Note that in both cases (a), (b), in \eq{wc5eq13} for $\mu$, $\mu'$ either both sides are zero, or $\bs d$, $\bs d'$ are {\it binary} in the sense of Definition~\ref{wc6def1} below, so the mysterious factors $\prod_v\bs d(v)!$, $\prod_{v'}\bs d'(v')!$ in~\eq{wc5eq13} are~1.

Now in both cases (a), (b), let $(\tau',T',\le)$ be any weak stability condition on $\modCQ$, set $(\tau,T,\le)=(\tau'\ci\la_*,T',\le)$, and let $\bs d\in\N^{Q_0}\sm\{0\}$ with $\la_*(\bs d)=\bs d'\in\N^{Q_0'}\sm\{0\}$. Then
\begin{gather}
 \Om^\pl\bigl([\M_{\bs d}^\ss(\tau)]_\inv\bigr)
 =\sum_{\substack{n\ge 1,\,\bs d_1,\dots,\bs d_n\in
\N^{Q_0}\sm\{0\}\colon \\ \bs d_1+\cdots+\bs d_n=\bs d, \\
\text{set $\bs d_i'=\la_*(\bs d_i)$} }}
\begin{aligned}[t]
U(\bs d_1,&\dots,\bs d_n;\mu,\tau)\cdot\frac{\prod_{i=1,\dots,n,\, v'\in Q'_0}\bs d'_i(v')!}{\prod_{i=1,\dots,n,\, v\in Q_0}\bs d_i(v)!}\\
&[\M_{\bs d_1'}^{\prime\ss}(\mu')]_\inv *\cdots *[\M_{\bs d_n'}^{\prime\ss}(\mu')]_\inv
\end{aligned}
\nonumber\\
\qquad{} =\sum_{\substack{n\ge 1,\,\bs d'_1,\dots,\bs d'_n\in
\N^{Q_0'}\sm\{0\}\colon \\ \bs d'_1+\cdots+\bs d'_n=\bs d' }}
\begin{aligned}[t]
U(\bs d'_1,&\dots,\bs d'_n;\mu',\tau')\cdot[\M_{\bs d_1'}^{\prime\ss}(\mu')]_\inv *\cdots *[\M_{\bs d_n'}^{\prime\ss}(\mu')]_\inv \\[-3pt]
&{}\cdot\raisebox{-6pt}{$\biggl[$}\,\sum_{\substack{\bs d_1,\dots,\bs d_n\in
\N^{Q_0}\colon \\ \bs d_1+\cdots+\bs d_n=\bs d, \,\la_*(\bs d_i)=\bs d_i' }} \raisebox{-6pt}{$\displaystyle\frac{\prod_{i=1,\dots,n,\, v'\in Q'_0}\bs d'_i(v')!}{\prod_{i=1,\dots,n,\, v\in Q_0}\bs d_i(v)!} \biggr]$}
\end{aligned}
\nonumber\\
\qquad{} =\sum_{\substack{n\ge 1,\,\bs d'_1,\dots,\bs d'_n\in
\N^{Q_0'}\sm\{0\}\colon \\ \bs d'_1+\cdots+\bs d'_n=\bs d'} }
\begin{aligned}[t] U(\bs d'_1,& \dots,\bs d'_n;\mu',\tau')\cdot[\M_{\bs d_1'}^{\prime\ss}(\mu')]_\inv *\cdots *[\M_{\bs d_n'}^{\prime\ss}(\mu')]_\inv\\
& {}\cdot\Biggl[\frac{\prod_{v'\in Q'_0}\bs d'(v')!}{\prod_{v\in Q_0}\bs d(v)!}\Biggr].
\end{aligned}\label{wc6eq3}
\end{gather}
Here the first step holds by applying the Lie algebra morphism $\Om^\pl$ to \eq{wc5eq8} for $\modCQ$ with~$\mu$,~$\tau$ in place of $\tau$, $\ti\tau$ and using \eq{wc5eq13} for $\mu$, $\mu'$. In the second we rewrite the sums and use $U(\bs d'_1,\dots,\bs d'_n;\mu',\tau')=U(\bs d_1,\dots,\bs d_n;\mu,\tau)$ as $\mu=\mu'\ci\la_*$, $\tau=\tau'\ci\la_*$ and $\bs d_i'=\la_*(\bs d_i)$.

The third step holds by a combinatorial identity showing the brackets $[\cdots]$ are equal. It may be written as the product over $v'\in Q_0'$ of the simpler identity
\begin{equation*}
\sum_{\substack{\bs d_1^{v'},\dots,\bs d_n^{v'}\in
\N^{\la_0^{-1}(v')}\colon \\ \bs d_1^{v'}+\cdots+\bs d_n^{v'}=\bs d^{v'}, \,
\sum_{v\in \la_0^{-1}(v')}\bs d_i^{v'}(v)=\bs d_i'(v') }}  \frac{\prod_{i=1,\dots,n}\bs d'_i(v')!}{\prod_{i=1,\dots,n,\, v\in \la_0^{-1}(v')}\bs d_i^{v'}(v)!}=
\frac{\bs d'(v')!}{\prod_{v\in \la_0^{-1}(v')}\bs d^{v'}(v)!},
\end{equation*}
writing $\bs d^{v'}$, $\bs d_i^{v'}$ for the restrictions of $\bs d$, $\bs d_i$ to $\la_0^{-1}(v')\subseteq Q_0$. Equivalently,
\begin{equation*}
\sum_{\substack{\bs d_1^{v'},\dots,\bs d_n^{v'}\in
\N^{\la_0^{-1}(v')}\colon \\ \bs d_1^{v'}+\cdots+\bs d_n^{v'}=\bs d^{v'}, \,
\sum_{v\in \la_0^{-1}(v')}\bs d_i^{v'}(v)=\bs d_i'(v') }}  \frac{\prod_{v\in \la_0^{-1}(v')}\bs d^{v'}(v)!}{\prod_{i=1,\dots,n,\, v\in \la_0^{-1}(v')}\bs d_i^{v'}(v)!}=
\frac{\bs d'(v')!}{\prod_{i=1,\dots,n}\bs d'_i(v')!},
\end{equation*}
which follows from considering the number of ways of dividing a set of size $\sum_{v\in \la_0^{-1}(v')}\bs d^{v'}(v)=\bs d'(v')$ into $n$ subsets of size $\bs d'_i(v')$ for $i=1,\dots,n$. Equation~\eq{wc6eq3} implies \eq{wc5eq13} for $\tau$, $\tau'$. This proves~(a), (b), and~(c) follows as above.	
\end{proof}

\subsection[\texorpdfstring{Proof of Theorem \ref{wc5thm1}(i) for $d$ binary, $Q$ a tree}{Proof of Theorem \ref{wc5thm1}(i) for d binary, Q a tree}]{Proof of Theorem \ref{wc5thm1}(i) for $\boldsymbol{\bs d}$ binary, $\boldsymbol{Q}$ a tree}
\label{wc63}

We will break the proof of Theorem \ref{wc5thm1}(i) into four cases in Sections~\ref{wc63}--\ref{wc66}, depending on the following definition.

\begin{dfn}\label{wc6def1}
Work in the situation of Theorem \ref{wc5thm1}. We call a dimension vector $\bs d\in\N^{Q_0}\sm\{0\}$ {\it binary} if $\bs d(v)\in\{0,1\}$ for all $v\in Q_0$. The {\it support} $\supp\bs d$ of $\bs d$ is the subquiver $Q'\subseteq Q$ containing all vertices $v\in Q_0$ with $\bs d(v)>0$, and all edges linking two such vertices. We say that $\supp\bs d$ {\it is a tree} if $Q'$ is connected and simply-connected.	

Let $\mu$ be a slope function on $\modCQ$, and $\bs d\in\N^{Q_0}\sm\{0\}$. We call the pair $\mu$, $\bs d$ {\it generic} if whenever $\bs d=\bs e+\bs f$ for $\bs e, \bs f\in\N^{Q_0}\sm\{0\}$ we have $\mu(\bs e)\ne\mu(\bs f)$. This implies that there are no strictly $\mu$-semistable objects in class $\bs d$, so $\M_{\bs d}^\rst(\mu)=\M_{\bs d}^\ss(\mu)$. Note that $\mu$, $\bs d$ cannot be generic if $\bs d=n\bs e$ for~$n>1$.
\end{dfn}

The proof of the next proposition uses the Donaldson--Thomas theory of quivers in Joyce--Song \cite[Section~7]{JoSo}, which satisfy a wall-crossing formula like~\eq{wc5eq7}.

\begin{prop}\label{wc6prop3}
The classes $\big[\M_{\bs d}^\ss(\mu)\big]_\inv$ in $\check H_0\big(\M_{\bs d}^\pl\big)$ defined in Proposition~{\rm \ref{wc6prop1}} satisfy Theo\-rem~{\rm \ref{wc5thm1}(i)} when $\bs d$ is binary and~$\supp\bs d$ is a~tree.
\end{prop}

\begin{proof} Suppose $\bs d$ is binary, $\supp\bs d$ is a tree, and~$\mu$ is a slope function on $\modCQ$ with $\M^\rst_{\bs d}(\mu)=\M^\ss_{\bs d}(\mu)$. Now $\M_{\bs d}^\pl$ is a smooth Artin $\C$-stack with $\dim_\C\M_{\bs d}^\pl=1-\chi_Q(\bs d,\bs d)$, for $\chi_Q$ as in \eq{wc5eq4}. As~$\bs d$ is binary and $\supp\bs d$ is a tree, we see that $\chi_Q(\bs d,\bs d)=1$, giving $\dim_\C\M_{\bs d}^\pl=0$, so as $\M^\rst_{\bs d}(\mu)$ is open in $\M_{\bs d}^\pl$, it is smooth of dimension 0. Also $\M^\rst_{\bs d}(\mu)=[R^\rst_{\bs d}(\mu)/\PGL_{\bs d}]$ with $R^\rst_{\bs d}(\mu)\subseteq R_{\bs d}\cong\bA^n$ open, so $R^\rst_{\bs d}(\mu)$ and hence $\M^\rst_{\bs d}(\mu)$ are connected. This implies that
\e
\text{$\M^\rst_{\bs d}(\mu)=\M^\ss_{\bs d}(\mu)$ is either a point $*$ or empty}.
\label{wc6eq4}
\e

Now let $\ti\mu$ be an increasing slope function on $\modCQ$. Equations \eq{wc3eq19} and \eq{wc5eq7} with $\ti\mu$, $\mu$ in place of $\tau$, $\ti\tau$ and \eq{wc5eq9} for $\ti\mu$ yield
\begin{gather}
\bar\ep{}^{\bs d}(\mu)=
\sum_{\substack{\bs d_1,\dots,\bs d_n\in
\N^{Q_0}\sm\{0\}\colon n=\md{\bs d},\\ \bs d_1+\cdots+\bs d_n=\bs d,\, \bs d_i=\de_{v_i},\; v_i\in Q_0 } }
\begin{aligned}[t]
\ti U(\bs d_1,&\dots,\bs d_n;\ti\mu,\mu)\cdot\bigl[\bigl[\cdots\bigl[\bar\ep{}^{\bs d_1}(\ti\mu),\\
&
\bar\ep{}^{\bs d_2}(\ti\mu)\bigr],\dots\bigr],\bar\ep{}^{\bs d_n}(\ti\mu)\bigr],
\end{aligned}
\label{wc6eq5}\\
[\M_{\bs d}^\ss(\mu)]_\inv=
\sum_{\substack{\bs d_1,\dots,\bs d_n\in
\N^{Q_0}\sm\{0\}\colon n=\md{\bs d},\\ \bs d_1+\cdots+\bs d_n=\bs d,\; \bs d_i=\de_{v_i},\, v_i\in Q_0 }}  \begin{aligned}[t]
\ti U(\bs d_1,&\dots,\bs d_n;\ti\mu,\mu)\cdot\bigl[\bigl[\cdots\bigl[1_{H_0(\M_{\bs d_1}^\pl)},\\
&
1_{H_0(\M_{\bs d_2}^\pl)}\bigr],\dots\bigr],1_{H_0(\M_{\bs d_n}^\pl)}\bigr].
\end{aligned}
\label{wc6eq6}
\end{gather}
Here we use Proposition \ref{wc5prop1} and \eq{wc5eq9} to deduce that the only nonzero terms on the right-hand sides of \eq{wc3eq19} and \eq{wc5eq7} are when $\bs d_i=\de_{v_i}$ for $i=1,\dots,n$, and then $n=\md{\bs d}:=\sum_{v\in Q_0}\bs d(v)$. Equations \eq{wc6eq5} and \eq{wc6eq6} are in the Lie algebras $\SFai(\M)$ and~$\check H_0\big(\M^\pl\big)$.

Now Joyce and Song \cite[Section~7]{JoSo} define invariants $\bar{DT}{}_Q^{\bs e}(\mu)$ for quivers $Q$ (this is the special case when the superpotential is $W=0$). In \cite[Section~7.3]{JoSo} they define an explicit Lie algebra~$L(Q)$ over~$R$ (they take $R=\Q$ and write $\ti L(Q)$), with basis given by symbols $\la^{\bs e}$ for $\bs e\in\N^{Q_0}$, and Lie bracket
\begin{equation*}
\big[\la^{\bs e}, \la^{\bs f}\big]=(-1)^{\bar\chi(\bs e, \bs f)} \bar\chi(\bs e, \bs f) \la^{\bs e + \bs f},
\end{equation*}
where $\bar\chi$ is the {\it anti-symmetric Euler form} $\bar\chi(\bs e,\bs f):=\chi_Q(\bs e,\bs f)-\chi_Q(\bs f,\bs e)$. They define a Lie algebra morphism $\Psi\colon \SFai(\M)\ra L(Q)$ (this is difficult), for $\SFai(\M)$ as in Section~\ref{wc32}, and they define $\bar{DT}{}_Q^{\bs e}(\mu)\in\Q\subseteq R$ by $\Psi(\bar\ep{}^{\bs e}(\mu))=-\bar{DT}{}_Q^{\bs e}(\mu)\la^{\bs e}$, for $\bar\ep{}^{\bs e}(\mu)$ as in Section~\ref{wc34}.

Define Lie algebra ideals
\begin{gather*}
L(Q)_{\text{non-bin}}= \bigop_{\bs e \ \text{non-binary}} R\cdot \la^{\bs e} \subseteq L(Q), \\
\check H_0(\M^\pl)_{\text{non-bin}}= \bigop_{\bs e \ \text{non-binary}} \check H_0\big(\M^\pl_{\bs e}\big) \subseteq  \check H_0\big(\M^\pl_Q\big),
\end{gather*}
noting that if $\bs e$ or $\bs f$ is non-binary then so is $\bs e+\bs f$. Define Lie subalgebras
\begin{gather*}
L(Q)^{\bs d}_{\text{conn}}= \bigop_{\substack{\text{$\bs e$ connected support}, \\ \supp(\bs e)\subseteq\supp(\bs d)}} R\cdot \la^{\bs e}  \subseteq L(Q), \\
\check H_0(\M^\pl)^{\bs d}_{\text{conn}}=  \bigop_{\substack{\text{$\bs e$ connected support}, \\ \supp(\bs e)\subseteq\supp(\bs d)}} \check H_0\big(\M^\pl_{\bs e}\big)  \subseteq \check H_0\big(\M^\pl\big).
\end{gather*}
These are Lie subalgebras as if $\supp(\bs e)$, $\supp(\bs f)$ are disjoint and not connected by any edges then the Lie brackets between $\la^{\bs e}$, $\la^{\bs f}$ and $\check H_0\big(\M^\pl_{\bs e}\big)$, $\check H_0\big(\M^\pl_{\bs f}\big)$ are zero. The quotients
\begin{gather*}
L(Q)_{\text{bin}}^{\bs d} = L(Q)^{\bs d}_{\text{conn}}/ (L(Q)^{\bs d}_{\text{conn}} \cap L(Q)_{\text{non-bin}}), \\
\check H_0\big(\M^\pl\big)^{\bs d}_{\text{bin}} = \check H_0\big(\M^\pl\big)^{\bs d}_{\text{conn}}/\big(\check H_0\big(\M^\pl\big)^{\bs d}_{\text{conn}} \cap \check H_0\big(\M^\pl\big)_{\text{non-bin}}\big)
\end{gather*}
are Lie algebras.

Define an $R$-linear map $\Up\colon L(Q)^{\bs d}_{\text{bin}} \ra \check H_0\big(\M^\pl\big)^{\bs d}_{\text{bin}}$ by $\la^{\bs e} \mapsto -1_{H_0(\M_{\bs e}^\pl)}$. We will show $\Up$ is a~Lie algebra morphism, as $\Up\big(\big[\la^{\bs e},\la^{\bs f}\big]\big)=\big[\Up(\la^{\bs e}),\Up\big(\la^{\bs f}\big)\big]$, that is
\e
\begin{split}
&\big[1_{H_0(\M_{\bs e}^\pl)}, 1_{H_0(\M_{\bs f}^\pl)}\big] = -(-1)^{\bar\chi(\bs e, \bs f)} \bar\chi(\bs e, \bs f) \cdot 1_{H_0(\M_{\bs e+\bs f}^\pl)} \\
&\qquad \mod \check H_0\big(\M^\pl\big)^{\bs d}_{\text{conn}} \cap \check H_0\big(\M^\pl\big)_{\text{non-bin}},
\end{split}
\label{wc6eq7}
\e
for $\bs e, \bs f\in\N^{Q_0}\sm\{0\}$ binary with $\supp(\bs e)$, $\supp(\bs f)$ connected and in $\supp(\bs d)$.

Using the definition of the Lie bracket on $\check H_0\big(\M^\pl\big)$ via Proposition~\ref{wc2prop1} and $\check H_*\big(\M^\pl\big)=\hat H_{*+2}(\M)/D\big(\hat H_*(\M)\big)$, and the definition~\eq{wc2eq12} of the vertex algebra structure on $\hat H_*(\M)$, we find that
\e
\begin{split}
&\big[1_{H_0(\M_{\bs e}^\pl)}, 1_{H_0(\M_{\bs f}^\pl)}\big]\\
&=\sum_{\substack{ i\ge 0\colon \\ i\ge\chi(\bs e,\bs f)+1}}\begin{aligned}[t]&\frac{(-1)^{\chi_Q(\bs e,\bs f)}}{(i-\chi(\bs e,\bs f)-1)!}H_*(\Phi)\ci \big(D^{i-\chi(\bs e,\bs f)-1}\ot\id\big)\\
&\bigl((1_{H_0(\M_{\bs e})}\bt 1_{H_0(\M_{\bs f})})\cap c_i(\Th^\bu)\bigr)
\end{aligned}
\\
&\quad {}+D\big(\hat H_0(\M)\big), \quad \mod \check H_0\big(\M^\pl\big)_{\text{conn}} \cap \check H_0\big(\M^\pl\big)_{\text{non-bin}}.
\end{split}
\label{wc6eq8}
\e
Equation \eq{wc6eq8} implies that
\ea
&\big[1_{H_0(\M_{\bs e}^\pl)}, 1_{H_0(\M_{\bs f}^\pl)}\big]
\nonumber\\
&=\begin{cases} 0, & \supp(\bs e)\cap\supp(\bs f)\ne\es, \\
0, & \supp(\bs e)\cap\supp(\bs f)=\es,\;\> \chi_Q(\bs e,\bs f)=\chi_Q(\bs f,\bs e)=0, \\
1_{H_0(\M_{\bs e+\bs f}^\pl)}, & \supp(\bs e)\cap\supp(\bs f)=\es,\  \chi_Q(\bs e,\bs f)=0,   \chi_Q(\bs f,\bs e)=-1,\\
-1_{H_0(\M_{\bs e+\bs f}^\pl)}, & \supp(\bs e)\cap\supp(\bs f)=\es,\  \chi_Q(\bs e,\bs f)=-1, \  \chi_Q(\bs f,\bs e)=0,
\end{cases}
\nonumber\\
&\quad \mod \check H_0\big(\M^\pl\big)^{\bs d}_{\text{conn}} \cap \check H_0\big(\M^\pl\big)_{\text{non-bin}},
\label{wc6eq9}
\ea
where the four cases realize all possibilities. To see this, note that if $\supp(\bs e)\cap\supp(\bs f)\ne\es$ then $\bs e+\bs f$ is non-binary, so $\big[1_{H_0(\M_{\bs e}^\pl)}, 1_{H_0(\M_{\bs f}^\pl)}\big]\in \check H_0\big(\M^\pl\big)_{\text{non-bin}}$, giving the first case. If $\supp(\bs e)\cap\supp(\bs f)=\es$ then as $\supp(\bs e)$, $\supp(\bs f)$ lie in $\supp(\bs d)$, which is a tree, there is at most one edge joining $\supp(\bs e)$, $\supp(\bs f)$, so the possibilities are $\chi_Q(\bs e,\bs f)=\chi_Q(\bs f,\bs e)=0$ (if no edges), or $\chi_Q(\bs e,\bs f)=0$, $\chi_Q(\bs f,\bs e)=-1$ (if one edge $\supp(\bs f)\ra\supp(\bs e)$, by \eq{wc5eq4}), or $\chi_Q(\bs e,\bs f)=-1$, $\chi_Q(\bs f,\bs e)=0$ (if one edge $\supp(\bs e)\ra\supp(\bs f)$).

To prove \eq{wc6eq9}, the first case is immediate. In \eq{wc6eq8}, we have $\big(1_{H_0(\M_{\bs e})}\bt 1_{H_0(\M_{\bs f})}\big)\cap c_i(\Th^\bu)=0$ unless $i=0$, for dimension reasons. As $\chi(\bs e,\bs f)=\chi_Q(\bs e,\bs f)+\chi_Q(\bs f,\bs e)$, the condition $i\ge\chi(\bs e,\bs f)+1$ in \eq{wc6eq8} means there are no nonzero terms in \eq{wc6eq8} in the second case of \eq{wc6eq9}. For the third and fourth cases, there is one nonzero term in \eq{wc6eq8}, with $i=0$, $\chi(\bs e,\bs f)=-1$, and $(-1)^{\chi_Q(\bs e,\bs f)}$ is 1 in the third case and $-1$ in the fourth. Equation \eq{wc6eq9} follows, and as  $\bar\chi(\bs e,\bs f)=\chi_Q(\bs e,\bs f)-\chi_Q(\bs f,\bs e)$ this implies \eq{wc6eq7}. Hence $\Up$ is a Lie algebra morphism, as claimed.

Applying the Lie algebra morphism $\Psi\colon \SFai(\M)\ra L(Q)$ to \eq{wc6eq5} gives an equation in $L(Q)^{\bs d}_{\text{conn}}\subseteq L(Q)$. Composing with the Lie algebra projection $L(Q)^{\bs d}_{\text{conn}}\ra L(Q)_{\text{bin}}^{\bs d}$, and then applying the Lie algebra morphism $\Up\colon L(Q)^{\bs d}_{\text{bin}} \ra \check H_0\big(\M^\pl\big)^{\bs d}_{\text{bin}}$, gives an identity in $\check H_0\big(\M^\pl\big)^{\bs d}_{\text{bin}}$:
\e
\Up \ci \Psi(\bar\ep{}^{\bs d}(\mu))=
\sum_{\substack{\bs d_1,\dots,\bs d_n\in
\N^{Q_0}\sm\{0\}\colon n=\md{\bs d},\\ \bs d_1+\cdots+\bs d_n=\bs d,\, \bs d_i=\de_{v_i},\, v_i\in Q_0 }}\begin{aligned}[t]
\ti U(\bs d_1,&\dots,\bs d_n;\ti\mu,\mu)\cdot\bigl[\bigl[\cdots\bigl[\Up\!\ci\!\Psi(\bar\ep{}^{\bs d_1}(\ti\mu)),\\
&
\Up \ci \Psi(\bar\ep{}^{\bs d_2}(\ti\mu))\bigr],\dots\bigr],\Up\!\ci\!\Psi(\bar\ep{}^{\bs d_n}(\ti\mu))\bigr].
\end{aligned}
\label{wc6eq10}
\e
Now \eq{wc6eq6} is an identity in $\check H_0\big(\M^\pl\big)^{\bs d}_{\text{conn}}$, so we can apply the Lie algebra morphism $\check H_0\big(\M^\pl\big)^{\bs d}_{\text{conn}}\ab\ra\check H_0\big(\M^\pl\big)^{\bs d}_{\text{bin}}$, and compare it with \eq{wc6eq10}. Using $\bs d_i=\de_{v_i}$, Proposition~\ref{wc5prop1}, and the definitions of $\Psi$, $\Up$, we find that
\e
\Up\ci\Psi\big(\bar\ep{}^{\bs d_i}(\ti\mu)\big)=\Up\big({-}\la^{\bs d_i}\big)=1_{H_0(\M_{\bs d_i}^\pl)}.
\label{wc6eq11}
\e

Note that this uses the complicated definition of $\Psi$ in \cite[Section~7.3]{JoSo}, which we have not explained. But in this case, $\bar\ep{}^{\bs d_i}(\ti\mu)$ is the stack function $\M_{\bs d_i}\hookra\M$, where $\M_{\bs d_i}\cong[*/\bG_m]$ as $\bs d_i=\de_{v_i}$, and showing $\Psi\big(\bar\ep{}^{\bs d_i}(\ti\mu)\big)=-\la^{\bs d_i}$ is straightforward. Therefore the right-hand sides of~\eq{wc6eq10} and the projection of~\eq{wc6eq6} agree, so
\e
\Up\ci\Psi\big(\bar\ep{}^{\bs d}(\mu)\big)=[\M_{\bs d}^\ss(\mu)]_\inv
\mod \check H_0\big(\M^\pl\big)^{\bs d}_{\text{conn}} \cap \check H_0\big(\M^\pl\big)_{\text{non-bin}}.
\label{wc6eq12}
\e
But \eq{wc6eq4} and the proof of \eq{wc6eq11} imply that
\e
\Up\ci\Psi\big(\bar\ep{}^{\bs d}(\mu)\big)=\begin{cases} 1_{H_0(\M_{\bs d_i}^\pl)}, & \M^\rst_{\bs d}(\mu)=\M^\ss_{\bs d}(\mu)\cong *, \\
0, & \M^\rst_{\bs d}(\mu)=\M^\ss_{\bs d}(\mu)=\es.	
\end{cases}
\label{wc6eq13}
\e
Combining \eq{wc6eq12}--\eq{wc6eq13} proves Theorem \ref{wc5thm1}(i) for $\mu$, $\bs d$.
\end{proof}

\subsection[\texorpdfstring{Proof of Theorem \ref{wc5thm1}(i) for $d$ binary}{Proof of Theorem \ref{wc5thm1}(i) for d binary}]{Proof of Theorem \ref{wc5thm1}(i) for $\boldsymbol{\bs d}$ binary}\label{wc64}

Consider the following situation:

\begin{dfn}\label{wc6def2}
Let $Q=(Q_0,Q_1,h,t)$, $Q'=(Q_0',Q_1',h',t')$ be quivers with $Q_0'=Q_0$, $Q_1'\subseteq Q_1$, $h'=h\vert_{Q_1'}$, $t'=t\vert_{Q_1'}$, that is, $Q'$ is a subquiver of~$Q$ obtained by deleting the edges $Q_1\sm Q_1'$. Then $\la=(\la_0,\la_1)\colon Q\ra Q'$ given by $\la_0=\id_{Q_0}$, $\la_1=\{(e',e')\colon e'\in Q_1'\}$ is a morphism in the sense of Section~\ref{wc54}, so Definition~\ref{wc5def7} gives a Lie algebra morphism $\Om^\pl\colon \check H_0\big(\M^\pl\big)\ra\check H_0\big(\M^{\prime\pl}\big)$, where $\M^\pl$, $\M^{\prime\pl}$ are the `projective linear' moduli stacks for~$\modCQ$, $\modCQ'$.

Let $\mu'$ be a slope function on $\modCQ'$, and $\mu=\mu'\ci\la_*$ the associated slope function on $\modCQ$. Let $\bs d\in\N^{Q_0}\sm\{0\}$, and $\bs d'=\la_*(\bs d)\in\N^{Q_0'}\sm\{0\}$ (in fact $\bs d=\bs d'$). As the factors $\prod_v\bs d(v)!$, $\prod_{v'}\bs d'(v')!$ in \eq{wc5eq13} are equal, Proposition~\ref{wc6prop2} implies that the classes $\big[\M_{\bs d}^\ss(\mu)\big]_\inv$ defined in Proposition~\ref{wc6prop1} satisfy
\e
\Om^\pl\bigl(\big[\M_{\bs d}^\ss(\mu)\big]_\inv\bigr)=\big[\M_{\bs d'}^{\prime\ss}(\mu')\big]_\inv.
\label{wc6eq14}
\e
\end{dfn}

\begin{prop}\label{wc6prop4}
In Definition~{\rm \ref{wc6def2}}, suppose also that $\M^\rst_{\bs d}(\mu)=\M^\ss_{\bs d}(\mu)$. Then $\M^{\prime\rst}_{\bs d'}(\mu')=\M^{\prime\ss}_{\bs d'}(\mu')$, so $\M^\ss_{\bs d}(\mu)$, $\M^{\prime\ss}_{\bs d'}(\mu')$ are smooth projective $\C$-schemes by Proposition~{\rm \ref{wc5prop2}}, and
\e
\Om^\pl\bigl(\io_*\big(\big[\M_{\bs d}^\ss(\mu)\big]_\fund\big)\bigr)=\io'_*\bigl(\big[\M_{\bs d'}^{\prime\ss}(\mu')\big]_\fund\bigr),
\label{wc6eq15}
\e	
where $\io\colon \M^\ss_{\bs d}(\mu)\hookra\M_{\bs d}^\pl$, $\io'\colon \M^{\prime\ss}_{\bs d'}(\mu')\hookra\M_{\bs d'}^{\prime\pl}$ are the inclusions.
\end{prop}

\begin{proof}
We have $\M_{\bs d}^\pl=[R_{\bs d}/\PGL_{\bs d}]$, $\M_{\bs d'}^{\prime \pl}=[R'_{\bs d'}/\PGL_{\bs d'}]$, where $\PGL_{\bs d}=\PGL_{\bs d'}$, and~$R_{\bs d}=R'_{\bs d'}\op R_{\bs d}''$, where $R'_{\bs d'}$ is the $\C$-vector space of edge maps $\rho_e$ coming from $e\in Q_1'\subset Q_1$, and~$R_{\bs d}''$ the $\C$-vector space of edge maps $\rho_e$ coming from $e\in Q_1\sm Q_1'$. By \eq{wc5eq12}, the vector bundle $G_{\bs d}^\pl\ra\M_{\bs d}^\pl$ in Definitions \ref{wc2def5} and \ref{wc5def7} is that associated to $R_{\bs d}''$ with its $\PGL_{\bs d'}$-action. Also $\si^\pl\colon \M_{\bs d}^\pl\ra \M_{\bs d'}^{\prime \pl}$ is induced by the projection $R_{\bs d}\ra R'_{\bs d'}$ and the identity $\PGL_{\bs d}\ra\PGL_{\bs d'}$.

There is a natural morphism $\io^\pl\colon \M_{\bs d'}^{\prime \pl}\hookra\M_{\bs d}^\pl$ induced by the inclusion $R'_{\bs d'}\hookra R_{\bs d}$ and the identity $\PGL_{\bs d'}\ra\PGL_{\bs d}$, which embeds $\M_{\bs d'}^{\prime \pl}$ as a closed substack of $\M_{\bs d}^\pl$. It has $\si^\pl\ci\io^\pl=\id$. There is a natural transverse section $s$ of the vector bundle $G^\pl\ra \M_{\bs d}^\pl$ mapping an object $((V_v)_{v\in Q_0},(\rho_e)_{e\in Q_1})$ to $(\rho_e)_{e\in Q_1\sm Q_1'}$ in $G^\pl$. The zero locus $s^{-1}(0)$ in $\M_{\bs d}^\pl$ is the substack~$\io^\pl\big(\M_{\bs d'}^{\prime \pl}\big)$.

Now $\M^\rst_{\bs d}(\mu)\cap s^{-1}(0)=\io^\pl\big(\M^{\prime\rst}_{\bs d'}(\mu')\big)$ and $\M^\ss_{\bs d}(\mu)\cap s^{-1}(0)=\io^\pl\big(\M^{\prime\ss}_{\bs d'}(\mu')\big)$. Hence $\M^\rst_{\bs d}(\mu)=\M^\ss_{\bs d}(\mu)$ implies that
$\io^\pl\big(\M^{\prime\rst}_{\bs d'}(\mu')\big)=\io^\pl\big(\M^{\prime\ss}_{\bs d'}(\mu')\big)$, and thus that $\M^{\prime\rst}_{\bs d'}(\mu')=\M^{\prime\ss}_{\bs d'}(\mu')$, as $\io^\pl$ is an embedding.

By Definition \ref{wc2def6}, $\Om^\pl$ on $H_*\big(\M_{\bs d}^\pl\big)$ is the composition
\e
\smash{\xymatrix@C=50pt{ H_*\big(\M_{\bs d}^\pl\big) \ar[r]^{\cap c_\top(G^\pl)} & H_*\big(\M_{\bs d}^\pl\big) \ar[r]^{\si_*^\pl} & H_*\big(\M_{\bs d'}^{\prime\pl}\big). }}
\label{wc6eq16}
\e
This acts on $\io_*\big(\big[\M_{\bs d}^\ss(\mu)\big]_\fund\big)$ by
\begin{equation*}
\xymatrix@C=33pt{ \io_*\big(\big[\M_{\bs d}^\ss(\mu)\big]_\fund\big) \ar@{|->}[r]^(0.41){\raisebox{7pt}{$\scriptstyle\cap c_\top(G^\pl)$}} & {\substack{\io_*([\M_{\bs d}^\ss(\mu)\!\cap\! s^{-1}(0)]_\fund) \\   =\io_*([\io^\pl(\M_{\bs d'}^{\prime\ss}(\mu'))]_\fund) \\   =\io^\pl_*(\io'_*([\M_{\bs d'}^{\prime\ss}(\mu')]_\fund))}} \ar@{|->}[r]^(0.55){\si_*^\pl} & \io'_*\big(\big[\M_{\bs d'}^{\prime\ss}(\mu')\big]_\fund\big). }
\end{equation*}
Here the first step holds as $\cap s^{-1}(0)$ and $\cap c_\top\big(G^\pl\big)$ have the same effect in homology, since~$s$ is transverse, and $\M_{\bs d}^\ss(\mu)\cap s^{-1}(0)=\io^\pl\big(\M_{\bs d'}^{\prime\ss}(\mu')\big)$, and the second step holds as $\si^\pl\ci\io^\pl=\id$. Equation~\eq{wc6eq15} follows.	
\end{proof}

\begin{cor}\label{wc6cor1}
In Proposition~{\rm \ref{wc6prop4}}, suppose also that $\bs d$ is binary. Then if either $\supp(\bs d')$ is a~tree in~$Q'$, or $\supp(\bs d')$ is disconnected, then
\e
\Om^\pl\bigl(\big[\M_{\bs d}^\ss(\mu)\big]_\inv\bigr)=\Om^\pl\bigl(\io_*\big(\big[\M_{\bs d}^\ss(\mu)\big]_\fund\big)\bigr).\label{wc6eq17}
\e
\end{cor}

\begin{proof}
If $\supp(\bs d')$ is a tree then
\begin{equation*}
\Om^\pl\bigl(\big[\M_{\bs d}^\ss(\mu)\big]_\inv\bigr) = \big[\M_{\bs d'}^{\prime\ss}(\mu')\big]_\inv = \io'_*\bigl(\big[\M_{\bs d'}^{\prime\ss}(\mu')\big]_\fund\bigr) = \Om^\pl\bigl(\io_*\big(\big[\M_{\bs d}^\ss(\mu)\big]_\fund\big)\bigr),	
\end{equation*}
using \eq{wc6eq14} in the first step, Proposition \ref{wc6prop3} in the second, and \eq{wc6eq15} in the third, proving~\eq{wc6eq17}.

If $\supp(\bs d')$ is disconnected then $\M^{\prime\rst}_{\bs d'}(\mu')=\es$, as every object in class $\bs d'$ in $\modCQ'$ is the direct sum of nonzero objects from each component of $\supp(\bs d')$, and so cannot be $\mu'$-stable. Thus $\M^{\prime\ss}_{\bs d'}(\mu')=\es$ as $\M^{\prime\rst}_{\bs d'}(\mu')=\M^{\prime\ss}_{\bs d'}(\mu')$ by Proposition \ref{wc6prop4}, so $[\M^{\prime\ss}_{\bs d'}(\mu')]_\fund=0$. Let $\ti\mu'$ be an increasing slope function on $\modCQ'$. Then $\big[\M_{\bs d'}^{\prime\ss}(\mu')\big]_\inv$ is given by~\eq{wc6eq6} with $\ti\mu'$, $\mu'$, $\bs d'$ in place of $\ti\mu$, $\mu$, $\bs d$. For each term on the right-hand side of \eq{wc6eq6} from $\bs d'_1=\de_{v_1'},\dots,\bs d'_n=\de_{v_n'}$, there exists unique $1\le k<n$ such that $v_1,\dots,v_k$ lie in one component of $\supp\bs d'$, and $v_{k+1}$ in a~different component. Then the nested Lie bracket in this term
\begin{equation*}
\bigl[\bigl[\cdots\bigl[1_{H_0(\M_{\bs d_1}^\pl)},1_{H_0(\M_{\bs d_2}^\pl)}\bigr],\dots1_{H_0(\M_{\bs d_k}^\pl)}\bigr],1_{H_0(\M_{\bs d_{k+1}}^\pl)}\bigr]=0,
\end{equation*}
as the outer Lie bracket is of the form $[A,B]$, where $A$, $B$ are supported on different connected components of $\bs d$, and their Lie bracket is zero. Hence $\big[\M_{\bs d'}^{\prime\ss}(\mu')\big]_\inv=\big[\M^{\prime\ss}_{\bs d'}(\mu')\big]_\fund=0$, so~\eq{wc6eq17} follows from~\eq{wc6eq14}--\eq{wc6eq15}.
\end{proof}

\begin{prop}\label{wc6prop5}
Let $Q=(Q_0,Q_1,h,t)$ be a quiver. For each $e\in Q_1$, let $Q'_e$ be $Q$ with edge $e$ deleted, and $\la_e\colon Q\ra Q_e'$, $\Om_e^\pl\colon \check H_0\big(\M_{\bs d}^\pl\big)\ra\check H_0\big(\M_{\bs d,e}^{\prime\pl}\big)$	be the morphisms in Definition~{\rm\ref{wc6def2}}, where $\M_{\bs d,e}^{\prime\pl}$ is the moduli stack of $(\bs V,\bs\rho)$ in $\text{\rm mod-}\C Q'_e$ with $\bdim(\bs V,\bs\rho)=\bs d$. Suppose $\bs d\in\N^{Q_0}\sm\{0\}$ is a binary dimension vector with connected support and $\chi_Q(\bs d,\bs d)\le 0$. Then the following is injective:
\e
 \bigop_{\text{$e$ edge in $\supp(\bs d)$}}\Om_e^\pl\colon \ \check H_0\big(\M_{\bs d}^\pl\big)\longra \bigop_{\text{$e$ edge in $\supp(\bs d)$}}\check H_0\big(\M_{\bs d,e}^{\prime\pl}\big).
\label{wc6eq18}
\e
\end{prop}

\begin{proof}
As $\bs d$ is binary, we have
\begin{equation*}
\GL_{\bs d}= \prod_{v\in Q_0\colon \bs d(v)=1}\bG_m\cong\bG_m^{\md{\bs d}}\qquad\text{and}\qquad \PGL_{\bs d}\cong\bG_m^{\md{\bs d}-1}.
\end{equation*}
Since $H^*([*/\bG_m])\cong R[c]$ for $c$ a formal variable of degree~2, we see that
\begin{equation*}
H^*(\M_{\bs d})=H^*([R_{\bs d}/\GL_{\bs d}])\cong H^*([*/\GL_{\bs d}])\cong R[c_v\colon v\in Q_0,\, \bs d(v)=1].
\end{equation*}
The projection $\Pi^\pl\colon \M_{\bs d}\ra\M_{\bs d}^\pl$ has pullback $\big(\Pi^\pl\big)^*\colon H^*\big(\M_{\bs d}^\pl\big)\ra H^*(\M_{\bs d})$, which is injective. This realizes $H^*\big(\M_{\bs d}^\pl\big)$ as the subalgebra
\begin{equation*}
H^*\big(\M_{\bs d}^\pl\big)\cong\an{c_w-c_v\colon v\ne w,\, v,w\in Q_0,\, \bs d(v)=\bs d(w)=1}\subset R[c_v\colon v\in Q_0,\, \bs d(v)=1]
\end{equation*}
generated by differences $c_w-c_v$. Since $\bs d$ has connected support, any such $c_w-c_v$ is a finite sum of $\pm(c_{w'}-c_{v'})$, for $\overset{v'}{\bu}\,{\buildrel e\over \longra}\,\overset{w'}{\bu}$ an edge in $\supp(d)$. Hence
\e
H^*\big(\M_{\bs d}^\pl\big)\cong\big\langle c_w-c_v\colon \text{$\overset{v}{\bu}\,{\buildrel e\over \longra}\,\overset{w}{\bu}$ an edge in $\supp(d)$}\big\rangle
\subset R[c_v\colon v\in Q_0,\, \bs d(v)=1].
\label{wc6eq19}
\e
Suppose $\eta \!\in\! \check H_0\big(\M_{\bs d}^\pl\big)=H_{2-2\chi_Q(\bs d,\bs d)}\big(\M_{\bs d}^\pl\big)$ lies in the kernel of \eq{wc6eq18}. Let $\ze \!\in\!  H^{2-2\chi_Q(\bs d,\bs d)}\big(\M_{\bs d}^\pl\big)$. By~\eq{wc6eq19}, as $\chi_Q(\bs d,\bs d)\le 0$ we may write
\e
\ze= \sum_{\text{$\overset{v}{\bu}\,{\buildrel e\over \longra}\,\overset{w}{\bu}$ edge in $\supp(\bs d)$}}\ze_e\cup(c_w-c_v),	
\label{wc6eq20}
\e
with $\ze_e\in H^{-2\chi_Q(\bs d,\bs d)}\big(\M_{\bs d}^\pl\big)$. Then
\begin{align*}
\ze\cdot\eta&= \sum_{\text{$\overset{v}{\bu}\,{\buildrel e\over \longra}\,\overset{w}{\bu}$ edge in $\supp(\bs d)$}}\bigl(\ze_e\cup(c_w-c_v)\bigr)\cdot\eta\\
&= \sum_{\text{$\overset{v}{\bu}\,{\buildrel e\over \longra}\,\overset{w}{\bu}$ edge in $\supp(\bs d)$}}\big(\big(\si^\pl_e\big)^*\big)^{-1}(\ze_e)\cdot\Om^\pl_e(\eta)=0,
\end{align*}
using \eq{wc6eq20} in the first step, and the definition of $\Om^\pl_e$ in Definition \ref{wc2def6} with $c_\top\big(G_e^\pl\big)=c_w-c_v$ in the second, and $\eta$ in the kernel of \eq{wc6eq18} in the third. Hence $\ze\cdot\eta=0$ for all $\ze\in H^{2-2\chi_Q(\bs d,\bs d)}\big(\M_{\bs d}^\pl\big)$, so $\eta=0$, and \eq{wc6eq18} is injective.
\end{proof}

\begin{prop}\label{wc6prop6}
The classes $\big[\M_{\bs d}^\ss(\mu)\big]_\inv$ in $\check H_0\big(\M_{\bs d}^\pl\big)$ defined in Proposition~{\rm \ref{wc6prop1}} satisfy Theo\-rem~{\rm \ref{wc5thm1}(i)} when $\bs d$ is binary.
\end{prop}

\begin{proof} Let $\bs d$ be a binary dimension vector, and $\mu$ a slope function on $\modCQ$ such that $\M^\rst_{\bs d}(\mu)=\M^\ss_{\bs d}(\mu)$. We must prove that
\e
\io_*\bigl(\big[\M^\ss_{\bs d}(\mu)\big]_\fund\bigr)=\big[\M_{\bs d}^\ss(\mu)\big]_\inv\qquad\text{in} \quad H_{2-2\chi_Q(\bs d,\bs d)}\big(\M^\pl_{\bs d}\big).
\label{wc6eq21}
\e
If $\chi_Q(\bs d,\bs d)>1$ this is automatic, as both sides lie in $H_{<0}\big(\M^\pl_{\bs d}\big)=0$. If $\chi_Q(\bs d,\bs d)=1$, then $\supp(\bs d)$ is a quiver with $n$ vertices and $n-1$ edges by \eq{wc5eq4}, so either $\supp(\bs d)$ is a~tree, when~\eq{wc6eq21} holds by Proposition \ref{wc6prop3}, or  $\supp(\bs d)$ is disconnected, when both sides of \eq{wc6eq21} are zero by the proof of Corollary~\ref{wc6cor1}.

Suppose by induction on $k=0,1,\dots$ that \eq{wc6eq21} holds when $\chi_Q(\bs d,\bs d)\ge 1-k$. The first step $k=0$ holds from above. For the inductive step, suppose the inductive hypothesis holds for some $k\ge 0$, and that $\bs d$ has $\chi_Q(\bs d,\bs d)=1-(k+1)\le 0$. Using the notation of Proposition~\ref{wc6prop5}, we have
\begin{gather*}
 \bigop_{\text{$e$ edge in $\supp(\bs d)$}}\Om_e^\pl\bigl(\io_*([\M^\ss_{\bs d}(\mu)]_\fund)\bigr)=\bigop_{\text{$e$ edge in $\supp(\bs d)$}}\io_e^*\bigl([\M_{\bs d'_e}^\ss(\mu'_e)]_\fund\bigr) \\
\qquad{} =\bigop_{\text{$e$ edge in $\supp(\bs d)$}}[\M_{\bs d'_e}^\ss(\mu'_e)]_\inv=\bigop_{\text{$e$ edge in $\supp(\bs d)$}}\Om_e^\pl\bigl([\M_{\bs d}^\ss(\mu)]_\inv\bigr),
\end{gather*}
where the first step uses Proposition~\ref{wc6prop4}, the second the inductive hypothesis for~$Q_e'$, noting that $\chi_{Q'_e}(\bs d,\bs d)=1-k$ as $Q_e'$ has one fewer edge in $\supp(\bs d)$, and the third Proposition~\ref{wc6prop2} for $\la_e\colon Q\ra Q_e'$. As~\eq{wc6eq18} is injective by Proposition~\ref{wc6prop5}, this implies~\eq{wc6eq21}, and the proposition follows by induction.
\end{proof}

\subsection[\texorpdfstring{Proof of Theorem \ref{wc5thm1}(i) when $\mu$, $d$ are generic}{Proof of Theorem \ref{wc5thm1}(i) when μ, d are generic}]{Proof of Theorem \ref{wc5thm1}(i) when $\boldsymbol{\mu}$, $\boldsymbol{\bs d}$ are generic}\label{wc65}

Consider the following situation:

\begin{dfn}
\label{wc6def3}
Let $Q=(Q_0,Q_1,h,t)$ be a quiver, and $\bs d \in\N^{Q_0}\sm\{0\}$. Define a quiver $\ti Q=\big(\ti Q_0,\ti Q_1,\ti h,\ti t\big)$ as follows: for each vertex $v \in Q_0$ there are $\bs d(v)$ vertices labelled by pairs $(v,i)$ where $v \in Q_0$ and $i=1,\dots,\bs d(v)$. For each edge $\overset{v}{\bu}\,{\buildrel e\over \longra}\,\overset{w}{\bu}$ in $Q$ there is an edge from $(v,i)$ to $(w,j)$ in $\ti Q$ for all $i=1,\dots,\bs d(v)$ and $j=1,\dots,\bs d(w)$. Explicitly we set
\begin{alignat*}{3}
&\ti Q_0 = \bigl\{(v,i)\colon v\in Q_0,\, i=1,\dots,\bs d(v)\bigr\}, \qquad && \ti h\colon \ (e,i,j) \mapsto  (h(e),i),&
\\
& \ti Q_1 = \bigl\{(e,i,j) \in  Q_1 \t \N^2\colon 1 \le  i \le  \bs d(h(e)),\, 1\!\le  j \le  \bs d(t(e))\bigr\},\qquad && \ti t\colon (e,i,j) \mapsto  (t(e),j).&
\end{alignat*}
This is illustrated in the next diagram, with $\bs d=(2,3)$:
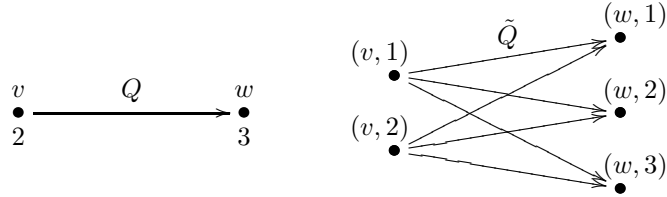
\begin{equation*}
\splinetolerance{.8pt}
\begin{xy}
0;<1mm,0mm>:
,(-15,3)*{Q}
,(35,10.5)*{\ti Q}
,(-30,0)*{\bu}
,(-30,3)*{v}
,(-30,-3)*{2}
,(0,0)*{\bu}
,(0,3)*{w}
,(0,-3)*{3}
,(50,0)*{\bu}
,(52,3)*{(w,2)}
,(20,5)*{\bu}
,(52,13)*{(w,1)}
,(50,10)*{\bu}
,(20,-5)*{\bu}
,(18,-2)*{(v,2)}
,(18,8)*{(v,1)}
,(50,0)*{\bu}
,(52,-7)*{(w,3)}
,(50,-10)*{\bu}
,(-30,0)*+{\bu} ; (0,0)*+{\bu} **@{-} ?>*\dir{>}
,(20,5)*+{\bu} ; (50,10)*+{\bu} **@{-} ?>*\dir{>}
,(20,5)*+{\bu} ; (50,0)*+{\bu} **@{-} ?>*\dir{>}
,(20,5)*+{\bu} ; (50,-10)*+{\bu} **@{-} ?>*\dir{>}
,(20,-5)*+{\bu} ; (50,10)*+{\bu} **@{-} ?>*\dir{>}
,(20,-5)*+{\bu} ; (50,0)*+{\bu} **@{-} ?>*\dir{>}
,(20,-5)*+{\bu} ; (50,-10)*+{\bu} **@{-} ?>*\dir{>}
\end{xy}
\end{equation*}

Define $\bs{\ti d}\in\N^{\ti Q_0}\sm\{0\}$ by $\bs{\ti d}(v,i)=1$ for all $(v,i)\in\ti Q_0$.

Define a morphism $\la\colon \ti Q\ra Q$ in the sense of Section~\ref{wc54} by $\la_0\colon (v,i)\mapsto v$ and $\la_1=\bigl\{((e,i,j),e)\colon (e,i,j)\in\ti Q_1\bigr\}$. Then Definition \ref{wc5def6}(i)--(iii) hold, so we have a $\C$-linear exact functor $\Si_\la\colon \modCtQ\ra\modCQ$ inducing morphisms of moduli stacks $\si_\la\colon \tiM\ra\M$ and $\si_\la^\pl\colon \tiM^\pl\ra\M^\pl$ for $\modCtQ$, $\modCQ$. We have $\la_*\big(\bs{\ti d}\big)=\bs d$, so $\si_\la$, $\si_\la^\pl$ map $\tiM_{\bs{\ti d}}\ra\M_{\bs d}$ and~$\tiM_{\bs{\ti d}}^\pl\ra\M_{\bs d}^\pl$.

Explicitly, as in Definition \ref{wc5def2} we have
\begin{gather*}
\M_{\bs d}=[R_{\bs d}/\GL_{\bs d}],
\qquad\text{where}\quad R_{\bs d}= \prod_{e \in Q_1} \Hom\big(\C^{\bs d(t(e))},\C^{\bs d(h(e))}\big),\\
\GL_{\bs d}=  \prod_{v \in Q_0} \GL(\bs d(v),\C), \qquad
\M_{\bs d}^\pl=[R_{\bs d}/\PGL_{\bs d}], \qquad \PGL_{\bs d}=\GL_{\bs d}/\bG_m, \\
\tiM_{\smash{\bs{\ti d}}}=\big[\ti R_{\smash{\bs{\ti d}}}/\GL_{\smash{\bs{\ti d}}}\big],
\qquad\text{where}\quad \ti R_{\smash{\bs{\ti d}}}= \prod_{e\in Q_1,\, i=1,\dots, \bs d(t(e)),\, j=1,\dots, \bs d(h(e))}\Hom(\C,\C),\\
\GL_{\smash{\bs{\ti d}}}=  \prod_{v \in Q_0,\; i=1,\dots, \bs d(v)} \bG_m, \qquad
\tiM_{\smash{\bs{\ti d}}}^\pl=\big[\ti R_{\smash{\bs{\ti d}}}/\PGL_{\smash{\bs{\ti d}}}\big], \qquad \PGL_{\smash{\bs{\ti d}}}=\GL_{\smash{\bs{\ti d}}}/\bG_m.
\end{gather*}
The morphisms $\si_\la\colon \tiM_{\bs{\ti d}}\ra\M_{\bs d}$ and $\si_\la^\pl\colon \tiM_{\bs{\ti d}}^\pl\ra\M_{\bs d}^\pl$ are induced by the obvious map $\la_*\colon \ti R_{\smash{\bs{\ti d}}}\ra R_{\bs d}$, which is an {\it isomorphism} in this case, and by morphisms $\la_*\colon \GL_{\smash{\bs{\ti d}}}\ra \GL_{\bs d}$ and $\la_*^\pl\colon \PGL_{\smash{\bs{\ti d}}}\ra \PGL_{\bs d}$, which are {\it inclusions of maximal tori} in this case. That is, we have $\M_{\bs d}=[V/G]$ and $\tiM_{\smash{\bs{\ti d}}}=[V/H]$ where $H\subseteq G$ is the maximal torus.

Next let $\mu$ be a slope function on $\modCQ$ such that $\mu$, $\bs d$ are generic in the sense of Definition~\ref{wc6def1}, and let $\ti\mu=\mu\ci\la_*$ be the associated slope function on $\modCtQ$. Then $\ti\mu$, $\bs{\ti d}$ are also generic for $\modCtQ$, so $\M_{\bs d}^\rst(\mu)=\M_{\bs d}^\ss(\mu)$ and $\tiM_{\bs{\ti d}}^\rst(\ti\mu)=\tiM_{\bs{\ti d}}^\ss(\ti\mu)$. Proposition~\ref{wc6prop2} gives
\e
\Om^\pl\bigl(\big[\tiM_{\bs{\ti d}}^\ss(\ti\mu)\big]_\inv\bigr)=\prod_{v\in Q_0}\bs d(v)!\cdot \big[\M_{\bs d}^\ss(\mu)\big]_\inv,
\label{wc6eq22}
\e
and as $\bs{\ti d}$ is binary with $\tiM_{\bs{\ti d}}^\rst(\ti\mu)=\tiM_{\bs{\ti d}}^\ss(\ti\mu)$, Proposition \ref{wc6prop6} gives
\e
\big[\tiM_{\bs{\ti d}}^\ss(\ti\mu)\big]_\inv=\ti\io_*\bigl(\big[\tiM_{\bs{\ti d}}^\ss(\ti\mu)\big]_\fund\bigr).
\label{wc6eq23}
\e
The reason we suppose $\mu$, $\bs d$ generic in this section is that otherwise $\M_{\bs d}^\rst(\mu)=\M_{\bs d}^\ss(\mu)$ does not imply that $\tiM_{\bs{\ti d}}^\rst(\ti\mu)=\tiM_{\bs{\ti d}}^\ss(\ti\mu)$, and if $\tiM_{\bs{\ti d}}^\rst(\ti\mu)\ne\tiM_{\bs{\ti d}}^\ss(\ti\mu)$ then $\big[\tiM_{\bs{\ti d}}^\ss(\ti\mu)\big]_\fund$ in \eq{wc6eq23} is not~defined.
\end{dfn}

We will use the following result of Martin \cite[Theorem~B]{Mart}. It can also be written in algebraic geometry in terms of smooth GIT quotients.

\begin{thm}\label{wc6thm1}
Let $(X,\om)$ be a symplectic $2n$-manifold, with a Hamiltonian action of a compact Lie group $G$ with moment map $\mu_G\colon X\ra\g^*$, where $\g$ is the Lie algebra of $G$. Let $T\subset G$ be a maximal torus with Lie algebra $\mathfrak{t}\subseteq\g$, so $\mu_G$ induces a moment map $\mu_T\colon X\ra\mathfrak{t^*}$. Suppose $\mu_G^{-1}(0)$ and $\mu_T^{-1}(0)$ are compact, with free $G$- and $T$-actions, so the quotients $X/\!/ G=\mu_G^{-1}(0)/G$ and $X/\!/ T=\mu_T^{-1}(0)/T$ are compact symplectic manifolds. Also $Y=\mu_G^{-1}(0)/T$ is a compact manifold, with projections $\pi\colon Y\ra X/\!/G$, $i\colon Y\ra X/\!/T$.

Write $\g=\mathfrak{t}\op\mathfrak{m}$ for the $T$-invariant splitting, and $E\ra X/\!/T$ for the complex vector bundle $E=\big(\mu_T^{-1}(0)\t\mathfrak{m}\ot_\R\C\big)/T$ associated to the complex representation of $T$ on $\mathfrak{m}\ot_\R\C$. Then for all classes $\eta\in H^{2n-2\dim G}(X/\!/G)$, $\ze\in H^{2n-2\dim G}(X/\!/T)$ with $\pi^*(\eta)=i^*(\ze)$ in $H^{2n-2\dim G}(Y)$, we have
\begin{equation*}
\int_{X/\!/G}\eta=\frac{1}{\md{W}}\int_{X/\!/T}\ze\cup c_\top(E),
\end{equation*}
where $W$ is the Weyl group of $G$.
\end{thm}

\begin{prop}\label{wc6prop7}
The classes $\big[\M_{\bs d}^\ss(\mu)\big]_\inv$ in $\check H_0\big(\M_{\bs d}^\pl\big)$ defined in Proposition~{\rm \ref{wc6prop1}} satisfy Theo\-rem~{\rm \ref{wc5thm1}(i)} if $\mu,\bs d$ are generic.
\end{prop}

\begin{proof} Use the notation of Definition \ref{wc6def3}. We apply Theorem~\ref{wc6thm1} with $X=R_{\bs d}$, and $G=\bigl(\prod_{v\in Q_0}\U(\bs d(v))\bigr)/\U(1)$, which is the maximal compact subgroup of $\PGL_{\bs d}$, and acts on $R_{\bs d}$ via the $\PGL_{\bs d}$-action preserving a Euclidean K\"ahler form $\om$ on $R_{\bs d}$. By the relationship between GIT quotients and symplectic quotients in Kirwan~\cite{Kirw}, there is a moment map $\mu_G\colon X\ra\g^*$ such that $\mu_G^{-1}(0)/G$ is the GIT quotient $R_d/\!/\PGL_{\bs d}=\M_{\bs d}^\ss(\mu)$, so as $\M_{\bs d}^\ss(\mu)$ is a smooth projective $\C$-scheme we see that $\mu_G^{-1}(0)$ is compact with a free $G$-action.

We take the maximal torus $T\subset G$ to be $T=\bigl(\prod_{v\in Q_0}\U(1)^{\bs d(v)}\bigr)/\U(1)\cong \U(1)^{\md{\bs d}-1}$. This is the maximal compact subgroup of $\PGL_{\smash{\bs{\ti d}}}\cong\bG_m^{\md{\bs d}-1}$ under the inclusion $\la_*^\pl\colon \PGL_{\smash{\bs{\ti d}}}\ra \PGL_{\bs d}$ as an algebraic maximal torus. Then the isomorphism $\la_*\colon \ti R_{\smash{\bs{\ti d}}}\ra R_{\bs d}$ and the fact that $\ti\mu=\mu\ci\la_*$ implies that $\mu_T^{-1}(0)/T$ is the GIT quotient $\ti R_{\smash{\bs{\ti d}}}/\!/\PGL_{\smash{\bs{\ti d}}}=\tiM_{\bs{\ti d}}^\ss(\ti\mu)$, so as $\tiM_{\bs{\ti d}}^\ss(\ti\mu)$ is a smooth projective $\C$-scheme we see that $\mu_T^{-1}(0)$ is compact with a free $T$-action.

Let $\th\in H^{2n-2\dim G}\big(\M_{\bs d}^\pl\big)$, and set $\eta=\io^*(\th)$ in $H^{2n-2\dim G}\big(\M_{\bs d}^\ss(\mu)\big)$ and $\ze=\ti\io^*\ci\big(\si_\la^\pl\big)^*(\th)$ in $H^{2n-2\dim G}\big(\tiM_{\bs{\ti d}}^\ss(\ti\mu)\big)$. Since $\eta,\ze$ are pullbacks of the same class $\th\in H^*\big(\big[X/G^\C\big]\big)$, we see that $\pi^*(\eta)=i^*(\ze)$ in $H^{2n-2\dim G}(Y)$. Hence as the Weyl group of $G$ is $\prod_{v\in Q_0}S_{\bs d(v)}$, Theorem~\ref{wc6thm1} gives
\begin{equation*}
\int_{\M_{\bs d}^\ss(\mu)}\io^*(\th)=\frac{1}{\prod_{v\in Q_0}\bs d(v)!}\int_{\tiM_{\bs{\ti d}}^\ss(\ti\mu)}\ti\io^*\ci\big(\si_\la^\pl\big)^*(\th)\cup c_\top(E).
\end{equation*}
We may rewrite this as
\begin{equation*}
\th\cdot\bigg( \prod_{v\in Q_0}\bs d(v)!\cdot \io_*\big(\big[\M_{\bs d}^\ss(\mu)\big]_\fund\big)\bigg)=\th\cdot\bigl(
\big(\si_\la^\pl\big)_*\bigl(\ti\io_*\big(\big[\tiM_{\bs{\ti d}}^\ss(\ti\mu)\big]_\fund\big)\cap c_\top(E)\bigr)\bigr).
\end{equation*}
As this holds for all $\th\in H^{2n-2\dim G}\big(\M_{\bs d}^\pl\big)$, where $2n-2\dim G=2-\chi_Q(\bs d,\bs d)$ is the dimension of $\big[\M_{\bs d}^\ss(\mu)\big]_\fund$, we see that in $H_{2-\chi_Q(\bs d,\bs d)}\big(\M_{\bs d}^\pl\big)$ we have
\e
 \prod_{v\in Q_0}\bs d(v)!\cdot \io_*\big(\big[\M_{\bs d}^\ss(\mu)\big]_\fund\big)=\big(\si_\la^\pl\big)_*\ci\ti\io_*\bigl(\big[\tiM_{\bs{\ti d}}^\ss(\ti\mu)\big]_\fund\cap c_\top(E)\bigr).\label{wc6eq24}
\e

Now one can show from the definitions that the vector bundle $E$ over $X/\!/T=\tiM_{\bs{\ti d}}^\ss(\ti\mu)$ in Theorem \ref{wc6thm1} is isomorphic to $\ti\io^*\big(G^\pl\big)$, where $G^\pl\ra\M_{\bs d}^\pl$ is defined as in Definition~\ref{wc2def5}(c),~(v) from the vector bundle $F\ra\M_{\bs d}\t\M_{\bs d}$ in \eq{wc5eq11}. Thus as in \eq{wc6eq16}, the right-hand side of \eq{wc6eq24} is $\Om^\pl\big(\big[\tiM_{\bs{\ti d}}^\ss(\ti\mu)\big]_\fund\big)$, giving
\e
 \prod_{v\in Q_0}\bs d(v)!\cdot \io_*\big(\big[\M_{\bs d}^\ss(\mu)\big]_\fund\big)=\Om^\pl\big(\big[\tiM_{\bs{\ti d}}^\ss(\ti\mu)\big]_\fund\big).
\label{wc6eq25}
\e
Comparing \eq{wc6eq22}, \eq{wc6eq23} and \eq{wc6eq25} gives $\big[\M_{\bs d}^\ss(\mu)\big]_\inv=\io_*\big(\big[\M_{\bs d}^\ss(\mu)\big]_\fund\big)$, as we want.
\end{proof}

\subsection{Proof of Theorem \ref{wc5thm1}(i) in the general case}\label{wc66}

The next proposition completes the proof of Theorem~\ref{wc5thm1}.

\begin{prop}\label{wc6prop8}
The classes $\big[\M_{\bs d}^\ss(\mu)\big]_\inv$ in $\check H_0\big(\M_{\bs d}^\pl\big)$ defined in Proposition~{\rm \ref{wc6prop1}} satisfy Theo\-rem~{\rm \ref{wc5thm1}(i)} for all~$\bs d$.
\end{prop}

\begin{proof} The proof is by induction on $k=0,1\dots$, with inductive hypothesis that the $\big[\M_{\bs d}^\ss(\mu)\big]_\inv$ in Proposition~\ref{wc6prop1} satisfy Theorem \ref{wc5thm1}(i) for all $\bs d\in\N^{Q_0}\sm\{0\}$ with $\md{\bs d}=\sum_{v\in Q_0}\bs d(v)\le k$. The first step $k=0$ is vacuous. For the inductive step, suppose the inductive hypothesis holds for some $k\ge 0$, and let $\bs d\in\N^{Q_0}\sm\{0\}$ with $\md{\bs d}=k+1$.

Consider the pair invariant set-up of Definition \ref{wc5def8}. Suppose $\bs d_1,\dots,\bs d_n$ lie in $\N^{Q_0}\sm\{0\}$ for $n\ge 2$ with $\bs d_1+\cdots+\bs d_n=\bs d$ and $\mu(\bs d_i)=\mu(\bs d)$. If $\M_{\bs d_i}^\ss(\mu)\ne\es$ for all $i=1,\dots,n$ then choosing a $\C$-point $[E_i]$ in $\M_{\bs d_i}^\ss(\mu)$, we see that $E_1\op\cdots\op E_n$ is strictly $\mu$-semistable in class $\bs d$, contradicting $\M^\rst_{\bs d}(\mu)=\M^\ss_{\bs d}(\mu)$. Hence $\M_{\bs d_i}^\ss(\mu)=\es$ for some $i=1,\dots,n$, so $[\M^\ss_{\bs d_i}(\mu)]_\inv=0$ by the inductive hypothesis, since $\md{\bs d_i}\le k$ as $\bs d_1+\cdots+\bs d_n=\bs d$ with $\md{\bs d}=k+1$, $n\ge 2$ and $\md{\bs d_j}>0$. Hence in \eq{wc5eq15} all terms with $n\ge 2$ are zero, so \eq{wc5eq15} reduces to
\e
\ti\io_*\bigl(\big[\tiM_{(\bs d,1)}^\ss(\ti\mu^{\bs d}_+)\big]_\fund\bigr)=\bigl[i_*^\pl\bigl(\big[\M_{\bs d}^\ss(\mu)\big]_\inv\bigr),1_{H_0(\tiM_{(0,1)}^\pl)}\bigr],
\label{wc6eq26}
\e
that is, the `lower order terms' in \eq{wc5eq16} vanish. Here we have used the fact that as $\ti\mu^{\bs d}_+$, $(\bs d,1)$ are generic for $\modCtQ$, Proposition \ref{wc6prop7} implies that{\samepage
\begin{equation*}
\big[\tiM_{(\bs d,1)}^\ss(\ti\mu^{\bs d}_+)\big]_\inv=\ti\io_*\bigl(\big[\tiM_{(\bs d,1)}^\ss(\ti\mu^{\bs d}_+)\big]_\fund\bigr),
\end{equation*}
which was used to rewrite the left-hand side of~\eq{wc4eq4} to get~\eq{wc5eq15}}.

There is a natural projection $\pi\colon \tiM_{(\bs d,1)}^\ss(\ti\mu^{\bs d}_+)\ra\M_{\bs d}^\ss(\mu)$ which acts by
\begin{equation*}
((V_v)_{v\in\ti Q_0},(\rho_e)_{e\in\ti Q_1})\longmapsto((V_v)_{v\in Q_0},(\rho_e)_{e\in Q_1})
\end{equation*}
on $\C$-points, noting that $Q_0\subset\ti Q_0$ and $\ti Q_1\subset Q_1$. That is, $\pi$ forgets the vector space $V_\iy\cong\C$ and the edge maps $\rho_{(v,i)}\colon V_\iy\ra V_v$ for edges $\overset{\iy}{\bu}\,{\buildrel (v,i)\over \longra}\,\overset{v}{\bu}$ in $\ti Q$, where $v\in Q_0$ and $i=1,\dots,n_v$. In terms of the exact sequence \eq{wc4eq3} in Definition~\ref{wc4def1}, $\pi$ maps $[B]\mapsto[A]$ on $\C$-points. The fibre of~$\pi$ over a~$\C$-point $[(V_v)_{v\in Q_0},(\rho_e)_{e\in Q_1}]$ is the projective space $\bP\big(\bigop_{v\in Q_0}V_v^{\op^{n_v}}\big)$ parametrizing the forgotten edge maps $\rho_{(v,i)}$ up to scale, where the rescalings come from changing the isomorphism $V_\iy\cong\C$, and the condition of $\ti\mu^{\bs d}_+$-semistability is that the image under $\pi$ should be $\mu$-semistable, and the $\rho_{(v,i)}$ should not all be zero. Thus we may identify $\tiM_{(\bs d,1)}^\ss\big(\ti\mu^{\bs d}_+\big)$ with the projective space~bundle
\e
\tiM_{(\bs d,1)}^\ss\big(\ti\mu^{\bs d}_+\big)\cong\bP\bigg( \bigop_{v\in Q_0}\cV_{v,\bs d}^{\op^{n_v}}\bigg)\ra\M_{\bs d}^\ss(\mu),
\label{wc6eq27}
\e
where the vector bundles $\cV_{v,\bs d}\ra\M_{\bs d}$ are as in Definition~\ref{wc5def2}. In fact $\cV_{v,\bs d}$ does not descend through $\M_{\bs d}\ra\M_{\bs d}^\pl$ to $\M^\pl_{\bs d}\supseteq\M_{\bs d}^\ss(\mu)$, but the projective bundle $\bP\bigl( \bigop_{v\in Q_0}\cV_{v,\bs d}^{\op^{n_v}}\bigr)\ra\M_{\bs d}$ does descend to $\M_{\bs d}^\pl$, which is what we mean in~\eq{wc6eq27}.

Now the second author \cite{Joyc12} gives an alternative, geometric definition of the Lie bracket $[\,,\,]$ on $\check H_*\big(\M^\pl\big)$ in Section~\ref{wc24} in terms of the `projective Euler class' $\mathop{\rm PE}(\Th^\bu)$ of the perfect complex $\Th^\bu\ra\M\t\M$ in Assumption~\ref{wc2ass1}(g). For the Lie bracket in \eq{wc6eq26}, the complex $\Th^\bu$ for $\modCtQ$ is given in \eq{wc5eq5}--\eq{wc5eq6}, and its restriction to $\tiM_{(\bs d,0)}\t\tiM_{(0,1)}$ reduces to
\e
\Th^\bu\vert_{\tiM_{(\bs d,0)}\t\tiM_{(0,1)}}\cong
 \bigg(\bigop_{v\in Q_0}\cV_{v,\bs d}^{\op^{n_v}}\bt\cV_{\iy,1}^*\bigg)[-1],
\label{wc6eq28}
\e
where $\cV_{\iy,1}^*\ra\tiM_{(0,1)}$ is a line bundle. When $\Th^\bu=F[-1]$ is a vector bundle $F$ in degree $1$, the definition of $\mathop{\rm PE}(\Th^\bu)$ in \cite{Joyc12} involves the projective bundle $\bP(F)$. Using this, one can show from \eq{wc6eq27}--\eq{wc6eq28} and \cite{Joyc12} that
\e
\ti\io_*\bigl(\big[\tiM_{(\bs d,1)}^\ss(\ti\mu^{\bs d}_+)\big]_\fund\bigr)=\bigl[i_*^\pl\bigl(\io_*([\M_{\bs d}^\ss(\mu)]_\fund)\bigr),1_{H_0(\tiM_{(0,1)}^\pl)}\bigr].
\label{wc6eq29}
\e
Comparing \eq{wc6eq26} and \eq{wc6eq29}, and using the fact that $\big[{-},1_{H_0(\tiM_{(0,1)}^\pl)}\big]$ is injective if $n_v>0$ for all $v\in Q_0$ (see the argument after \eq{wc5eq14} above), and $i_*^\pl$ is an isomorphism, we see that $\big[\M_{\bs d}^\ss(\mu)\big]_\inv=\io_*\big(\big[\M_{\bs d}^\ss(\mu)\big]_\fund\big)$. This proves the inductive step, and the proposition follows by induction.
\end{proof}

\subsection*{Acknowledgements}
This research was supported by the Simons Collaboration on Special Holonomy in Geometry, Analysis and Physics. The third author was partially supported by JSPS Grant-in-Aid for
Scientific Research numbers JP16K05125 and JP21K03246. The authors would like to thank Arkadij Bojko, Yalong Cao, Frances Kirwan, and Markus Upmeier for helpful conversations, and the anonymous referees for careful proofreading.


\LastPageEnding

\end{document}